\documentclass[12pt,reqno]{amsart}
\usepackage[margin=1in]{geometry}

\usepackage{kopplagarias24}

\begin{document}

\title[Ray classes for orders]{Ray class groups and ray class fields for orders of number fields}
\author{Gene S. Kopp}
\address{Department of Mathematics, Louisiana State University, Baton Rouge, LA, USA}
\email{kopp@math.lsu.edu}
\author{Jeffrey C. Lagarias}
\address{Department of Mathematics, University of Michigan, Ann Arbor, MI, USA}
\email{lagarias@umich.edu}



\date{November 30, 2024}

\begin{abstract} 
This paper contributes to the theory of orders of number fields.
This paper defines a  notion of \textit{ray class group} associated to an arbitrary order in a number field together with an arbitrary ray class modulus for that order (including Archimedean data), constructed using invertible fractional ideals of the order. 
It shows existence of  \textit{ray class fields} corresponding to the class groups. These ray class groups (resp., ray class fields) 
specialize to classical ray class groups (resp., fields) of a number field in the case of the maximal order, and they specialize to 
ring class groups (resp., fields) of orders in the case of trivial modulus. The paper gives  exact sequences for 
simultaneous change of order and change of modulus. As a consequence, we identify the ray class field of an order 
with a given modulus as a specific subfield of a ray class field of the maximal order with a larger modulus. 
We also uniquely describe each ray class field of an order in terms of the splitting behavior of primes. 
\end{abstract}

\maketitle

\tableofcontents

\section{Introduction}

This paper contributes to the theory of orders of number fields. An \textit{order} $\OO$ of an algebraic number field $K$ is a subring of $K$ containing $1$ and having finite rank equal to $[K : \Q]$ as a $\Z$-module. Dedekind showed that the ring $\OO_K$ of all algebraic integers in $K$ is the maximal order, and all other orders $\OO$ are of finite index in $\OO_K$. 

The object of this paper is to extend to general orders $\OO$ the notions of ray class groups and ray class fields for the maximal order $\OO_K$. A ray class group or field of an order is specified by a \textit{level datum} $\ddD = (\OO; \mm, \rS)$ for the field $K$, in which $\OO$ is an order of $K$, $\mm$ is an integral ideal of $\OO$, and $\rS$ is a set of real places of $K$. The pair $(\mm,\rS)$ is called a \textit{modulus}.

Ray class groups are groups of ray classes, which are sets of fractional ideals satisfying congruence conditions modulo $\mm$ and sign conditions at the places in $\rS$. The set of ray classes modulo $(\mm, \rS)$ of a non-maximal order under ideal multiplication forms a monoid (semigroup with identity) rather than a group, because it  includes non-invertible ideal classes.
The set of invertible ray classes forms a group, the \textit{ray class group} of the order.

A ray class field of an order is a certain abelian Galois extension of the number field $K$ associated to a level datum $\ddD = (\OO; \mm,\rS)$. The ray class fields of an order $\OO$ form a ``distinguished'' class of abelian extensions of $K$, cofinal in the partially ordered set of abelian extensions of $K$, whose prime decomposition properties over $K$ are described by the corresponding ray class group and by algebraic properties of the specific order $\OO$ of $K$. 

\subsection{Background}

The first complete version of class field theory for number fields was developed by Takagi \cite{Takagi:1920} in 1920, building on work of Weber, Hilbert, Furtw\"{a}ngler, and Fueter; see \cite{Hasse:67, Katsuya:95}. In Takagi's treatment, the ray class fields of a number field $K$ are associated to ray class groups, which are defined as groups of fractional ideals modulo subgroups of principal ideals satisfying congruence and sign conditions.

The ray class fields of a number field $K$ comprise an infinite set of finite abelian extensions of $K$ that are cofinal in the set of abelian extensions:~That is, every finite abelian extension is contained in some ray class field. The ray class fields $H_{\mm,\rS}$ are attached to moduli $(\mm,\rS)$, where $\mm$ is an ideal of the ring of integers $\OO_K$, and $\rS$ is a subset of the real embeddings of $K$. The extension $H_{\mm,\rS}/K$ is ramified at a subset of the primes diving $\mm$ and the infinite places in $\rS$. The ray class field attached to the modulus $(\OO_K, \emptyset)$ is the \textit{Hilbert class field}.

Associated to each non-maximal order $\OO$ of $K$, there is a separate classical notion of a \textit{ring class field} $H^{\OO}$, associated to a \textit{ring class group} $\Cl(\OO)$. When $K/\Q$ is an imaginary quadratic field, the set of ring class fields arise naturally as the fields $K(j(\tau))$ generated by values of the Klein $j$-invariant at points $\tau \in K$; see Schertz \cite[Ch.~6]{Schertz10}. The set of all ring class fields of a field $K$, obtained by varying the order $\OO$, are generally not cofinal in the set of abelian extensions of $K$ and do not generate the maximal abelian extension $K^{\rm ab}$. The extension $H^\OO/K$ is ramified at a subset of the primes dividing the conductor ideal $\ff(\OO) = \colonideal{\OO}{\OO_K}$.

\subsection{Ray class groups for orders}\label{sec:11}

The first object of the paper is the definition of ray class groups for orders, made in terms of fractional ideals of the order, and the derivation of formulas for computing invariants of such ray class groups, allowing variation of both the order $\OO$ and the ray class modulus $(\mm, \rS)$. Ray class groups for orders specialize to Takagi ray class groups (when $\OO=\OO_K$) and to ring class groups (when $(\mm,\rS)=(\OO,\emptyset)$).

In order to define ray class groups, this paper first gives a detailed review of the properties of integral and fractional ideals of  orders, emphasizing their differences from maximal orders, with examples. 
For integral ideals of non-maximal orders: 
\begin{enumerate}
\item
Ideals do not always factor (uniquely or otherwise) into products of prime ideals.
\item
If an integral ideal $\aa$ divides $\bb$, then $\bb \subseteq \aa$, but the converse need not hold. ``Greatest common divisor'' and ``least common multiple'' of ideals are not well defined. However, the notion of {\em coprimality} of two integral ideals is well defined.
\item
Unique factorization of ideals is restored for the set of all integral ideals coprime to the \textit{conductor ideal} $\ff(\OO)$ of the order, which is the set-theoretically largest ideal of $\OO$ that is also an ideal of $\OO_K$.
\end{enumerate}

For fractional ideals of non-maximal orders: 
\begin{enumerate}
\item 
There exist non-invertible nonzero fractional ideals, so that nonzero fractional ideals under the ideal product operation form a monoid rather than a group.
\item
The monoid of nonzero fractional ideals $\rJ(\OO)$ for ideal product may contain non-integral fractional ideals $\aa$ having a power $\aa^k$ that is an integral ideal. The ideal $\aa^k$ may be $\OO$, giving non-trivial torsion elements in the group of invertible fractional ideals $\rJ^{\ast}(\OO)$. (For the maximal order, $\rJ(\OO)=\rJ^{\ast}(\OO)$ is a free abelian group.)
\item
There is a notion of {\em coprimality} of a fractional ideal and an integral ideal. When a ray class modulus $\mm$ contains the conductor ideal $\ff(\OO)$, then the set $\rJ_{\mm}(\OO)$ of fractional ideals coprime to $\mm$ is a free abelian group.
\end{enumerate}
The paper also treats \textit{extension} and \textit{contraction} of integral and fractional ideals between one order $\OO$ and a larger order $\OO'$ of a fixed number field $K$. This treatment is intended to be more broadly useful beyond its application to ray class groups; see \Cref{subsec:14} for details.

The \textit{ray class group} $\Cl_{\mm, \rS}(\OO)$ is a quotient group of the group $\rJ_{\mm}^{\ast}(\OO)$ of invertible fractional ideals coprime to the modulus $\mm$ by a suitable subgroup $\rP_{\mm,\rS}(\OO)$ of principal ideals.
The modulus $\mm$ may be any nonzero integral ideal---it is permitted to be non-invertible. For this reason, the paper necessarily studies the complexities and pitfalls of non-invertible ideals. More generally, there is a \textit{ray class monoid} $\Clm_{\mm, \rS}(\OO)$ under ideal product built from the monoid of fractional ideals $\rJ_{\mm}(\OO)$ coprime to $\mm$ by modding out by the action of the subgroup $\rP_{\mm,\rS}(\OO)$. Its structure is not considered in this paper and will be treated separately in \cite{KLmonoid}.

The main result of the paper for ray class groups $\Cl_{\mm, \rS}(\OO)$ is an exact sequence relating ray class groups of orders, in which  both the order and the ideal modulus conditions can be simultaneously varied, given as \Cref{thm:exseq}. This exact sequence is obtained using extension and contraction maps relating integral ideals in two orders $\OO \subseteq \OO'$. 
The exact sequence also yields a formula for the cardinality of a ray class group of an order (\Cref{thm:neukirch2}). 

\subsection{Ray class fields for orders and class field theory for orders}\label{sec:12}

The second object of the paper is to establish the existence of ray class fields of orders attached to ray class groups of orders.  These are a distinguished set of abelian extensions of $K$, whose arithmetic of splitting of primes is described by the given ray class group of the order.  We formulate three results which together comprise a ray class field theory for orders.

The first result states that the ray class field of an order $\OO$ with modulus $(\mm, \rS)$ is uniquely specified by its splitting of primes associated to the principal ray class in the ray class group with the given level data. This definition is in the spirit of Weber's original definition of a class field in terms of a law of decomposition of prime ideals, motivated by special values of modular functions. (See \cite[p.\ 164]{Weber1908}, \cite[p.\ 266] {Hasse:67}, and \cite{Weber1897-1898}.) 

\begin{thm}\label{thm:main1} 
Let $K$ be a number field, $\OO$ an order of $K$, $\mm$ an ideal of $\OO$, and $\rS$ a (possibly empty) subset of the set of real embeddings of $K$. Then for the level datum $(\OO; \mm, \rS)$, there exists a unique abelian Galois extension $H_{\mm,\rS}^{\OO}/K$ with the property that a prime ideal $\pp$ of $\OO_K$ that is coprime to the quotient ideal $\colonideal{\mm}{\OO_K}$ (as defined in \eqref{eqn:quotient-ideal}) splits completely in $H_{\mm,\rS}^{\OO}/K$ if and only if $\pp \cap \OO = \pi\OO$, a principal prime $\OO$-ideal having $\pi \in \OO$ with $\pi \equiv 1 \Mod{\mm}$ and $\rho(\pi)>0$ for $\rho \in \rS$.
\end{thm}

\Cref{thm:main1} is an existence theorem which, when given the level datum $(\OO; \mm, \rS)$, produces an associated ray class field. The map $(\OO; \mm, \rS) \mapsto H_{\mm,\rS}^{\OO}$ is neither one-to-one nor onto:\ A given abelian extension $H/K$ might be a ray class field for none or many different triples $(\OO'; \mm', \rS')$.

To understand \Cref{thm:main1}, it is helpful to compare the quotient ideal $\colonideal{\mm}{\OO_K}$ to more familiar ideals. Note first that, given two orders $\OO\subseteq \OO'$, any (integral) $\OO'$-ideal $\aa$ is automatically an $\OO$-ideal. In particular, all $\OO_K$-ideals are automatically $\OO$-ideals for all orders $\OO$ of $K$. Next, recall that, given
two $\OO$-ideals $\aa, \bb$, the {\em quotient ideal} 
\begin{equation}\label{eqn:quotient-ideal}
\colonideal{\aa}{\bb} := \{x \in K : x\bb \subseteq \aa\}.
\end{equation} 
The quotient ideal $\colonideal{\aa}{\bb}$ is an $\OO$-ideal, and if in addition $\bb$ is an $\OO'$-ideal, then $\colonideal{\aa}{\bb}$ will   also be an $\OO'$-ideal. Thus, given  an  $\OO$-ideal $\mm$, the quotient ideal $\colonideal{\mm}{\OO_K}$ will be an $\OO_K$-ideal, which is the (set-theoretically) largest ideal of $\OO_K$ contained in $\mm$. 

An important invariant of an order $\OO$ of an algebraic number field $K$ is its (absolute) conductor $\ff(\OO) = \colonideal{\OO}{\OO_K}$.  The conductor ideal encodes information on all the non-invertible ideals in the order $\OO$; see \Cref{lem:singular}. It is the (set-theoretically) largest integral $\OO$-ideal that is also an $\OO_K$-ideal. It follows that for any integral $\OO$-ideal $\mm$, we have $ \colonideal{\mm}{\OO_K} \subseteq \ff(\OO)$. 

The ideal $\colonideal{\mm}{\OO_K}$ satisfies the inclusions (all as $\OO_K$-ideals)
\begin{equation}\label{eq:112} 
\ff(\OO) \mm \subseteq \colonideal{\mm}{\OO_K} \subseteq \ff(\OO) \cap \mm\OO_K;
\end{equation} 
these inclusions are proven in \Cref{lem:basicinclusions} (taking $\OO' = \OO_K$ in its statement).
The three ideals in \eqref{eq:112} have the same prime divisors in the maximal order $\OO_K$: A prime ideal $\pp$ of $\OO_K$ such that $\pp \supseteq \ff(\OO) \mm = \ff(\OO) \mm\OO_K $ satisfies either $\pp \supseteq \ff(\OO)$ or $\pp \supseteq \mm\OO_K$, and thus, $\pp \supseteq \ff(\OO) \cap \mm\OO_K$. 

The second result locates the ray class field $H_{\mm,\rS}^{\OO}/K$ of an order $\OO$ for the level datum $(\OO; \mm, \rS)$ as falling between two ray class fields on the maximal order.

\begin{thm}\label{thm:main2} 
For an order $\OO$ in a number field $K$ and any level datum $(\OO; \mm, \rS)$, there are inclusions of ray class fields $H_{\mm\OO_K,\rS}^{\OO_K} \subseteq H_{\mm,\rS}^{\OO} \subseteq H_{\colonideal{\mm}{\OO_K},\rS}^{\OO_K}$. 
\end{thm}

In the special case $\mm=\OO$,  
we have $\colonideal{\mm}{\OO_K}= \ff(\OO)$.
The smallest ray class field $H_{\OO, \emptyset}^{\OO}$ of the order $\OO$ is the ring class field associated to $\OO$, which always contains the (wide) Hilbert class field of $\OO_K$ and whose ramification over $K$ occurs  only at prime $\OO_K$-ideals containing the conductor ideal $\ff(\OO)$.

It follows from \Cref{thm:main2} and \eqref{eq:112}, together with standard class field theory, that the set of all ray class fields of a fixed order $\OO$ is cofinal in the set of all finite abelian extensions of $K$. \Cref{thm:main2} is obtained as  a special case of a general result  relating ray class fields of two given orders $\OO \subseteq \OO'$, given as \Cref{thm:main2b}.

The third result adapts Artin reciprocity to our setting to give a correspondence between class groups and Galois groups of class fields. 
The correspondence asserts that the ray class field $H_{\mm,\rS}^{\OO}$ of an order $\OO$ is associated to an appropriate ray class group $\Cl_{\mm,\rS}(\OO)$ in such a way that a Galois correspondence holds: $\Gal(H_{\mm,\rS}^{\OO}/K) \isom \Cl_{\mm,\rS}(\OO)$ as abelian groups. 

\begin{thm}\label{thm:main3} 
For an order $\OO$ in a number field $K$ and any level datum $(\OO; \mm, \rS)$, with associated ray class field $H_0:=H_{\mm,\rS}^{\OO}$,
there is an isomorphism $\Art_{\OO} : \Cl_{\mm, \rS}(\OO) \to \Gal(H_0/K)$, uniquely determined by its behavior on prime ideals $\pp$ of $\OO$ that are coprime to $\ff(\OO) \cap \mm$, having the property that
\begin{equation}
\Art_{\OO}([\pp])(\alpha) \equiv \alpha^q \Mod{\mathfrak{P}},
\end{equation}
where $\mathfrak{P}$ is any prime of $\OO_{H_0}$ lying over $\pp\OO_K$,
and $q=p^j$ is the number of elements in the finite field $\OO/\pp$.
For any (not necessarily prime) ideal $\aa$ of $\OO$ coprime to $\ff(\OO) \cap \mm$,
\begin{equation}\label{eqn:artin-3}
\Art_{\OO}([\aa]) = \left.\Art([\aa\OO_K])\right|_{H_0},
\end{equation}
where $\Art : \Cl_{\colonideal{\mm}{\OO_K},\rS}(\OO_K) \to \Gal(H_1/K)$ is the usual Artin map in class field theory, with $H_1 = H_{\colonideal{\mm}{\OO_K},\rS}^{\OO_K}$ being a (Takagi) ray class field for the maximal order $\OO_K$, and $H_0 \subseteq H_1$.
\end{thm}

In \Cref{thm:main3}, the set of prime ideals coprime to $\ff(\OO)\cap \mm$ includes all but finitely many of the prime ideals of $\OO$. 

For the classical case of the maximal order $\OO_K$, the Artin map $\Art( \cdot)$ on the right side of \eqref{eqn:artin-3} is given  in Neukirch \cite[Ch.~VI, Thm.~7.1]{Neukirch13} for any ray ideal $\mm$, in the case of $\rS=\rSmax$, the set of all real places of $K$. 
Additionally, the special case $\mm=\OO_K$ with $\rS=\rSmax$ produces the {\em narrow Hilbert class field} (termed by Neukirch the {\em big Hilbert class field}) of $K$, denoted $H_{\OO_K, \rSmax}^{\OO_K}$; compare \cite[Ch.~VI, Prop.~6.8]{Neukirch13}. 
The  alternative choice of\footnote{Neukirch does not introduce the real place parameters $\rS$, and his definition \cite[Ch.~VI, Defn.~6.2]{Neukirch13} of ray class groups corresponds to the choice $\rS=\rSmax$.}
the minimal set $\rS= \emptyset$ of real places produces the original (or {\em wide}) {\em Hilbert class field} (called by Neukirch the {\em small Hilbert class field}) of $K$, denoted $H_{\OO_K, \emptyset}^{\OO_K}$.  The wide class field has Galois group isomorphic to the ideal class group of $\OO_K$; see \cite[Ch.~VI, Prop.~6.9]{Neukirch13}. The difference between wide and narrow Hilbert class fields exhibits the role of the real places $\rS$ in the level datum.

Our three main theorems do not provide an independent development of global class field theory; their proofs are completed using  the existing global class field theory. Their contribution is to identify a new set of distinguished abelian extensions of $K$ whose structure encodes properties of the arithmetic of the ray class groups attached to a fixed order $\OO$ of $K$ (rather than the maximal order).

The proofs of these these results proceed on the ray class group side. We identify the $\Cl_{\mm,\rS}(\OO)$  with a certain quotient group of a Takagi ray class group of the maximal order. The exact sequences in \Cref{thm:exseq} are used. Our definition (\Cref{defn:61}) of the ray class field associated to the level datum $(\OO; \mm, \rS)$ is formulated using this identification.  
\Cref{thm:main1}, \Cref{thm:main2} and \Cref{thm:main3} are then obtained via the main correspondence and existence theorems of global class field theory in a form given in \Cref{sec:6}.

\subsection{Applications}\label{subsec:12} 
The constructions in this paper are motivated by several mathematical phenomena involving ray class groups and ray class fields of orders
that arise outside pure class field theory. We briefly describe three of them.
\begin{enumerate}
\item
Special configurations of complex lines, called \textit{SIC-POVMs (symmetric informationally complete positive operator-valued measures)} or \textit{SICs}, are of interest in quantum information theory and also suggest a geometric approach to explicit class field theory for real quadratic fields that fundamentally involves non-maximal orders. A SIC is a generalized quantum measurement (a \textit{POVM}) with restrictive information-theoretic properties that make it optimal for certain protocols in quantum information processing. It is equivalent to a maximal set of $d^2$ complex lines in $\C^d$ that are ``equiangular'' with respect to the Hermitian inner product. SICs are known to exist in at least dimensions $d \le 53$ and are conjectured by Zauner \cite{Zauner:99, Zauner:11} to exist in all dimensions. 

Recently, a surprising connection has been made between SICs and the explicit class field theory of real quadratic fields \cite{appleby1, appleby2, koppsic}. The connection was originally discovered through numerical experimentation. Work by the current authors and others \cite{KLsic, AFK} uses ray class groups and ray class fields of real quadratic orders in a conjectural framework for classifying SICs. That work also indicates SICs and generalized SICs may be used to study abelian extensions of real quadratic fields.
\item
Class field theory for imaginary quadratic orders arises naturally in the theory of complex multiplication (CM) for elliptic curves. In particular, if $E$ is an elliptic curve with CM by $\OO$ having a Weierstrass equation with coefficients in $\Q(j(E))=H_{\OO,\emptyset}^\OO$, then the $n$-torsion points of $E$ have coordinates in the field $H_{n\OO,\emptyset}^\OO$; see \cite[Thm.~1.4]{BourdonC20}. A related construction describes \textit{elliptic units} in abelian extensions of imaginary quadratic fields as special values of modular functions called \textit{modular units} \cite{KubertLang}; this construction also naturally extends to non-maximal orders.

Ray class fields of orders of higher degree arise in the study of higher-dimensional abelian varieties with CM. Pathologies of non-maximal orders of degree greater than 2, such as potentially being non-Gorenstein, create obstructions to proving theorems in CM theory \cite{Clark22}.
\item
Ray class groups and ray class fields of orders appear naturally in approaches to Hilbert's 12th Problem for real quadratic and complex cubic fields.
For example, when replacing the complex upper half plane by the $p$-adic Drinfeld upper half plane, as in the work of Darmon and Dasgupta on elliptic units for real quadratic fields \cite{DD} and the work of Darmon, Pozzi, and Vonk \cite{DV,DarmonPozziVonk21} on rigid meromorphic cocycles, non-maximal orders arise in the study of ``singular moduli'' in the same way as they do in the classical theory of complex multiplication. Non-maximal orders are needed to describe arbitrary real multiplication (RM) values of rigid meromorphic cocycles. Complex meromorphic modular cocycles introduced by the first author \cite{KoppQpoch} have RM values that are (essentially) the classical Archimedean Stark class invariants conjectured to be algebraic units by Stark in the real quadratic case. The RM values studied are naturally parameterized by elements of ray class groups (more generally, ray class monoids) of orders, and they are conjectured to lie in ray class fields of orders. 
We also expect the elliptic gamma functions connected by Bergeron, Charollois, and Garc\'{i}a to Stark units over complex cubic fields \cite{BCG} to have special values parameterized by ray classes of general complex cubic orders, which generate ray class fields of those orders.
\end{enumerate}

The treatment of integral and fractional ideals of orders, and extension and contraction maps,
given in \Cref{sec:2,sec:2a,sec:CE},
is intended to apply more broadly, outside its use in defining ray class groups of orders. The related concept of {\em ideal lattices} is foundational to various versions of lattice based-cryptography \cite{Micciancio07,MicciancioR09,LPR13} and ring-based schemes proposed for homomorphic encryption \cite{Gentry09,LPR13b}. Ideal lattices are identified with ideals in polynomial rings $\Z[x]/(f(x))$, which for irreducible monic $f(x)$ correspond to ideals of the monogenic order $\OO= \Z[\theta]$ generated by a root of $f$. An ideal lattice can be encoded as an integer matrix associated to an integral ideal of an order, as studied by Taussky \cite{Taussky62,Taussky63,Taussky78}; see also Taussky's earlier work \cite{Taussky49}. A number-theoretic perspective on ideal lattices is given by Bayer-Fluckiger \cite{BayerF99,BayerF02}.

\subsection{Prior work}\label{subsec:13}

There has been an extensive algebraic study of the structure of orders of number fields, beginning with Dedekind \cite{Dedekind:1877} in 1877; see also \cite{Dedekind:1894}. Two general references are Stevenhagen \cite{Stevenhagen08} and Neukirch \cite[Ch.\ 1, Sec.\ 12]{Neukirch13}. Neukirch views orders of number fields as number rings with ``singularities'' at the primes dividing the conductor ideal, by analogy with geometric interpretations of subrings of function fields (e.g., $F[t^2, t^3] \subseteq F(t)$ as the coordinate ring of the cuspidal cubic).

An important 1962 paper of Dade, Taussky, and Zassenhaus \cite{DTZ:62} presented fundamental results on the structure of invertible fractional ideals, class groups of orders, and the class monoids obtained when including non-invertible ideals. They developed a structure theory for one-dimensional Noetherian domains \cite[p.\ 32]{DTZ:62}. (An integral domain $\DD$ has dimension one if and only if all nonzero prime ideals are maximal. Orders in number fields form a strict subclass of one-dimensional Noetherian integral domains.) Dade, Taussky, and Zassenhaus gave a general definition of \textit{fractional ideals} valid for all integral domains $\dD$ (with quotient field denoted $K$) in \cite[Defn.\ 1.1.6]{DTZ:62}. The set $\rJ(\dD)$ of all such fractional ideals of $\dD$ is closed under four operations: $+, \cdot, \cap, \colonideal{}{}$, in which $\cdot$ is ideal multiplication and $\colonideal{}{}$ is the ideal quotient as defined by \eqref{eqn:quotient-ideal}; see \cite[Prop.\ 1.10]{DTZ:62}. They note the set $\rJ(\dD)$ carries the structure of a semigroup under ideal multiplication. Dade, Taussky, and Zassenhaus define a \textit{$\dD$-order} to be any fractional ideal $\aa$ of $\DD$ that is also an integral domain \cite[p.\ 32]{DTZ:62}. Every fractional ideal $\aa$ has an associated $\DD$-order $\ord(\aa) := \colonideal{\aa}{\aa}$, which in other contexts is called its \textit{multiplier ring}. For Noetherian integral domains, they define \textit{invertible fractional ideals} and characterize them using ideal quotient \cite[Defn.\ (p.\ 41), Prop.\ 1.3.6]{DTZ:62}.

Class field theory has gone through many versions. We use the version using ray classes, formulated in the work of Takagi, Hasse, and Artin. A useful treatment is given in Cohn \cite{CohnTau78}. An appendix of Cohn's book gives Emil Artin's 1932 lectures.

Computational class field theory, generally using ray classes, is especially important in applications (such as those mentioned in \Cref{subsec:12}). Cohen and Stevenhagen have written a useful survey \cite{CohenS:08}. 

Ring class groups and ring class fields go back to fundamental work of Weber in 1897--1898 \cite{Weber1897-1898}, motivated (in part) by complex multiplication; see his books \cite{Weber1894, Weber1896, Weber1908}. The ring class groups associated to orders of imaginary quadratic fields appear in the theory of complex multiplication, because ideal classes of those orders are classified by homothety classes of lattices in $\C$ having a given endomorphism ring $\OO \neq \Z$; see \cite[Cor.\ 10.20]{Cox13}. The Weber prime splitting criteria for ring class fields of orders $\OO= \Z[\sqrt{-n}]$ $(n >0)$ over imaginary quadratic fields are expressible in terms primes represented by the principal quadratic forms $x^2+ny^2$. Explaining this connection is the main objective of the book of Cox \cite{Cox13}, as stated in \cite[Thm.~9.2]{Cox13}.

It is known that the compositum of all ring class fields of a field $K$ need not be equal to the maximal abelian extension $K^{\rm ab}$. That is, there can be abelian extensions of $K$ that are not contained in any ring class field. In 1914, Fueter \cite[p.\ 178]{Fueter14} showed that the field $\Q(i, \sqrt[4]{1+2i})$ is not contained in any ring class field of $\Q(i)$; see also Schappacher \cite[p.\ 258]{Schappacher98}. Moreover, Bruckner \cite[Satz 8]{Bruckner66} showed for a quadratic field $K$ that the compositum of ring class fields for $K$ is the compositum of all Galois fields $L$ containing $K$ for which $\Gal(L/\Q)$ is a generalized dihedral group. The case of imaginary quadratic fields is also treated in Cox \cite[Thm.\ 9.18, Cor.\ 11.35]{Cox13}. 

In 2015, Lv and Deng \cite{LvDeng15} treated ring class fields for arbitrary orders in number fields in a classical setting, corresponding to the ``unramified'' case $(\OO; \mm,\rS)=(\OO; \OO,\emptyset)$. (The term ``unramified'' here refers to the fact that the modulus is trivial; the ring class field usually is a ramified extension at the primes containing the conductor.) An extension to general number rings (allowing inversion of arbitrary sets of nonzero elements) was given in 2018 by Yi and Lv \cite{YiLv18}.

There was significant further work on of class field theory for orders in the general ``ramified'' case, in situations related to complex multiplication. In 1935, S\"{o}hngen \cite{Sohngen35} constructed  a ``class field theory for orders'' for imaginary quadratic fields, with explicit generators given by special values of Weber functions, which is described in detail in the book of Schertz \cite[Ch.\ 3, Ch.\ 6.2]{Schertz10}. The ray class field theory for imaginary quadratic orders has been applied to CM theory for elliptic curves; we  mention  Bourdon and Clark \cite{BourdonC20} and Lozano-Robledo \cite[Sec.\ 3]{Lozano-R22}.

In another direction of generalization, in the late 1980's, Stevenhagen \cite{Stevenhagen85,Stevenhagen89} formulated an abstract development of \textit{unramified class field theory for orders}, which extends beyond orders of number fields; see also \cite{Stevenhagen85}. In the case of orders of number fields, his results would specialize to the ``unramified case'' $(\OO; \mm,\rS)=(\OO; \OO,\emptyset)$ treated here. In 2001 Stevenhagen \cite[Sec.\ 4]{Stevenhagen01} recast much of S\"{o}hngen's class field theory for imaginary quadratic orders into a profinite, id\`{e}lic framework.

Very recently Campagna and Pengo \cite[Sec.\ 4]{CampagnaP22} developed a class field theory for orders of general number fields using an  id\`{e}lic framework. They define ray class fields of orders by specifying a particular id\`{e}lic unit group attached to a level datum $(\OO; \mm, \emptyset)$, which defines the associated class field by the idelic main theorem of class field theory. They did not consider ray conditions at the Archimedean places, since their interest was  CM fields. The PhD thesis of Pengo \cite[Sec.\ 6]{Pengo20} formulated an id\`{e}lic definition of ray class fields for orders in the general case.

\subsection{Contents of the paper}\label{subsec:14} 

This paper works in  the classical framework of integral and fractional ideals. Schertz \cite[p.\ 82]{Schertz10} used the classical framework in treating the theory of complex multiplication ``because the classical language marries well with singular values of elliptic and modular functions that are essentially dependent on ideals.'' These functions are used in explicit class field theory.
The classical ideal viewpoint is appropriate also for formulating the structures of monoids of all ray classes, allowing non-invertible classes, which arise in applications \cite{KLmonoid,KoppQpoch}.

\Cref{sec:2} gives an extensive treatment of integral ideals of an order, including invertible and  non-invertible ideals, needed for the constructions of ray class groups and maps between them. It emphasizes the special features of non-maximal orders and contains many examples. This section is needed because many calculations and proofs in the paper use possibly non-invertible integral ideals, sometimes in semilocal rings obtained by inverting elements coprime to a single auxiliary ideal $\dd$. Each non-invertible integral ideal contains some non-invertible prime ideal, and the non-invertible prime ideals are exactly those prime ideals containing the conductor ideal of the order. The set of integral ideals of an order of a number field coprime to a modulus ideal $\mm$ forms a monoid under ideal multiplication.

\Cref{sec:2a}  treats  fractional ideals of an order and gives further examples of behavior special to non-maximal orders. The ideal quotient of two fractional ideals is always well-defined; however, not all fractional ideals are invertible. There exist invertible ideals that are torsion elements of the invertible ideal group. The monoid of all nonzero fractional ideals coprime to the conductor ideal of the order is a free abelian group.

\Cref{sec:CE} studies, inside a fixed number field $K$, the effect of change of order  on the structure of ideals, via the contraction and extension maps on integral and fractional ideals. A detailed treatment is needed for showing good behavior away from the conductor ideal and for proving exact sequences of ray class groups in \Cref{sec:5}. A ring-theoretic subtlety is that the contraction map $\con(\aa) = \aa \cap \OO$ is not a monoid homomorphism for integral ideals. We show the contraction map when restricted to integral ideals coprime to the (relative) conductor ideal $\ff_{\OO'}(\OO)$ is a monoid homomorphism. Using this property, we are able to define a contraction map on fractional ideals coprime to $\ff(\OO)$, which is a homomorphism but no longer agrees with the formula $\con(\aa) = \aa \cap \OO$.

\Cref{sec:group} defines ray class groups of orders for a level datum $(\OO; \mm, \rS)$. It relates such groups under extension and contraction of order. \Cref{lem:reldok} is a key technical result, showing that ray class groups do not change when adding coprimality conditions to an auxiliary modulus $\dd$, permitting comparison of different orders and different moduli. \Cref{subsec:44} determines a surjective map from a general ray class group of an order to a particular ray class group of the maximal order.

\Cref{sec:5} gives exact sequences relating ray class groups under change of order and change of modulus. It first analyzes the effect of change of order $\OO \subseteq \OO'$ on unit groups and principal ideal groups, in \Cref{prop:exseq2}. The exact sequence of \Cref{thm:exseq} in \Cref{subsec:52}, which relates unit groups and class groups of different orders and moduli, is the main formula of this paper for applications. \Cref{subsec:53} gives a formula for a generalized class number, that is, the cardinality of a given ray class group of an order. This formula generalizes a formula in Neukirch \cite[Thm.\ I.\ 12.12]{Neukirch13} in the ring class group case.

\Cref{sec:6} gives the construction of ray class fields of orders. For this purpose, it is necessary to obtain a given ray class group of an order $\OO$ as a quotient group of a particular Takagi ray class group of the maximal order $\OO_K$. This is done in \Cref{subsec:61}.
The desired ray class field of the order is then identified as a subfield $L$ of a Takagi ray class field $H_1$. It is defined as the fixed field $L = H_1^{\Art(\ker(\psi))}$ of the kernel of the map $\psi$ in \eqref{eqn:psi} acting via the Artin map.
\Cref{subsec:62} recalls the main existence theorems of class field theory  in a suitable form encoding both the formulation of Takagi (in terms of prime splitting) and of Artin (in terms of an isomorphism between ray class group and Galois groups). \Cref{subsec:63} proves \Cref{thm:main1} by identifying the map $\psi$ in \eqref{eqn:psi} with the contraction map between fractional ideals of the maximal order $\OO_K$ and the given order $\OO$. \Cref{subsec:64} states and proves a result generalizing \Cref{thm:main2}, replacing the maximal order $\OO_K$ with a general order $\OO'$ with $\OO \subseteq \OO' \subseteq \OO_K$. \Cref{subsec:65} proves \Cref{thm:main3}. 
 
\Cref{sec:examples} presents examples of ray class groups and ray class fields of quadratic orders. 

\Cref{sec:final} presents some remarks on extending results of the paper.
 
\Cref{appendix:norms} discusses norms of ideals of an order and the (consistent) extension of the norm to fractional ideals of an order. The norm is multiplicative when multiplying two fractional ideals, at least one of which is invertible, but is not multiplicative in general.

\subsection{Notation}

\begin{itemize}
\item
$\OO =$ arbitrary order of a number field $K$.
\item
$\ff(\OO)=\colonideal{\OO}{\OO}=$  conductor ideal of $\OO$.
\item
$\OO' =$ another arbitrary order of $K$ satisfying $\OO \subseteq \OO'$.
\item
$\ff_{\OO'}(\OO) = \colonideal{\OO}{\OO'} = $ relative conductor ideal.
\item
$\ord(\aa) = \colonideal{\aa}{\aa} = $ multiplier ring of the $\OO$-ideal $\aa$, which is an order $\OO'$.
\item
$\rI(\OO) =$ monoid of integral ideals of the order $\OO$.
\item
$\rI^{\ast}(\OO) =$ submonoid of $\rI(\OO)$ of integral ideals invertible as fractional ideals of $\OO$.
\item
$\mm=$  general integral $\OO$-ideal; finite part of the ray class modulus.
\item
$\rS=$ subset of the set of real embeddings of $K$; infinite part of the ray class modulus.
\item
$\rI_{\mm}(\OO) =$ monoid of integral ideals of $\OO$ coprime to the integral
ideal $\mm$.
\item
$\rI_{\mm}^{\ast} (\OO) =$ submonoid of $\rI_{\mm}(\OO)$ of integral ideals invertible as fractional ideals of $\OO$.
\item
$\rJ(\OO) =$ monoid of all  fractional ideals of $\OO$
\item
$\rJ^{\ast}(\OO) =$ group of invertible fractional ideals of $\OO$.
\item
$\rJ_{\mm} (\OO) =$ monoid of fractional ideals coprime to the
(nonzero) integral ideal
 $\mm$ of $\OO$.
\item
$\rJ_{\mm}^{\ast}(\OO) =$ group of invertible fractional ideals coprime to $\mm$.
\item
$\rP(\OO) =$ group of nonzero principal fractional ideals $\alpha\OO$, with $\alpha \in K^\times$.
\item
$\rP_{\mm, \rS}(\OO) =$ group of nonzero principal fractional ideals $\alpha\OO$, with 
$\alpha \in K^\times$, $\alpha \equiv 1 \Mod{\mm}$, and $\rho(\alpha)>0$ for $\rho \in \rS$. 
\item
$\rP_{\mm, \rS}^{\dd}(\OO) =$ subgroup of $\rP_{\mm, \rS}(\OO)$ of 
all $\alpha\OO= \colonideal{\aa}{\bb} = \aa \bb^{-1}$ an ideal quotient of invertible integral ideals $\aa, \bb$, each coprime to a given integral ideal $\dd$ of $\OO$.
\item
$\Cl_{\mm,\rS}(\OO)=$ ray class group of order $\OO$ with modulus $(\mm, \Sigma)$.
\item
$\ddD :=(\OO; \mm, \rS)$ abbreviates a (ray class) level datum, where $\OO$ is an order of a number field $K$, $\mm$ is an integral $\OO$-ideal, and $\rS$ is a subset of the real places of $K$.
\end{itemize}

\section{Ideals of orders}\label{sec:2} 

Let $\OO_K$ be the maximal order of all algebraic integers in a number field $K$. Then $\OO_K$ is a Dedekind domain, having unique prime factorization of nonzero integral ideals. All nonzero fractional ideals are invertible, and they form a free abelian group. One has an ideal class group defined as the quotient of the group of nonzero fractional ideals by the group of nonzero principal ideals. One can define ray class groups by restricting to ideals coprime to a modulus $\mm$ and quotienting by the group of principal prime ideals having a generator $\alpha \equiv 1 \Mod{\mm}$ and with some positivity conditions at a subset $\Sigma$ of real places. 

Non-maximal orders of a number field are never Dedekind domains; rather, they are one-dimensional Noetherian integral domains that are not regular rings.
The ideal theory of non-maximal orders has notable differences from that of the maximal order.
Not all integral ideals factor into prime ideals (uniquely or otherwise). There exist non-invertible integral ideals.
``Greatest common divisor'' and ``least common multiple'' of ideals are not well-defined, although there is a notion of {\em coprimality} of two ideals.
The failure of unique factorization into prime ideals is restored on restricting to ideals coprime to the conductor ideal $\ff(\OO)= \colonideal{\OO}{\OO_K}$.
In addition, the monoid of invertible integral ideals need not be free.

In the rest of \Cref{sec:2} and in \Cref{sec:2a}, we state and prove required foundational results at varying levels of generality, always restricted to commutative rings with unity. In decreasing generality, these include integral domains, Noetherian integral domains, Noetherian integral domains of dimension one, and orders $\OO$ of algebraic number fields. 

Compared to more general integral domains, orders of a number field $K$ have strong finiteness properties arising from the $\Q$-lattice structure on $K$.
A \textit{full rank $\Z$-lattice} of $K$ is any $\Z$-module $\Lambda=\alpha_1\Z + \alpha_2\Z + \cdots + \alpha_n\Z$ with each $\alpha_i \in K$, having $\Z$-rank $n = [K:\Q]$. 
The \textit{multiplier ring} $\ord(\Lambda)$ of a full rank $\Z$-lattice $\Lambda$ of a number field $K$ is given by 
\begin{equation}
\ord(\Lambda) = \colonideal{\Lambda}{\Lambda} := \{\alpha \in K : \alpha \Lambda \subseteq \Lambda\},
\end{equation}
and it is an order of $K$. The multiplier ring is the largest order (as a set) such that $\Lambda$ is a fractional ideal of that order. 
Each order $\OO$ of $K$ is the multiplier ring of some full rank $\Z$-lattice, namely, itself:~$\OO = \ord(\OO)$.
Finiteness properties of orders of number fields include the well-known finiteness of the class group $\Cl(\OO)$ of an order.

\subsection{Integral ideals, prime ideals, and primary ideals} \label{subsec:21} 

We start with a general integral domain $\DD$.
An \textit{integral ideal} (or simply an \textit{ideal} or \textit{$\DD$-ideal}) $\aa$ of $\DD$ is an $\DD$-submodule $\aa \subseteq \DD$.
The \textit{$\DD$-ideal product} $\aa \bb$ of two $\DD$-ideals $\aa, \bb$ is the 
$\DD$-ideal
\begin{equation}\label{eq:idealproduct}
\aa\bb = \aa\cdot\bb = \left\{\sum_j \alpha_j\beta_j : \alpha_j \in \aa, \beta_j \in \bb\right\}. 
\end{equation}

We let $\rI(\DD)$ denote the set of integral ideals of an integral domain, which forms a monoid for the operation of $\DD$-ideal product, with $\DD$ as the identity.
A $\DD$-ideal $\pp$ is \textit{prime} if $xy \in \pp$ implies $x \in \pp$ or $y \in \pp$, and $\pp \neq \DD$. 
Additionally, there is a notion of coprimality of $\DD$-ideals. 
\begin{defn}[Coprimality of integral ideals] 
An integral ideal $\aa \subseteq \DD$ of an integral domain $\DD$ is said to be \textit{coprime} (or \textit{relatively prime}) to another integral ideal $\mm \subseteq \DD$ if $\aa + \mm = \DD$.
\end{defn}
If $\aa, \bb$ are both coprime to $\mm$, then their  product $\aa\bb$ is coprime to $\mm$, because
\begin{equation}
\DD =(\aa+\mm)(\bb+\mm)= \aa\bb + \aa\mm+\bb\mm + \mm\mm\subseteq  \aa\bb+\mm \subseteq \DD.
\end{equation}

We now specialize $\DD$ to be a Noetherian integral domain. Thus, all ideals of $\DD$ are finitely generated.

Commutative Noetherian rings generally do not have unique factorization into products of powers of prime ideals; they possess a weaker form of decomposition of ideals under intersection, called \textit{primary decomposition} \cite[Ch.~4]{AM:69}.
A \textit{primary ideal} $\qq$ is an ideal such that, if $xy \in \qq$, then either $x \in \qq$ or $y^n \in \qq$ for some $n \ge 1$, and $\qq \neq \DD$. Any power of a prime ideal is primary. (Indeed, any power of a primary ideal is primary.) 
The primary decomposition for a Noetherian ring is a decomposition of an ideal into an intersection of primary ideals. 
In a commutative Noetherian ring, all ideals have a primary decomposition (Lasker--Noether theorem); however, a primary decomposition is not necessarily unique.

For one-dimensional Noetherian domains, stronger results hold. One-dimensional Noetherian domains are Noetherian domains for which all nonzero prime ideals are maximal. The class of one-dimensional Noetherian domains includes all orders of number fields and is closed under localization and under completion. For a one-dimensional Noetherian domain $\DD$,  primary decompositions of ideals exist and are unique. In addition, the primary decomposition given as an intersection of ideals coincides with its decomposition as a product of the same primary ideals. 
To state the result precisely, recall that the \textit{radical} of an ideal is
\begin{equation}\label{eqn:radical} 
\rad(\mm) := \{x \in A : x^n \in \mm \mbox{ for some } n \ge 1\}.
\end{equation} 
The radical $\rad(\qq)$ of a primary ideal is the unique prime ideal $\pp$ containing $\qq$. We say that such a primary ideal $\qq$ is \textit{associated} to the prime ideal $\pp$, or alternatively, that $\qq$ is {\em $\pp$-primary.}
 
\begin{prop}[Primary decomposition in dimension 1]\label{prop:221}
Let $\DD$ be a commutative Noetherian integral domain in which all nonzero prime ideals are maximal (i.e., $\DD$ has Krull dimension 1).
Then
\begin{enumerate}
\item Every non-zero ideal $\mm$ in $\DD$ has a unique primary decomposition 
\begin{equation}
\mm = \bigcap_{i} \qq_i,
\end{equation}
in which $\qq_i$ are primary ideals whose radicals $\pp_i= \rad(\qq_i)$ are pairwise distinct. 
\item The primary decomposition agrees with its product decomposition
\begin{equation}\label{eqn:pd} 
\mm= \prod_{i} \qq_i.
\end{equation} 
\end{enumerate}
\end{prop} 
\begin{proof} 
This is \cite[Prop.\ 9.1]{AM:69}. The last assertion \eqref{eqn:pd} is established in its proof.
\end{proof} 
 
\Cref{prop:221}(2) gives a form of unique factorization into pairwise coprime ideals, in which every factor is primary.
 
A prime ideal $\pp$ is called \textit{non-singular} if all the $\pp$-primary ideals are powers of $\pp$; it is \textit{singular} otherwise. The singular prime ideals of orders of number fields are characterized in \Cref{lem:singular}. Each order $\OO$ of $K$ has finitely many singular prime ideals; the maximal order $\OO_K$ is the only order having no singular prime ideals.
 
\subsection{Invertible integral ideals of orders of number fields} \label{subsec:22} 
 
Recall that there is associated to each integral $\OO$-ideal $\aa$ of a number field a \textit{multiplier ring}
\begin{equation}
\ord(\aa) := \colonideal{\aa}{\aa} = \{x \in K : x\aa \subseteq \aa\}.
\end{equation} 
Necessarily $\OO \subseteq \ord(\aa) \subseteq \OO_{K}$.
All orders $\OO'$ between $\OO$ and $\OO_K$ occur this way; one may choose $\gamma \in K^{\times}$ so that  $\aa=\gamma \OO' \subseteq \OO$, and $\aa$ is then an integral $\OO$-ideal having $\ord(\aa) = \OO'.$

\begin{defn}[Invertible integral ideal]\label{defn:invertible-integral} 
An integral ideal $\aa$ of $\OO$ is \textit{invertible} if there exists another integral $\OO$-ideal $\bb$ and a nonzero $\gamma \in \OO$ such that the $\OO$-ideal product $\aa\bb = \gamma \OO$. Otherwise $\aa$ is {\em non-invertible}.
\end{defn} 

The invertibility property is preserved under $\OO$-ideal product. If $\aa \cc = \lambda \OO$ and $\bb \dd = \mu\OO$, then $(\aa\cc) (\bb\dd) = \lambda\mu \OO$, so $\aa \bb$ is invertible. This statement also has a converse.

\begin{lem}\label{lem:invertible}
Let $\OO$ be an order of a number field.
\begin{itemize}
\item[(1)] If $\cc$ is an invertible integral $\OO$-ideal, and $\cc= \aa\bb$ as $\OO$-ideals,  then both $\aa$ and $\bb$ are 
invertible $\OO$-ideals.
\item[(2)] If $\aa,  \bb, \cc$ are  $\OO$-ideals (invertible or not) and  $\ord(\cc)= \OO$, then $\cc=\aa\bb$ implies
$\ord(\aa) = \ord(\bb) = \OO$.
In particular, all invertible integral $\OO$-ideals  $\aa$ have $\ord(\aa)= \OO$. 
\end{itemize}
\end{lem} 
\begin{proof}
[Proof of (1)] Given  $\cc= \aa \bb$ is invertible, then there is an integral ideal $\dd$ with $\cc\dd=\gamma\OO$.
Now $\aa$ is invertible since  $\aa(\bb\dd)= \cc\dd=\gamma\OO$, and similarly for $\bb$. 
\end{proof}
\begin{proof}
[Proof of (2)] Given $\cc=\aa\bb$, we have $\ord(\aa) \subseteq \ord (\aa\bb)= \ord(\cc)$. Since $\OO \subseteq \ord(\aa) \subseteq \ord(\cc)$ and  $\ord(\cc) = \OO$ by hypothesis, we have $\ord(\aa) =\OO$; similarly $\ord(\bb) = \OO$. For invertible ideals we have $\aa \cc= \gamma \OO$. Since $\ord(\gamma \OO) = \OO$, we deduce $\ord(\aa)=\OO$.
\end{proof}  

We let $\rI^{\ast}(\OO)$ denote the monoid of invertible integral ideals under ideal multiplication. All nonzero principal ideals $\aa= \alpha \OO \in \rI^{\ast}(\OO)$, because $(\alpha\OO)(\OO) = \alpha\OO$. 

Not all ideals of a general order $\OO$ are invertible; a necessary condition for invertibility of an $\OO$-ideal $\aa$ is that $\ord(\aa)=\OO$, as given in \Cref{lem:invertible}(2). 
The next example shows this necessary condition is in general not sufficient.

\begin{eg}\label{exam:26} 
[Non-invertible ideal $\qq$ of $\OO$ with $\ord(\qq)=\OO$]
(This phenomenon occurs only for number fields $K$ with $[K:\Q] \ge 3$ having a non-Gorenstein order $\OO$; see also \cite{Clark22}.)
Take $K= \Q(\sqrt[3]{2})$, and consider three orders of $K$ with $\OO \subsetneq \OO' \subsetneq \OO_K$ given by:
\begin{align}
\OO &= \Z[2 \sqrt[3]{2}, 2 \sqrt[3]{4}]    = \Z + 2 \sqrt[3]{2}\Z + 2\sqrt[3]{4}\Z; \\
\OO' &= \Z[\sqrt[3]{4}]  = \Z + 2 \sqrt[3]{2}\Z + \sqrt[3]{4}\Z; \\
\OO_K &= \Z[\sqrt[3]{2}]  = \Z +  \sqrt[3]{2}\Z + \sqrt[3]{4}\Z.
\end{align}
We set $\qq= 2\Z + 2 \sqrt[3]{2}\Z + 4 \sqrt[3]{4}\Z$ and note it is a $\OO$-ideal of index $4$ in $\OO$, hence a primary ideal of $\OO$. 
We show $\qq$ is not invertible as an $\OO$-ideal  by contradiction. If it were invertible, then $\qq^2$ would also be an invertible $\OO$-ideal.  Now $\qq^2 = 4\Z + 4 \sqrt[3]{2}\Z + 4\sqrt[3]{4}\Z = 4 \OO_K$ has $\ord(\qq^2) = \colonideal{\qq^2}{\qq^2}= \OO_K$, so $\qq^2$ is not invertible for $\OO$ by \Cref{lem:invertible}(2), a contradiction.

Secondly, $\qq$ has multiplier ring $\ord(\qq) = \colonideal{\qq}{\qq}=\OO$. To see this, we observe that the only orders containing $\OO$ are $\OO'$ and $\OO_K$. It suffices to show $\qq\OO' \ne \qq$, which holds because $2 \in \qq$ and $\sqrt[3]{4} \in \OO'$ have product $2 \sqrt[3]{4} \nin \qq$. Thus, $\qq$ is not an $\OO'$-ideal, hence $\ord (\qq)= \OO$. 
\end{eg}

There is a factorization theory for invertible integral ideals based on a notion of irreducible integral ideal. 
In general this theory does not result in unique factorizations, nor are irreducible integral ideals always prime.

\begin{defn}[Irreducible integral ideal]\label{defn:irred-integral} 
An invertible integral ideal $\qq$ is said to be \textit{irreducible} for the one-dimensional Noetherian domain $\DD$ if $\qq \neq \DD$ and the factorization $\qq = \aa \bb$ for invertible ideals $\aa,\bb \subseteq \DD$ implies that $\aa=\DD$ or $\bb=\DD$. 
\end{defn} 

Irreducible invertible ideals of one-dimensional Noetherian domains $\DD$ are necessarily primary ideals. (If they were not primary, they would have a nontrivial primary decomposition, contradicting irreducibility.) There may be more than one irreducible invertible ideal whose radical is a given prime ideal, as well as more than one irreducible invertible ideal associated to a given non-Archimedean valuation on $K$; both phenomena are illustrated by \Cref{exam:210}.

\begin{eg}[Primary and prime ideals in non-maximal orders; invertibility]\label{exam:29} 
Consider $K= \Q(\sqrt{-13})$, which has ring of integers $\OO_K = \Z + \sqrt{-13}\Z$ of discriminant $-52$.
Let $q$ be an inert prime in $\OO_K$; for example, $q=5$.
Then $\mm = q\OO_K = q\Z + q\sqrt{-13}\Z$ is a maximal ideal of $\OO_K$ of norm $q^2$ in $\OO_K$.
It is an invertible principal ideal in the maximal order $\OO_K$.

Consider the non-maximal order $\OO = \Z + q\sqrt{-13}\Z$.
The lattice $\mm= q\Z + q \sqrt{-13}\Z$ is a maximal $\OO$-ideal; hence, it is a prime ideal of $\OO$.
It has $\ord(\mm)= \OO_K$, so it is not an invertible integral ideal of $\OO$.
Since it is not invertible, it cannot be a principal ideal of $\OO$.

On the other hand, the ideal $\qq := q \OO = q\Z + q^2 \sqrt{-13}\Z$, which has $\qq \subseteq \mm$,
is a principal ideal of $\OO$; hence, 
it is an invertible $\OO$-ideal. It is is a primary ideal of $\OO$, and 
its associated prime ideal in $\OO$ is $\rad(\qq)= \mm$, noting that $(q \sqrt{-13})^2 \in \qq$. 
\end{eg} 

\begin{eg}[Nonunique factorization of an invertible ideal  into irreducible  invertible factors] \label{exam:210} 
Let $K=\Q(\sqrt{2})$ with $\OO_K = \Z + \sqrt{2}\Z$, and let $\OO = \Z + 2\sqrt{2}\Z$.
Then $\OO$ does not contain the fundamental unit $\e = 1+ \sqrt{2}$, but does contain $\e^2= 3+ 2 \sqrt{2}$. 
Now the two $\OO$-ideals $\qq_1= (2\e) \OO = 4\Z + (2+2 \sqrt{2})\Z$ and $\qq_2 = 2\OO = 2\Z + 4 \sqrt{2}\Z$ are principal $\OO$-ideals, hence invertible. 
They are both primary ideals associated to the prime ideal 
$\pp = 2\Z +  2\sqrt{2}\Z=2\OO_K \subsetneq \OO$, which is not invertible.
The ideal $\pp$ has index $2$ in $\OO$.
Therefore $\qq_1$ and $\qq_2$, which are each of index $4$ in $\OO$ and index $2$ in $\pp$, must be irreducible.
One has 
\begin{equation}
(\qq_1)^2 = (\qq_2)^2 = 4 \OO = 4\Z + 8\sqrt{2}\Z.
\end{equation}
Thus the invertible ideal $4\OO$ has two different irreducible factorizations.
\end{eg} 

\subsection{Conductors and relative conductors of orders of number fields}\label{subsec:23} 

The conductor ideal $\ff(\OO)$ of an order $\OO$ of a number field is an important invariant of the order that contains information on the non-invertible ideals of $\OO$. 

\begin{defn}\label{def:conductor} 
The absolute and relative conductor are defined as follows.
\begin{enumerate}
\item
The \textit{(absolute) conductor} of $\OO$ (in $\OO_K$) is 
\begin{equation}\label{eq:abs-conductor} 
\ff(\OO) := \ff_{\OO_K}\!(\OO) = \colonideal{\OO}{\OO_K} = \{\alpha \in \OO_K : \alpha\OO_K \subseteq \OO\}.
\end{equation}
It is the largest $\OO_K$-ideal in $\OO$. 
\item
More generally, if $\OO \subseteq \OO'$,
then the \textit{relative conductor} of $\OO$ in $\OO'$ is
\begin{equation}\label{eq:rel-conductor}
\ff_{\OO'}\!(\OO) = \colonideal{\OO}{\OO'} = \{\alpha \in \OO' : \alpha\OO' \subseteq \OO\}.
\end{equation}
It is the largest $\OO'$-ideal in $\OO$.
\end{enumerate}
\end{defn}

The absolute conductor ideal $\ff(\OO) = \ff_{\OO_K}\!(\OO)$ is contained in all relative conductors $\ff_{\OO'}(\OO)$.
 
\begin{eg}[Conductors of quadratic orders]
If $K$ is a quadratic field of discriminant $\Delta$, then the maximal order of $K$ is given by $\OO_K = \OO_\Delta = \Z\left[\frac{\Delta+\sqrt{\Delta}}{2}\right]$. The orders of $K$ are of the form
\begin{equation}
\OO_{f^2\Delta} = \Z\left[\frac{f^2\Delta+\sqrt{f^2\Delta}}{2}\right] = \Z + f\frac{\Delta+\sqrt{\Delta}}{2}\Z
\end{equation}
for $f \in \N$. The order $\OO_{f^2\Delta}$ has discriminant $f^2\Delta$. We have $\OO_{f^2\Delta} \subseteq \OO_{(f')^2\Delta}$ if and only if $f'|f$, and the relative conductor is $\ff_{\OO_{(f')^2\Delta}}\!\left(\OO_{f^2\Delta}\right) = \frac{f}{f'}\OO_{(f')^2\Delta}$.
\end{eg}

In the quadratic field case, the absolute conductor determines the order. This does not hold in general, as the following biquadratic example shows.

\begin{eg}[The absolute conductor does not determine the order]\label{exam:abs-cond}
Let $K$ be the field generated by the $12$-th roots of unity, and write it as a biquadratic field $K = \Q(\omega,i)$, where $\omega^2+\omega+1=0$ and $i^2+1=0$. The maximal order of $K$ is $\OO_K = \Z[\omega,i]$.

Consider the two orders $\OO \subsetneq \OO' \subsetneq \OO_K$ given by
\begin{align}
\OO &= \Z[5\omega,5i,5\omega i] = \Z+5\omega\Z+5i\Z+5\omega i\Z \mbox{ and} \\ 
\OO' &= \Z[\omega,5i]= \Z+\omega\Z+5i\Z+5\omega i\Z.
\end{align}
If $\alpha = w+\omega x+iy+\omega iz \in \ff(\OO)$, then $\alpha \in \OO \implies 5|x, 5|y, 5|z$, and $i\alpha \in \OO \implies 5|w$, so we see that $\ff(\OO) = 5\OO_K$. On the other hand, if $\alpha = w+\omega x+iy+\omega iz \in \ff(\OO')$, then $\alpha \in \OO' \implies 5|y, 5|z$, and $i\alpha \in \OO' \implies 5|w, 5|x$, so we see that $\ff(\OO') = 5\OO_K$. Thus, $\OO$ and $\OO'$ have the same absolute conductor $5 \OO_K$.
\end{eg}

\begin{eg}[Noninvertible relative conductor]
We compute the relative conductor for orders $\OO$ and $\OO'$ in \Cref{exam:abs-cond}.
Taking $\alpha = w+\omega x+iy+\omega iz \in \ff_{\OO'}(\OO)$, then $\alpha \in \OO \implies 5|x, 5|y, 5|z$, and $\omega\alpha \in \OO \implies 5|(w-x)$ and thus $5|w$; we see that
the relative conductor $\ff_{\OO'}(\OO) = 5\OO_K$. This relative conductor is not only an $\OO$-ideal, it is also an $\OO_K$-ideal. 
Therefore it is not invertible as an integral $\OO$-ideal by \Cref{lem:invertible}(2).
\end{eg} 

The conductor determines all singular prime ideals.

\begin{lem}\label{lem:singular} 
Let $\OO$ be an order of a number field and $\pp$ a nonzero prime ideal of $\OO$. The following are equivalent.
\begin{enumerate}
\item[(1)]
$\pp$ is not coprime to $\ff(\OO)$, that is, $\pp + \ff(\OO) = \pp$.
Equivalently, $\ff(\OO) \subseteq \pp$.
\item[(2)]
$\pp$ is a non-invertible prime ideal of $\OO$.
\item[(3)]
$\pp$ is a singular prime ideal of $\OO$.
\end{enumerate} 
\end{lem} 
\begin{proof} 
For any prime ideal $\pp$ we  have $\pp \subseteq \pp+ \ff(\OO) \subseteq \OO$. Since all nonzero prime ideals $\pp$ are maximal, 
$\pp$ is not coprime to $\ff(\OO)$ if and only if $\pp + \ff(\OO) = \pp$.

\textit{Proof of $(1) \Leftrightarrow (2)$.} 
The contrapositive of this assertion says: $\pp$ is an invertible prime ideal of $\OO$
if and only if 
$\ff(\OO) \not\subseteq \pp$. Since $\pp$ is maximal, the latter means $\pp+\ff(\OO) = \OO$, that is, $\pp$ is coprime to $\ff(\OO)$.
The contrapositive is proved as \cite[Ch.\ I, Prop.\ 12.10]{Neukirch13}.

\textit{Proof of $(2) \Leftrightarrow (3)$.} 
The contrapositive of the assertion says: 
$\pp$ is an invertible prime ideal of $\OO$ if and only if $\pp$ is a non-singular prime ideal of $\OO$. 
Now a prime ideal $\pp$ is invertible if and only if $\pp\OO_{\pp}=\alpha_{\pp}\OO_{\pp}$ is a principal ideal, where $\OO_{\pp}$ denotes the localization of $\OO$ at the maximal ideal $\pp$. (See \cite[Ch.\ I, Sec.~11, 12 and Lem.~12.4]{Neukirch13}.)
Equivalently, the localization $\OO_{\pp}$ is a discrete valuation ring, so the set of all nonzero ideals in the local ring are powers of the maximal ideal $\pp\OO_{\pp}$, that is, $\pp$ is non-singular, as required.
\end{proof} 

In the next example, the {\em norm} of an integral ideal $\aa$ of an order $\OO$ is the index
\begin{equation}
\Nm_{\OO}\!\left(\aa\right) := [\OO: \aa]
\end{equation}
with $\OO$ and $\aa$ regarded as $\Z$-modules. This definition coincides with the usual definition of {\em norm} of an ideal for the maximal order. Norms are treated in \Cref{appendix:norms}, where the definition of {\em norm} is extended to fractional ideals of an order. 

\begin{eg}
[Structure of primary ideals for $\pp$ containing $\ff(\OO)$; failure of gcd and lcm] 
\label{exam:214} 
The set of primary ideals $\qq$ associated to  a prime ideal $\pp$ containing the conductor ideal $\ff(\OO)$ 
can have a very complicated structure.
\begin{enumerate}
\item[(1)]
The containment of ideals $\qq_1 \subseteq \qq_2$  does not  always imply the divisibility of $\qq_1$ by $\qq_2$.
\item[(2)] 
For suitable pairs of  primary ideal, non-existence of $\gcd$ and of $\lcm$ can occur.
\item[(3)]
The ideal norm function may not multiplicative on primary ideals.
\end{enumerate}

Let $K=\Q(i)$ with maximal order $\OO_K = \Z[i]= \Z + i\Z$, and consider the order $\OO = \Z[2i] = \Z + 2i\Z$ of index $2$ in $\OO_K$. The conductor ideal  $\ff(\OO) = 2\OO_K = 2\Z + 2i \Z$ is a maximal  ideal of $\OO$, since it is of index $2$ in $\OO$.  
We consider the set $\mathcal{F}$ of all primary ideals associated to $\ff(\OO)$, and sort them by the size of their norm, which is always a power of $2$. The unique element of norm $2$ in $\OO$  is $\QQ_2 := \ff(\OO)$. Since $\ord(\QQ_2) = \OO_K$ is larger than $\OO$, $\QQ_2$ is not invertible for $\OO$ by \Cref{lem:invertible}(2).

There are two invertible ideals in $\mathcal{F}$ of norm $4$, both principal: 
\begin{equation} 
\qq_4 := 2 \OO= 2\Z + 4i \Z \quad \mbox{and} \quad \qq_4' := 2i \OO = 4\Z + 2i \Z.
\end{equation} 
There is also one  non-invertible ideal of norm $4$,
\begin{equation}
\QQ_4 := 2(1+i) \OO_K = 4\Z + (2+2i)\Z.
\end{equation} 
Both $\qq_4$ and $\qq_4'$ are contained in $\QQ_2$, but they cannot be divisible by $\QQ_2$ because it is a non-invertible ideal, while all divisors of invertible ideals are invertible by \Cref{lem:invertible}(1). This illustrates (1). 

Note that $\qq_4$ and $\qq_{4}'$ are irreducible integral ideals, as their only (integral ideal) common divisor is $\OO$. However, they are not coprime integral ideals, since $\qq_4 + \qq_4'= \QQ_2$. 

In addition $\QQ_4$ is an irreducible $\OO$-integral ideal, because it is not divisible by $\QQ_2$ viewed as an $\OO$-integral ideal.  This holds since $(1+i) \OO_K$ is not an integral $\OO$-ideal, as it contains $1+i \not\in \Z[2i]$. ($\QQ_2$ is divisible by $(1+i)\OO_K$ viewed as an $\OO_K$-ideal.) 

There is a single ideal in $\mathcal{F}$ of norm $8$, which is $\QQ_{8} := (\QQ_2)^2 = 4\OO_K = 4\Z + 4i \Z$. It is a non-invertible ideal, since $\ord(\QQ_{8}) = \OO_K$. Now $\QQ_{8} = (\QQ_2)^2$, so we have 
\begin{equation}
8 = \Nm_{\OO}(\QQ_8) = \Nm_{\OO}(( \QQ_2)^2) > (\Nm_{\OO} \QQ_2)^2 = 4.
\end{equation} 
This illustrates that the ideal norm on $\OO$ is not always multiplicative, giving assertion (3).

We also note that 
$\qq_4 \QQ_2 =  \qq_4' \QQ_2 = \QQ_{8}$.
The norms are multiplicative in this case, since $\qq_4$ is invertible; see \Cref{prop:normmult} in \Cref{appendix:norms}.

We now study the primary ideals of norm $16$ in $\mathcal{F}$. There are two invertible ideals:
\begin{equation}
\qq_{16} := 4 \OO= 4\Z + 8i \Z \quad \mbox{and} \quad \qq_{16}' := 4i \OO =  8\Z+ 4i \Z.
\end{equation} 
There is also a non-invertible ideal of norm $16$,
\begin{equation}
\QQ_{16} := 4(1+i) \OO_K = 8\Z + (4+4i)\Z. 
\end{equation} 
As invertible ideals, $\qq_{16}$ and $\qq_{16}'$   cannot be divisible by $\QQ_2$, $\QQ_4$ or $\QQ_{8}$.
We find that
\begin{equation}\label{eq:cover}
\qq_{16} = (\qq_4)^2= (\qq_4')^2 \quad \mbox{and} \quad \qq_{16}' = \qq_4 \qq_4'.
\end{equation}  
It follows that each of  $\qq_{16}$ and $\qq_{16}'$ are contained in both
$\qq_4$ and in $\qq_4'$. 
In the divisibility partial order, the factorization \eqref{eq:cover} shows that there are no intermediate elements of the partial order between $\qq_{4}$ (respectively, $\qq_4'$) and either $\qq_{16}$ or $\qq_{16}'$ (since $\qq_4$ and $\qq_4'$ are irreducible).
Consequently, the $\lcm$ of $\qq_{4}$ and $\qq_{4}'$ is not well-defined, and the $\gcd$ of $\qq_{16}$ and $\qq_{16}'$ is not well defined. This illustrates (2). 

The divisibility partial order relating these ideals is pictured in \Cref{fig:21}.

\begin{figure}[h] 
\begin{tikzcd}
{{\mathfrak q}_{16}} \ar[dash]{d} \ar[dash]{dr} & {\mathfrak q}_{16}' \ar[dash]{d} \ar[dash]{dl} \\
{\mathfrak q}_{4} & {\mathfrak q}_4'
\end{tikzcd}
\caption{Hasse diagram for $\{\qq_{4}, \qq_4', \qq_{16}, \qq_{16}'\}$.} 
\label{fig:21} 
\end{figure}

\end{eg} 

\begin{rmk}[Structure of primary ideals for prime ideals $\pp$ not containing $\ff(\OO)$]\label{rmk:reg}
For prime ideals $\pp$ coprime to the conductor ideal $\ff(\OO)$, none of the pathologies (1)--(3) shown in \Cref{exam:214} occur. The contrapositive of \Cref{lem:singular} shows that $\pp$ is an invertible ideal and is non-singular, that is, its complete
set of $\pp$-primary ideals is $\{\pp^j : j \ge 1\}$. For $\pp$-primary ideals, containment is equivalent to divisibility, so $\gcd$ and $\lcm$ are well-defined. The $\OO$-norms of $\pp$-primary ideals are multiplicative by \Cref{prop:normmult}, because all $\pp^j$ are invertible. 
\end{rmk}

%
%
\subsection{Monoids of integral ideals coprime to a fixed modulus $\mm$}\label{subsec:24a} 

Recall that $\rI(\OO)$ denotes the monoid of all integral ideals and $\rI^{\ast}(\OO)$ the monoid of all invertible integral ideals.
We introduce monoids of integral ideals and invertible integral ideals coprime to a fixed modulus $\mm$.

\begin{defn}[Monoids of integral ideals coprime to $\mm$]\label{defn:ideal-monoids}  
Given an integral ideal $\mm$ of an order $\OO$ of a number field, let $\rI_{\mm}(\OO)$ denote the set of {\em $\mm$-coprime integral  ideals} of $\OO$, and let $\rI_{\mm}^{\ast}(\OO)$ denote the set of {\em $\mm$-coprime invertible integral ideals}. That is,
\begin{align}
\rI_{\mm}(\OO) &= \{\aa \in \rI(\OO) : \aa+\mm = \OO\}, 
& \mbox{and} &&
\rI_{\mm}^{\ast}(\OO) &= \{\aa \in \rI^{\ast}(\OO) : \aa+\mm = \OO\}.
\end{align}
\end{defn} 

\begin{lem}\label{lem:monoid-well-defined} 
Let $\OO$ be an order of a number field and $\mm$ an integral ideal of $\OO$.
\begin{enumerate}
\item[(1)]
The sets $\rI_{\mm}(\OO)$ and $\rI_{\mm}^{\ast}(\OO)$ are monoids for $\OO$-ideal product.
\item[(2)]
If $\mm, \mm'$ are ideals of $\OO$  with $\mm \subseteq \mm'$, then $\rI_{\mm}(\OO) \subseteq \rI_{\mm'}(\OO)$, and $\rI_{\mm}^{\ast}(\OO) \subseteq \rI_{\mm'}^{\ast}(\OO)$.
\end{enumerate}
\end{lem}

\begin{proof}
[Proof of (1)]
We check that these sets are closed under $\OO$-ideal product. If $\aa + \mm=\OO$ and $\bb+\mm=\OO$, then
\begin{equation}
\OO = (\aa + \mm)(\bb+\mm) = \aa\bb+ (\aa\mm+ \bb\mm+ \mm \mm) = \aa\bb+\mm.
\end{equation}
We noted earlier that  invertibility of integral ideals is preserved under $\OO$-ideal product.
\end{proof}
\begin{proof}
[Proof of (2)]
If $\aa+\mm=\OO$ and $\mm \subseteq \mm'$, then $\aa+ \mm'=\OO$, since $\aa+ \mm \subseteq \OO$.
\end{proof}

The next result shows that for any $\mm$ contained in the conductor ideal, the associated monoid of integral ideals coprime to $\mm$ has unique factorization into prime ideals.

\begin{prop}\label{prop:m-integral-coprime} 
Let $\OO$ be an order of a number field $K$, with $\ff(\OO)$ its absolute conductor ideal.
Suppose the integral ideal $\mm$ has $\mm \subseteq \ff(\OO)$.  Then, the following hold.
\begin{enumerate}
\item[(1)]
The monoid $\rI_{\mm}^{\ast}(\OO)$ is a free abelian monoid whose set of generators is the set of prime ideals $\sP(\mm) = \{\pp: \, \mm \not\subseteq \pp, \, \pp \neq \{0\}\}$.
All the prime ideals in $\sP(\mm)$ are non-singular, and all ideals $\aa \in \rI_{\mm}^{\ast}(\OO)$ have unique prime factorization into nonnegative powers of prime ideals in $\sP(\mm)$. 
\item[(2)]
The monoid $\rI_{\mm}(\OO)$ of nonzero ideals coprime to $\mm$ and the monoid $\rI_{\mm}^{\ast}(\OO)$ of invertible ideals coprime to $\mm$ coincide.
\end{enumerate}
\end{prop}
\begin{proof}
[Proof of (1)] 
We suppose $\mm \subseteq \ff(\OO) = \colonideal{\OO}{\OO_K}$.
All ideals in $\rI_{\mm}(\OO)$ are coprime to $\mm$, hence they are coprime to $\ff(\OO)$.
But the prime ideals not coprime to $\ff(\OO)$ are exactly the singular prime ideals by \Cref{lem:singular}(3).
Thus $\sP(\mm)$ contains only non-singular prime ideals.
Since coprimality to $\mm$ is preserved by ideal product, all products of a finite number of ideals drawn from $\sP(\mm)$ (allowing multiplicity) belong to $\rI_{\mm}(\OO)$. This exhausts the allowed primary decompositions of elements of $\rI(\OO)$, hence $\rI(\OO)$ is a free abelian monoid.
\end{proof}
\begin{proof}
[Proof of (2)]
The result follows (1) because all non-singular prime ideals of $\OO$ are invertible, by \Cref{lem:singular}.
\end{proof}

\begin{rmk}\label{rmk:212N}
The conclusions of \Cref{prop:m-integral-coprime} need not hold for non-maximal orders when $\mm \not\subseteq \ff(\OO)$. The following examples take $\mm=\OO$.
\begin{enumerate}
\item[(1)] 
For some non-maximal orders $\OO$, the  monoid of all invertible integral ideals $\rI^{\ast}(\OO) = \rI_{\OO}^{\ast}(\OO)$ is not a  free abelian monoid, because its irreducible elements may satisfy nontrivial relations.
See \Cref{exam:214}, \eqref{eq:cover}, where $\qq_4$ and $\qq_4^{'}$ are distinct  irreducible ideals, but $(\qq_4)^2= (\qq_4^{'})^2$.
\item[(2)]
For any non-maximal order, the monoid $\rI(\OO) = \rI_{\OO}(\OO)$ of all integral ideals contains noninvertible ideals.
The conductor ideal has $\ord(\ff(\OO)) = \OO_K \ne \OO$, so is not invertible by \Cref{lem:invertible}(2). 
\end{enumerate}
\end{rmk} 

%
%
\section{Fractional ideals of orders}\label{sec:2a} 

This section treats fractional ideals 
of Noetherian integral domains of dimension one, 
and in particular,
orders of number fields. 

The fractional ideal theory of non-maximal orders of number fields has some key differences from that of the maximal order.
Not all fractional ideals are invertible, and the set of fractional ideals under ideal multiplication has the structure of a monoid (a semigroup with identity) that is not a group. The group of invertible fractional ideals need not be free; it may contain (a finite number of) torsion elements.

\subsection{Fractional ideals for integral domains}\label{subsec:24A} 

Following \cite[Defn.\ 1.1.6]{DTZ:62}, define fractional ideals for Noetherian integral domains.

\begin{defn}[Fractional ideal]
A \textit{fractional ideal} $\aa$ of a Noetherian integral domain $\dD$ is a $\dD$-submodule of its fraction field $K$ with the property that $\lambda\aa \subseteq \dD$ for some $\lambda \in K^\times$.
(A \textit{integral ideal} is a fractional ideal contained in $\dD$.)
\end{defn}

The set of fractional ideals of $\dD$ has sum, product, intersection, and quotient operations.
In particular, the \textit{ideal product} operation $\cdot$ is defined as in \eqref{eq:idealproduct}.
The \textit{ideal quotient} (or \textit{colon ideal}) operation $\colonideal{}{}$ is 
\begin{equation}\label{eqn:ideal-quotient1} 
\colonideal{\aa}{\bb} := \{ x \in K : x \bb \subseteq \aa\}.
\end{equation}
As for integral ideals, we define the \textit{multiplier ring} of a fractional ideal $\aa$ as $\ord(\aa) = \colonideal{\aa}{\aa}$.

\begin{prop}\label{prop:fractional-ops} 
The set $\rJ(\dD)$ of all fractional ideals of a Noetherian integral domain $\dD$ is closed under the four operations
$+$, $\cdot$, $\cap$, and $\colonideal{}{}$ (addition, multiplication, intersection, and ideal quotient.) 
\end{prop}
\begin{proof} 
This is \cite[Prop.\ 1.1.11]{DTZ:62}.
\end{proof} 

The set of all fractional ideals of $\rJ(\dD)$ forms a monoid under the ideal product operation, with identity element $\dD$.

We extend the notion of coprimality to that between a fractional ideal and an integral ideal; we do not define coprimality between two non-integral fractional ideals of $\dD$. 

\begin{defn}[Coprimality of a fractional ideal and an integral ideal]\label{def:frac-coprime}
A fractional $\dD$-ideal $\dd$ is \textit{coprime} to an integral ideal $\cc \subseteq \dD$ if it may be written as a quotient $\dd = \colonideal{\aa}{\bb}$, where $\aa$ is an $\dD$-integral ideal coprime to $\cc$, and $\bb$ is an invertible $\dD$-integral ideal  coprime to $\cc$.
\end{defn} 

In this definition, one cannot always choose $\aa, \bb$ to be coprime to each other, 
even when $\dD=\OO$ is an order of a number field; see \Cref{exam:torsion}.
Nevertheless one may check that, if $\dd_1, \dd_2$ are coprime to $\cc$, then so is their ideal product $\dd_1 \dd_2$. 

\begin{defn}[Invertible fractional ideal] \label{defn:invertible}
A fractional $\dD$-ideal $\aa$ is \textit{invertible} for $\dD$ if there is another fractional ideal $\bb$ of $\dD$ such that $\aa\bb = \dD$. 
Otherwise it is {\em non-invertible} for $\dD$. 
\end{defn}

We let $\rJ^{\ast}(\dD)$ denote the set of all invertible fractional ideals of $\dD$; they form a group under multiplication
in which $\dD$ is the identity element.  
There is a converse invertibility result for fractional ideals that extends \Cref{lem:invertible} for integral ideals.

\begin{lem}\label{lem:fract-invertible}
Let $\dD$ be a Noetherian integral domain.
\begin{itemize}
\item[(1)] If $\cc$ is an invertible fractional $\dD$-ideal, and $\cc= \aa\bb$ with $\aa, \bb$ fractional $\dD$-ideals,  then  $\aa$ and $\bb$ are 
both invertible fractional  $\dD$-ideals.
\item[(2)] If $\aa,  \bb, \cc$ are all $\dD$-fractional ideals  (invertible or not) and  $\ord(\cc)= \dD$, then $\cc=\aa\bb$ implies
$\ord(\aa) = \ord(\bb) = \dD$.
In particular all invertible fractional $\dD$-ideals  $\aa$ have $\ord(\aa)= \dD$. 
\end{itemize}
\end{lem} 
\begin{proof}
The proof is identical to the proof of \Cref{lem:invertible}, replacing $\OO$ by $\dD$ and ``integral'' by ``fractional.'' 
\end{proof}
 
The next two lemmas show that the definitions of invertibility and coprimality for fractional ideals are consistent with 
the earlier definitions for integral ideals.

\begin{lem}\label{lem:coloninv}
Let $\dD$ be a Noetherian integral domain, and let $\aa$ and $\bb$ be integral $\dD$-ideals. If $\bb$ is invertible as a fractional ideal of $\dD$, then the fractional ideal $\colonideal{\aa}{\bb} = \aa\bb^{-1}$.
\end{lem}
\begin{proof}
We first show $\colonideal{\aa}{\bb} \subseteq \aa\bb^{-1}$, or equivalently (since $\bb$ is invertible) that $\colonideal{\aa}{\bb}\bb \subseteq \aa$. Let $c \in \colonideal{\aa}{\bb}$ and $b \in \bb$.
Thus $c\bb \subseteq \aa$ by the definition of $\colonideal{\aa}{\bb}$, so in particular, $cb \in \aa$.

We show the reverse inclusion $\colonideal{\aa}{\bb} \supseteq \aa\bb^{-1}$. 
Let $a \in \aa$ and $d \in \bb^{-1}$. Then, $d\bb \subseteq \dD$, so $ad\bb \subseteq a\dD \subseteq \aa$. 
By definition of the quotient ideal, $ad \in \colonideal{\aa}{\bb}$.
\end{proof}

\begin{lem}\label{lem:invert-frac-int}
Let $\dD$ be a Noetherian integral domain.
\begin{itemize}
\item[(1)] Any invertible fractional ideal $\aa$ of $\dD$ with $\aa \subseteq \dD$ is an invertible integral ideal, and conversely. 
\item[(2)] If two integral ideals $\cc, \dd$ of $\dD$ are coprime as integral ideals, then 
$\dd$, regarded as a fractional ideal, is coprime to $\cc$, and conversely. 
\end{itemize}
\end{lem} 
\begin{proof}[Proof of (1)]
Given an integral ideal $\aa \subseteq \dD$ invertible as a fractional ideal, there is a fractional ideal $\cc$ such that  $\aa \cc = \dD$. There exists $\lambda \in \dD$ so that $ \lambda\cc \subseteq \dD$ is an integral ideal, and $\aa(\lambda \cc) = \lambda\dD$ certifies that $\aa$ is an invertible integral ideal.

Given an invertible integral ideal $\aa$, there is some integral ideal $\bb$ such that 
$\aa \bb = \gamma \dD$ for a nonzero $\gamma\in \dD$. Then $\aa (\gamma^{-1}\bb) = \dD$ certifies $\aa$ is an invertible fractional ideal.
\end{proof}
\begin{proof}[Proof of (2)]
Suppose $\cc, \dd$ are integral ideals satisfying the coprimality condition $\cc + \dd= \dD$.
Viewing $\dd$ as a fractional ideal, the  decomposition $\dd = \colonideal{\dd}{\dD}$ certifies the fractional ideal $\dd$ is 
coprime to $\cc$, since $\dd$ (on the right side of the decomposition) is by hypothesis coprime to $\cc$ as an integral ideal, and $\dD$ is an invertible integral ideal coprime to $\cc$. 

Conversely, we are given integral ideals $\cc, \dd$ for which $\dd$, regarded as a fractional ideal, is coprime to $\cc$. 
Then we have $\dd=\colonideal{\aa}{\bb}$ for two integral ideals $\aa, \bb$ in which $\bb$ is invertible
and $\cc+\aa = \cc+\bb = \dD$.
By \Cref{lem:coloninv}, $\dd= \aa\bb^{-1}$.
Thus, as fractional ideals, we have $\dd\bb = \aa$, and $\aa \subseteq \dd$, since $\bb$ is an integral ideal.
Now $\cc + \aa = \dD$ means  there  exists $c \in \cc$ and $a \in \aa$ such that $c+a=1$, and necessarily  $a \in \dd$, 
so $\cc+\dd=\dD$.
\end{proof}

Invertible fractional ideals are characterized by local data. 

\begin{prop}\label{prop:invertible-ops3} 
If $\dD$ is any Noetherian integral domain, then a fractional ideal $\aa \in \rJ(\dD)$ is invertible for $\dD$ if and only if
each localized ideal $\aa_{\mm}$ is a principal $\dD_{\mm}$-ideal for each maximal ideal $\mm$ of $\dD$.
\end{prop} 
\begin{proof} 
This is a special case of of \cite[Cor.~2.1.7]{DTZ:62}.
\end{proof} 

The next lemma collects basic properties of ideal quotients of fractional ideals and deduces
the inclusions in (5), which justify \eqref{eq:112}, taking $\dD=\OO$ and $\dD'=\OO'$.

\begin{lem}\label{lem:basicinclusions}
Let $\aa, \bb, \cc$ be fractional ideals in a Noetherian integral domain $\dD$.
\begin{itemize}
\item[(1)] $\colonideal{\aa}{\bb} \subseteq \colonideal{\aa\cc}{\bb\cc}$, with equality if $\cc$ is invertible.
\item[(2)] $\colonideal{\aa}{\bb}\cc \subseteq \colonideal{\aa\cc}{\bb}$, with equality if $\cc$ is invertible.
\item[(3)] $\colonideal{\colonideal{\aa}{\bb}}{\cc} = \colonideal{\aa}{\bb\cc}$.
\item[(4)] If $\aa \subseteq \bb \subseteq \cc$, then $\colonideal{\aa}{\cc} \subseteq \colonideal{\aa}{\bb}$ and $\colonideal{\aa}{\cc} \subseteq \colonideal{\bb}{\cc}$. 
\end{itemize}
Now let $\mm$ be an integral ideal of $\dD$, and suppose $\dD \subseteq \dD'$ with $\dD'$ a Noetherian integral domain
that is also a fractional ideal of $\dD$.  Set $\ff_{\dD'}(\dD) = \colonideal{\dD}{\dD'}$. Then
\begin{itemize}
\item[(5)] $\ff_{\dD'}(\dD)\mm \subseteq \colonideal{\mm}{\dD'} \subseteq \ff_{\dD'}(\dD) \cap \mm\dD'$.
\end{itemize}
\end{lem}

\begin{proof}
Properties (1), (2), (3) and (4) follow directly from the definitions of the quotient and product ideals. To prove the left-hand inclusion of property (5), observe using (2) that
\begin{equation}
\ff_{\dD'}(\dD)\mm = \colonideal{\dD}{\dD'}\mm \subseteq \colonideal{\dD\mm}{\dD'} = \colonideal{\mm}{\dD'}.
\end{equation}
To prove the right-hand inclusion, use (1) and (4):
\begin{equation}
\colonideal{\mm}{\dD'} \subseteq \colonideal{\mm\dD'}{\dD'\dD'} = \colonideal{\mm\dD'}{\dD'} = \mm\dD',
\end{equation}
and $\colonideal{\mm}{\dD'} \subseteq \colonideal{\dD}{\dD'} = \ff_{\dD'}(\dD)$.
\end{proof}

%
%
\subsection{Fractional ideals for orders of number fields}\label{subsec:fractionalideals} 

We specialize to the case of an order $\OO$ of a number field $K$.

For the maximal order $\OO_K$ (and more generally for Dedekind domains),
the fractional ideal group $\rJ^{\ast}(\OO_K)$  is a free abelian group. For non-maximal orders $\OO$,
$\rJ^{\ast}(\OO)$ may contain a (finite) nontrivial torsion subgroup; see \Cref{exam:torsion}. 

Concerning coprimality, each nonzero fractional ideal $\dd$ of an order $\OO$ is coprime to all but a finite set of nonzero prime ideals (i.e., maximal ideals) of $\OO$. 

\begin{eg}[Non-integral fractional ideals having an ideal power that is an integral ideal]
\label{exam:224} 
Consider the non-maximal order $\OO=\Z[2i]$ of the Gaussian field $K=\Q(i)$, whose maximal order is $\OO_K= \Z[i]$. 
\begin{itemize}
\item[(1)] The non-integral fractional ideal 
\begin{equation}
\aa := (1+i)\OO = 4\Z + (1+i)\Z
\end{equation}
is a principal fractional ideal; hence, it is an invertible fractional ideal. It has ideal square $\aa^2 = 2i\OO = \qq_4'$ (in the notation of \Cref{exam:214}), which is an integral ideal, shown to be irreducible in \Cref{exam:214}.

Thus the irreducible integral ideal $\qq_4'$ 
becomes a perfect square viewed in the group $\rJ^{\ast}(\OO)$ of
invertible fractional ideals.
In contrast, the irreducible integral ideal  $\qq_{4}= 2\OO$ is not a perfect
square in $\rJ^{\ast}(\OO)$. 

\item[(2)] The non-integral fractional ideal
\begin{equation}
\bb := (1+i)\OO_K = 2\Z + (1+i)\Z
\end{equation}
is a non-invertible fractional $\OO$-ideal. It has ideal square $\bb^2= 2\OO_K=\QQ_2$ (in the notation of \Cref{exam:214}),
which is a non-invertible irreducible integral $\OO$-ideal.
The integral ideal $\QQ_2$ is the square of an element $\bb$ of
the monoid $\rJ(\OO)$ of fractional ideals.
\end{itemize}
\end{eg} 

\begin{eg}[A finite order element of $\rJ^{\ast}(\OO)$]\label{exam:torsion} 
Consider the order  $\OO = \Z + 2i \Z$, which is a suborder of $\OO_K= \Z+ i \Z$ in $K= \Q(i)$. Its conductor is $\ff(\OO) = 2 \OO_K = 2\Z + 2i\Z$.

The $\OO$-fractional ideal $\dd:= i \OO = 2\Z + i\Z$,
is not an integral $\OO$-ideal since $i \not\in \OO$. The  ideal $\dd$ is a torsion element of $\rJ^{\ast}(\OO)$, with $\dd^2= \OO$.
It is an invertible fractional ideal, and  it is another example of a (non-integral) fractional ideal that has  an ideal power that is an integral ideal.
\end{eg}
 
\begin{eg}[An element of $\rJ^{\ast}(\OO)$ not a quotient of two coprime integral ideals]\label{exam:non-coprime-invertible fractional} 
We continue to consider the  $\OO$-fractional ideal $\dd=i\OO$ for the non-maximal order $\OO=  \Z + 2i \Z$ of  $K= \Q(i)$.
It is expressible  as a ratio $i\OO= \aa \bb^{-1}$
of two invertible integral $\OO$-ideals, taking $\aa= 2 \OO$ and $\bb= 2i \OO$.
But $\aa+ \bb =2 \OO_K$, so $\aa, \bb$ are not coprime integral ideals of $\OO$.

In fact, $i\OO$ cannot be expressed as a ratio of two coprime integral $\OO$-ideals, which we now show by contrapositive. Suppose $i\OO = \aa\bb^{-1}$ for integral $\OO$-ideals $\aa$ and $\bb$, and multiply both sides by $\bb$ to obtain $i\bb = \aa$. Thus, $\aa = i\bb \subseteq i\OO$, and $\aa \subseteq \OO$, so $\aa \subseteq \OO \cap i\OO = (\Z + 2i\Z) \cap (2\Z + i\Z) = 2\OO_K$. Similarly, $\bb = i\aa \subseteq i\OO$ and $\bb \subseteq \OO$, so $\bb \subseteq 2\OO_K$. Thus, $\aa + \bb \subseteq 2\OO_K$, so $\aa$ and $\bb$ cannot be coprime.
\end{eg}

\subsection{Monoids of fractional ideals of orders coprime to a modulus $\mm$}\label{subsec:33a} 

Recall that $\rJ(\OO)$ is the monoid of fractional ideals of the order $\OO$, and $\rJ^{\ast}(\OO)$ is the group of invertible fractional ideals of $\OO$.
We have introduced a notion of $\mm$-coprimality for fractional ideals in \Cref{def:frac-coprime}.
We now define monoids of fractional ideals and groups of invertible fractional ideals coprime to a modulus $\mm$ which is an integral ideal. 

\begin{defn}[Monoids of fractional ideals coprime to $\mm$]\label{def:frac-ideal-monoids} 
Given an integral ideal $\mm$ of an order $\OO$ of a number field, let
$\rJ_{\mm}(\OO)$ denote the set of {\em $\mm$-coprime fractional ideals} of $\OO$,
\begin{equation}
\rJ_{\mm}(\OO) = \{ \aa \in \rJ(\OO) : \aa \mbox{ is coprime to } \mm \}. 
\end{equation}
Let  $\rJ_{\mm}^{\ast}(\OO)$ denote the set of {\em $\mm$-coprime invertible fractional ideals},
\begin{equation}
\rJ_{\mm}^{\ast}(\OO) = \{ \aa \in \rJ^{\ast}(\OO) : \aa \mbox{ is coprime to } \mm \}.
\end{equation}
\end{defn} 

\begin{lem}\label{lem:mm-fract-monoids} 
Let $\OO$ be an order of a number field and $\mm$ an integral $\OO$-ideal.
\begin{itemize}
\item[(1)]
The set $\rJ_{\mm}(\OO)$ of $\mm$-coprime fractional ideals is a monoid under $\OO$-fractional ideal product.
\item[(2)] 
The set $\rJ_{\mm}^{\ast}(\OO)$ of $\mm$-coprime invertible fractional ideals is a group under $\OO$-fractional ideal product.
\item[(3)]
If $\mm, \mm'$ are ideals of $\OO$  with $\mm \subseteq \mm'$, then $\rJ_{\mm}(\OO) \subseteq \rJ_{\mm'}(\OO)$, and $\rJ_{\mm}^{\ast}(\OO) \subseteq \rJ_{\mm'}^{\ast}(\OO)$.
\end{itemize}
\end{lem}
\begin{proof}[Proof of (1) and (2)]
Let $\cc_1, \cc_2 \in \rJ_{\mm}(\OO)$; we check that $\cc_1\cc_2 \in \rJ_{\mm}(\OO)$. Write $\cc_i = \aa_i\bb_i^{-1}$ for integral ideals $\aa_i, \bb_i$ coprime to $\mm$. Then $\cc_1\cc_2 = (\aa_1\aa_2)(\bb_1\bb_2)^{-1}$. The ideals $\aa_1\aa_2$ and $\bb_1\bb_2$ are coprime to $\mm$ by \Cref{lem:monoid-well-defined}(1), so $\cc_1\cc_2 \in \rJ_{\mm}(\OO)$. Moreover, if $\cc_1$ and $\cc_2$ are invertible (that is, in $\rJ_{\mm}^\ast(\OO)$), then so is $\cc_1\cc_2$. Both $\rJ_{\mm}(\OO)$ and $\rJ_{\mm}^\ast(\OO)$ have identity element $\OO$, and $\rJ_{\mm}^\ast(\OO)$ has inverses $(\aa\bb^{-1})^{-1} = \bb\aa^{-1}$.
\end{proof}
\begin{proof}[Proof of (3)]
The result is inherited from the inclusions in \Cref{lem:monoid-well-defined}(2) applied to integral ideals $\aa,\bb$ defining an element $\cc = \aa\bb^{-1} \in \rJ_\mm(\OO)$.
\end{proof}
The next result shows that for all ideals $\mm$ contained in the conductor ideal, the
associated monoid of integral ideals coprime to $\mm$ have unique factorization
into prime ideals.

\begin{prop}\label{prop:m-frac-coprime} 
Let $\OO$ be an order of a number field $K$. Suppose the integral ideal $\mm$
has $\mm \subseteq \ff(\OO)$.  Then the following hold.

\begin{itemize}
\item[(1)] 
The monoid $\rJ_{\mm}^{\ast}(\OO)$ of invertible fractional ideals is a free abelian group whose generators are the set of prime ideals $\sP(\mm) = \{\pp : \mm \not\subseteq \pp, \pp \neq 0\}$. All of these prime ideals are non-singular. Ideals $\aa \in \rJ_{\mm}^{\ast}(\OO)$ have unique factorization into integral powers of the prime ideals in $\sP(\mm)$. 
\item[(2)]
The monoid $\rJ_{\mm}(\OO)$ of fractional ideals coprime to $\mm$ under ideal multiplication and the monoid $\rJ_{\mm}^{\ast}(\OO)$ of invertible fractional ideals coprime to $\mm$ coincide.
\end{itemize}
\end{prop}
\begin{proof}[Proof of (1)]
This result follows from the unique factorization property of \Cref{prop:m-integral-coprime}(1) using the fact that all $\pp \in \Pj(\mm)$ are invertible integral ideals, as shown in \Cref{prop:m-integral-coprime}(2), hence invertible fractional ideals.
\end{proof}
\begin{proof}[Proof of (2)]
This follows from the fact that all non-singular prime $\OO$-ideals are invertible.
\end{proof}

\begin{rmk}\label{rek:37}
The conclusions of \Cref{prop:m-frac-coprime} for fractional ideals do not generally hold for $\mm \not\subseteq \ff(\OO)$, e.g., for $\mm=\OO$ and $\OO$ non-maximal. 
\begin{enumerate}
\item[(1)] For some non-maximal orders $\OO$, choosing $\mm=\OO$, the monoid of  all invertible fractional ideals $\rJ^{\ast}(\OO) = \rJ_{\OO}^{\ast}(\OO)$ is not a  free abelian monoid, because it contains nontrivial torsion elements. See \Cref{exam:torsion}. 
\item[(2)] For any non-maximal order $\OO$, the monoid $\rJ(\OO) = \rJ_{\OO}(\OO)$ of all fractional ideals contains non-invertible ideals. The conductor ideal $\ff(\OO)$ is a noninvertible fractional ideal, because its non-invertibility as an integral ideal (by \Cref{lem:singular}) implies the same as a fractional ideal by \Cref{lem:invert-frac-int}(1).
\end{enumerate}
\end{rmk}

%
%
\section{Change of orders in a number field:~extension and contraction of ideals}\label{sec:CE} 

For two orders $\OO \subseteq \OO'$ having the same quotient field $K$, we consider the effect of natural maps sending (integral and fractional) $\OO$-ideals to $\OO'$-ideals, and vice versa.

\begin{defn}\label{def:EC}
The inclusion map $\OO \inj \OO'$ defines extension and contraction maps on integral ideals. 
\begin{enumerate}
\item[(1)]
If $\aa$ is an integral ideal of $\OO$, then the \textit{extension} $\ext(\aa) := \aa\OO'$ is the integral $\OO'$-ideal generated by the elements of $\aa$. 
\item[(2)]
If $\aa'$ is an integral ideal of $\OO'$, then the \textit{contraction} $\con(\aa') := \aa' \cap \OO$ is the integral $\OO$-ideal of elements of $\aa'$ also in $\OO$.
\end{enumerate} 
\end{defn}

This is a special case of extension and contraction of ideals under ring homomorphism \cite[p.\ 9]{AM:69}.
In \Cref{subsec:31}, we study the effect of $\ext$ and $\con$ on general integral ideals.
In \Cref{subsec:32}, we show these maps are bijective monoid homomorphisms when restricted to submonoids of integral ideals coprime to the relative conductor $f_{\OO'}(\OO)$. 
In \Cref{subsec:33}, we extend $\ext$ and $\con$ to bijective homomorphisms on groups of invertible fractional ideals 
coprime to $f_{\OO'}(\OO)$.

%
%
\subsection{Extension and contraction of general integral ideals}\label{subsec:31} 

Given the inclusion of two orders $\iota : \OO \to \OO'$ 
of a fixed algebraic number field $K$, 
we have well-defined functions $\ext: \rI(\OO) \to \rI(\OO')$ and
$\con: \rI(\OO') \to \rI(\OO)$. 

The extension map $\ext: \rI(\OO) \to \rI(\OO')$ is a monoid homomorphism for ideal multiplication.
It is easy to check that this
map preserves invertibility of ideals, and it preserves the property of being a principal ideal.
The map $\ext$ may not be surjective; see \Cref{exam:non-maximal}.

The contraction map $\con: \rI(\OO') \to \rI(\OO)$ in general is not a monoid homomorphism. 
One always has $\con(\aa') \con(\bb') \subseteq \con(\aa' \bb')$, but strict inclusion may sometimes hold; see \Cref{exam:con-non-hom}.
Contraction need not preserve invertibility of ideals nor the property of being a principal ideal; see \Cref{exam:29cont}.
We will show in \Cref{subsec:32} that $\con$ is a monoid homomorphism (and thus preserves invertibility) when restricted to $\I_\ff(\OO')$, where $\ff = \ff_{\OO'}\!(\OO)$.
 
For later use, we study the effect of $\ext$ and $\con$ on maximal ideals. 
 
\begin{lem}\label{lem:conextmax}
Let $\OO \subseteq \OO'$ be orders of a number field $K$, and let $\con$ and $\ext$ be the contraction and extension maps
on integral ideals.
\begin{itemize}
\item[(1)] If $\pp$ is a maximal ideal of $\OO$, then $\con(\ext(\pp)) = \pp$.
\item[(2)] If $\pp'$ is a maximal ideal of $\OO'$, then $\pp= \con(\pp')$ is a maximal ideal of $\OO$.
\item[(3)] If $\pp$ is a maximal ideal of $\OO$, then $\pp = \con(\pp')$ for some maximal ideal $\pp'$ of $\OO'$.
\item[(4)] If $\pp'$ is a maximal ideal of $\OO'$, then $\ext(\con(\pp')) \subseteq \pp'$, and strict inequality 
may occur. 
\end{itemize}
\end{lem}

To prove \Cref{lem:conextmax}, we will need a standard result in commutative algebra. 

\begin{lem}\label{lem:idealclosure}
Let $A \subseteq B$ and $C$ be commutative rings with unity such that $C$ 
is the integral closure of $A$ in $B$. Let $\aa$ be an ideal of $A$, and let $\ext(\aa)$ the extension of $\aa$ to $C$. 
Let $\overline{\aa}$ be the integral closure of $\aa$ in $B$, that is,
\begin{equation}
\overline{\aa} = \{\alpha \in C : f(\alpha) = 0 \mbox{ for a monic polynomial } f(x) \mbox{ with all coefficients in } \aa\}.
\end{equation}
Then, $\overline{\aa} = \rad(\ext(\aa))$,
so in particular $\overline{\aa}$ is a $B$-ideal.
\end{lem}
\begin{proof}
This is \cite[Lem.\ 5.14]{AM:69}.
\end{proof}
\begin{proof}[Proof of \Cref{lem:conextmax}(1)] 
The inclusion $\aa \subseteq \con(\ext(\aa))$ holds for all $\aa \in \rI(\OO)$ \cite[Prop.\ I.17]{AM:69}.
Since $\pp$ is maximal, either $\con(\ext(\pp)) = \OO$ or $\con(\ext(\pp))=\pp$.
Suppose for a contradiction that $\con(\ext(\pp)) = \OO$. Then, $\ext(\pp)$ contains $\OO$, so it contains $\OO' \OO= \OO'$, so 
$\ext(\pp) = \OO'$. It follows that the extension to the maximal order $\ext_{\OO_K}\!(\pp) = \OO_K$.
But $\OO_K$ is the integral closure of $\OO$ in $K$, so by \Cref{lem:idealclosure}, $\overline{\pp} = \rad(\OO_K) = \OO_K$. 
Thus, $1 \in \overline{\pp}$, so $f(1) = 0$ 
for some monic polynomial $f(x) = x^n + \mu_{n-1}x^{n-1} + \cdots + \mu_0$ with 
$\mu_j \in \pp$. So $1 = -\left(\mu_{n-1} + \cdots + \mu_0\right) \in \pp$, 
which means $\pp=\OO$, contradicting the hypothesis that $\pp$ is a maximal ideal of $\OO$. It follows that $\con(\ext(\pp))=\pp$.
\end{proof}
\begin{proof}[Proof of \Cref{lem:conextmax}(2)] 
Consider any $a \in \OO$ such that $a \nin \con(\pp')$. Then, $\pp'+a\OO'$ is an ideal of $\OO'$ 
strictly containing $\pp'$, so $\pp'+a\OO' = \OO'$. Thus, $a$ is invertible in $\OO'/\pp'$, and $\OO'/\pp'$ is finite 
(indeed, a finite field), so there is some $n \in \N$ such that $a^n \equiv 1 \Mod{\pp'}$. That is, $a^n = 1 + p'$ for some $p' \in \pp'$. However, $p' = a^n-1 \in \OO$, 
so in fact $a^n \equiv 1 \Mod{\con(\pp')}$. Thus, $\con(\pp')+a\OO = \OO$. 
Since this holds for any $a \in \OO \setminus \con(\pp')$, we have shown that $\con(\pp')$ is a maximal ideal.
\end{proof}
\begin{proof}[Proof of \Cref{lem:conextmax}(3)] 
By (1) we have $\con(\ext(\pp)) = \pp$. Let $\pp'$ be any maximal ideal containing $\ext(\pp)$. 
By maximality of $\pp$, either $\con(\pp')=\pp$ or $\con(\pp')=\OO$.
The latter case implies $1 \in \pp'$, a contradiction.
\end{proof}
\begin{proof}[Proof of \Cref{lem:conextmax}(4)] 
The inclusion $\ext(\con(\aa)) \subseteq \aa$ holds for general ideals in $\rI(\OO')$ \cite[Prop.\ I.17]{AM:69}.
An example of strict inclusion 
is given in \Cref{exam:non-maximal} below. 
\end{proof}

\begin{eg}\label{exam:non-maximal} 
[For a maximal $\OO$-ideal $\pp$, $\ext(\pp)$ need not be a maximal $\OO'$-ideal.]
Consider the orders $\OO = \Z + 5^2\sqrt{2}\Z$ and $\OO' = \Z + 5 \sqrt{2}\Z$ in $K=\Q(\sqrt{2})$,
with $\OO \subseteq \OO'$. Both these orders 
are strictly smaller than the maximal order $\OO_K = \Z + \sqrt{2}\Z$.
Consider $\ext: \rI(\OO) \to \rI(\OO')$ and $\con: \rI(\OO') \to \rI(\OO).$
Set $\pp = 5\Z + 5^2 \sqrt{2}\Z$. By inspection
$\pp$ is an integral ideal of $\OO$, and since it is of
prime index $5$ in $\OO$, it is a maximal ideal of $\OO$. By \Cref{lem:idealclosure}(3) there
is a maximal ideal $\pp'$ of $\OO'$ such that $\con(\pp') = \pp.$ We may in fact choose $\pp'=5\Z+5\sqrt{2}\Z$.

On the other hand, $\pp=5\OO'$ is by inspection a principal $\OO'$-ideal.
It follows that $\ext(\pp)=\pp$. 
Now $\pp \subsetneq \pp' \subsetneq \OO'$, so 
$\pp= \ext(\con(\pp'))$ is not maximal as an $\OO'$-ideal. 

The maximal ideal $\pp'$ of $\OO'$ is not in the image of the map $\ext$:~Suppose it were, so $\pp'= \ext(\aa)$ for some $\aa \in \rI(\OO)$. Then, $\ext (\con(\ext(\aa))) = \ext(\aa)$ by \cite[Prop.\ I.17]{AM:69}. But 
\begin{equation}
\ext(\con(\ext(\aa))) = \ext(\con(\pp')) = \ext(\pp) = \pp,
\end{equation}
so we would obtain $\pp = \ext(\aa) = \pp'$, which is false.
\end{eg} 

\begin{eg}\label{exam:con-non-hom} 
[The contraction map between ideal monoids  $\rI(\OO')$ and $\rI(\OO)$ need not be a monoid homomorphism.]
As in \Cref{exam:non-maximal}, consider the orders $\OO = \Z + 5^2\sqrt{2}\Z$ 
and $\OO' = \Z + 5 \sqrt{2}\Z$ in $K=\Q(\sqrt{2})$,
with $\OO \subseteq \OO'$. We consider $\con: \rI(\OO') \to \rI(\OO)$. 
Also, as in \Cref{exam:non-maximal}, let $\pp = 5\Z+5^2\sqrt{2}\Z$ and $\pp'=5\Z+5\sqrt{2}\Z$; we have $\con(\pp') = \pp$.

Take $\aa'=\bb'= \pp'$. As an $\OO'$-ideal,
\begin{equation}
\aa'\bb' = (\pp')^2= \left(5\Z+ 5 \sqrt{2}\Z\right)^2= 5^2 \Z + 5^2 \sqrt{2}\Z.
\end{equation}
Now $\aa'\bb'=(\pp')^2$ is also an $\OO$-ideal, so that 
\begin{equation}
\con(\aa' \bb') = \con((\pp')^2) = (\pp')^2= 5^2 \Z + 5^2 \sqrt{2}\Z. 
\end{equation}
Since $\con(\pp') =\pp= (5 \Z+ 5^2 \sqrt{2} \Z)$, we have 
\begin{equation}
\con(\aa') \con (\bb') = \left(\con (\pp)\right)^2 = (\pp^2) = 5^2 \Z + 5^3 \sqrt{2} \Z.
\end{equation}
We have shown $\con(\aa') \con(\bb') \subsetneq \con(\aa' \bb')$.
\end{eg} 

\begin{eg}\label{exam:29cont} 
[Contraction and extension of ideals in non-maximal orders]
Consider $K= \Q(\sqrt{-13})$ with maximal order $\OO_K = \Z + \sqrt{-13}\Z$,
and consider a non-maximal order $\OO = \Z + q\sqrt{-13}\Z$
where $q$ is an odd inert prime (for example, $q=5$). 
Recall from \Cref{exam:29} that $\mm = q\OO_K = q\Z + q\sqrt{-13}\Z$ is a maximal $\OO_K$-ideal of norm $q^2$, and $\mm$ is principal in $\OO_K$, hence invertible.

Now $\mm$ is also an $\OO$-ideal and is a maximal ideal for $\OO$. It is
the conductor ideal for $\OO$, so it is non-invertible as an $\OO$-ideal. 
It is immediate that $\ext(\mm) = \mm$.
For the contraction map $\con : \OO_K \to \OO$, we have 
\begin{equation}
\con(\mm) = \mm \cap \OO = \mm.
\end{equation}
This example verifies $\con(\ext(\mm)) =\mm$, and also $\ext(\con(\mm))=\mm$, preserving
the maximal ideal property.
However, the contracted ideal $\mm$ is not an invertible fractional $\OO$-ideal hence also not a principal $\OO$-ideal.

Furthermore, consider $\qq := q \OO = q\Z + q^2\sqrt{-13}\Z$.
It is a principal $\OO$-ideal, so it is an invertible $\OO$-ideal; hence, it cannot be an $\OO_K$-ideal.
We have $\qq \subsetneq \mm$. 
It is easy to see that 
$\ext(\qq) = \mm$.
So we now have $\con(\ext(\qq)) = \con (\mm) = \mm$ and $\qq \subsetneq \con(\ext(\qq))=\mm$.
(Thus $\con(\ext(\aa))$ can be a maximal ideal while $\aa$ is not maximal.)
\end{eg}

\subsection{Extension and contraction of integral ideals coprime to the relative conductor}\label{subsec:32} 

Extension and contraction operations behave well when restricted to integral ideals coprime to
the relative conductor ideal.
\begin{lem}\label{lem:conext}
On the set of
integral ideals $\rI_{\ff}(\OO)$ coprime to the (relative) conductor 
$\ff = \ff_{\OO'}(\OO) \in \rI(\OO) \cap \rI(\OO')$, contraction $\con : \rI_{\ff}(\OO') \to \rI_{\ff}(\OO)$ 
defines an isomorphism of monoids, with inverse the extension map $\ext : \rI_{\ff}(\OO) \to \rI_{\ff}(\OO')$.
\end{lem}
\begin{proof}
We first show that $\con$ and $\ext$ are bijections inverse to each other. For general ring maps,
it is easily seen that $\ext(\con(\aa')) \subseteq \aa'$ and $\aa \subseteq \con(\ext(\aa))$ \cite[Prop.\ I.17]{AM:69}. 
The reverse inclusions are not true in general, even in our case of the inclusion map of a suborder in an order. 
We must use coprimality to the conductor.

For the first, consider $\aa' \in \rI_{\ff}(\OO')$, and set $\aa= \con(\aa') = \aa' \cap \OO$.
We will show $\ext(\con(\aa'))= \aa'$.
By coprimality of $\aa'$ to the relative conductor $\ff= \ff_{\OO'}(\OO)$ we have
\begin{equation}\label{eq:ccc-prime}
1 = a' + c'
\end{equation}
for some $a' \in \aa'$ and $c \in \ff \subseteq \OO$. Then $a' = 1-c\in \OO$, so $a' \in \aa= \con(\aa')$. Then 
\eqref{eq:ccc-prime} certifies that $\aa$ is coprime to $\ff$ in the order $\OO$. Also 
$a' \in \ext(\con(\aa'))$, so 
$\ext(\con(\aa'))$ is coprime to $\ff$. 
We must show $\aa' \subseteq \ext(\con(\aa'))$. We have $\aa' \ff \subseteq \res(\aa)$, because $\aa' \ff \subseteq \ff \subseteq \OO$ and $\aa' \ff \subseteq \aa'\OO' = \aa'$.
Now, given $b' \in \aa'$, we have from \eqref{eq:ccc-prime} that 
\begin{equation}
b' = b'a + b'c'.
\end{equation}
Now $b'a \in \ext(\con(\aa))$ since $b' \in \OO$ and $a \in \con(\aa)$, while $b'c' \in \aa'\ff \subseteq \con(\aa) \subseteq \ext(\con(\aa))$, hence $b'\in \ext(\con(\aa'))$. Thus, $\aa' \subseteq \ext(\con(\aa'))$, so we conclude that $ \ext(\con(\aa'))=\aa'$.

For the second, consider $\aa \in \rI_{\ff}(\OO)$, and set $\aa' = \ext(\aa)$. 
We will show $\con(\ext(\aa))= \aa$. By the coprimality assumption, there exist $a \in \aa$ and $f \in \ff= \ff_{\OO'}(\OO)$ with
$1=a+f$. Since $a \in \aa'=\ext(\aa)$, the ideals $\aa'$ and $\ff$ are coprime in $\OO'$, and in addition $a \in \con(\ext(\aa))$. We must show
$\con(\ext(\aa)) \subseteq \aa$.
Suppose $b \in \con(\ext(\aa)) \subseteq \OO$; then the coprimality equation implies
$b= ba+ bf$. We show $b \in \aa$ by showing both summands of the right hand side are in $\aa$.
Now $b \in \OO$, so $ba \in \OO \aa = \aa$. Now 
\begin{equation}
bf \in \ext(\aa) \ff = (\aa \OO') \ff = \aa (\OO' \ff) = \aa \ff \subseteq \aa \OO= \aa.
\end{equation} 
Thus $\con(\ext(\aa)) \subseteq \aa$, whence $\con(\ext(\aa)) =\aa$. 

Finally, for two ideals $\aa, \bb \in \rI_{\ff}(\OO)$, it follows from the definition of the extension map
that $\ext(\aa\bb) = \ext(\aa)\ext(\bb)$. Because $\con$ defines an inverse to $\ext$, 
it follows that $\con$ is also a homomorphism from $\rI_{\ff}(\OO')$ onto $\rI_{\ff}(\OO)$.
\end{proof}

\subsection{Extension and contraction of fractional ideals coprime to the relative conductor}\label{subsec:33} 

The extension and contraction maps between orders $\OO \subseteq \OO'$ of a number field $K$ consistently extend from integral ideals to fractional ideals, provided that one restricts to fractional ideals coprime to the relative conductor $\ff_{\OO'}(\OO)$.

\begin{prop}\label{prop:conext}
Consider two orders $\OO \subseteq \OO'$ of the number field $K$.
Let $\ff = \ff_{\OO'}(\OO)$ denote the relative conductor.
Let $\mm'$ be an integral ideal of $\OO'$ having $\mm' \subseteq \ff$.
 Then 
the contraction and extension maps 
extend uniquely 
to isomorphisms between
groups of fractional ideals coprime to $\mm'$.
That is, the maps 
\begin{align}
\con &: \rJ_{\mm'}(\OO') \to \rJ_{\mm'}(\OO) \mbox{ and}\\
\ext &: \rJ_{\mm'}(\OO) \to \rJ_{\mm'}(\OO')
\end{align} 
are well-defined and are inverses of each other.
\end{prop}

\begin{rmk}\label{rmk:con-2}
The extension map on fractional ideals $\aa$ coprime to $\ff_{\OO'}(\OO)$ is $\ext(\aa)=\aa\OO'$, as in the case for integral ideals
in \Cref{def:EC}(1). 
However, the contraction map on fractional ideals coprime to $\ff_{\OO'}(\OO)$ requires a new definition
different from \Cref{def:EC}(2) for integral ideals. That is, the contraction map 
on fractional ideals does not always have $\con(\aa') = \aa' \cap \OO$, although the inclusion $ \aa' \cap \OO \subseteq \con(\aa')$ holds.
For example, let $\bb'$ be any proper integral $\OO'$-ideal coprime to $\ff_{\OO'}(\OO)$, and set $\aa'= (\bb')^{-1}$. Then $\OO' \subsetneq \aa'$, so $\aa' \cap \OO= \OO$. However, by \Cref{prop:conext}, we will have $\con(\aa') = \con(\bb')^{-1} = (\bb \cap \OO)^{-1} = \bb^{-1} = \aa'$. Since $\bb'$ is a proper ideal, $\con(\aa') \neq \aa' \cap \OO$.
\end{rmk} 

\begin{proof}[Proof of \Cref{prop:conext}]
Note that $\mm'$ is both an integral $\OO$-ideal and an integral $\OO'$-ideal, the latter by assumption, and the former because $\OO \subseteq \OO'$ (so $\mm'$ is a fractional $\OO$-ideal) and $\mm' \subseteq \ff_{\OO'}(\OO) \subseteq \OO$, hence
 $\mm'$ is integral as an $\OO$-ideal by \Cref{lem:invert-frac-int}(1). 

We first claim that the maps $\con: \rI_{\mm'}(\OO') \to \rI_{\mm'}(\OO) $ and $\ext: \rI_{\mm'}(\OO) \to \rI_{\mm'}(\OO')$ 
send integral ideals in the specified domains into integral ideals in the specified codomains.

To prove the claim for $\con$, suppose $\aa' \in \rI_{\mm'}(\OO')$, so $\aa'+\mm'=\OO'$.
Then there exist $a' \in \aa'$ and $m' \in \mm'$ such that $a'+m'=1$. 
Now $m' \in \mm' \subseteq \ff \subseteq \OO$, so
$a' = 1-m' \in \OO$, showing that $a' \in \con(\aa')$. 
Therefore $\con(\aa')$ is coprime to $\mm'$ in $\OO$.

To prove the claim for $\ext$, suppose $\aa \in \rI_{\mm'}(\OO)$, so $\aa+\mm'=\OO$.
Then, there exist $a \in \aa$ and $m \in \mm'$ with $a+m=1$. Clearly, $a \in \ext(\aa)$. 
Therefore, $\ext(\aa)$ is coprime to $\mm'$ in $\OO'$.

Now \Cref{lem:conext} asserts that $\con(\ext(\aa)) = \aa$ and $\ext(\con(\aa')) = \aa'$ for all integral ideals in their respective domains. 
Because the domains and codomains of the maps above match on integral ideals, 
the isomorphisms given by \Cref{lem:conext}
restricts to bijective isomorphisms
$\con : \rI_{\mm'}(\OO') \to \rI_{\mm'}(\OO)$ and 
$\ext : \rI_{\mm'}(\OO) \to \rI_{\mm'}(\OO')$ 
that are inverses of each other.

We now consider any 
fractional ideal $\dd = \colonideal{\aa}{\bb} = \aa\bb^{-1}\in \rJ_{\mm'}(\OO)$ with $\aa \in \rI_{\mm'}(\OO)$ and $\bb \in \rI_{\mm'}^\ast(\OO)$,
where $\dd = \aa\bb^{-1}$ by \Cref{lem:coloninv}.
We define $\ext(\dd)=\ext(\aa)\ext(\bb)^{-1}$; 
any group homomorphism extending $\ext$ must be defined thus. To show this definition is independent of the choice of expression of $\dd$ as a ratio of integral ideals, 
consider two such expressions $\dd = \aa_1\bb_1^{-1} = \aa_2\bb_2^{-1}$. 
Then, $\aa_1\bb_2 = \aa_2\bb_1$, so $\ext(\aa_1)\ext(\bb_2) = \ext(\aa_2)\ext(\bb_1)$, so $\ext(\aa_1)\ext(\bb_1)^{-1} = \ext(\aa_2)\ext(\bb_2)^{-1}$, whence $\ext(\dd)$ is well-defined.
 By a similar argument, defining $\con\!\left(\aa'(\bb')^{-1}\right)=\con(\aa')\con(\bb')^{-1}$ for $\aa'\in \rI_{\mm'}(\OO')$ and $\bb'\in \rI_{\mm'}^\ast(\OO')$ gives a unique well-defined homomorphism. The fact that $\con$ and $\ext$ are inverses of each other then follows from the same fact for integral ideals.
\end{proof}

\section{Ray class groups of orders}\label{sec:group}

In this section, we define ray class groups of an order $\OO$ as quotients of certain groups of fractional ideals, and we show that those groups can be taken to satisfy auxiliary coprimality conditions to an arbitrary ideal $\dd$.

%
%
\subsection{Definition of ray class groups of orders}\label{subsec:41}

Let $\OO$ be an order of a number field $K$.
\begin{defn}\label{def:principal} 
A fractional ideal $\aa$ of $\OO$ is \textit{principal} if it may be written as $\aa = \alpha\OO$ for some $\alpha \in K$. The group of principal fractional ideals is denoted by $\rP(\OO)$.
\end{defn}

When working with ray class groups with level datum $(\OO; \mm, \rS)$, 
one needs to talk about modular congruences between non-integral elements of $K$. This requires a bit of care,
given in the following definition. 

\begin{defn}\label{defn:RC-congruence} 
Given $\alpha, \beta \in K$, we say $\alpha \equiv_{\OO} \beta \Mod{\mm}$ (abbreviated $\alpha \equiv \beta \Mod{\mm}$ when $\OO$ is known from context) 
if $\alpha = \frac{\alpha_1}{\alpha_2}$ and $\beta = \frac{\beta_1}{\beta_2}$ for some $\alpha_1, \alpha_2, \beta_1, \beta_2 \in \OO$ with $\alpha_2\OO$ and $\beta_2\OO$ coprime to $\mm$, satisfying $\alpha_1\beta_2 - \alpha_2\beta_1 \in \mm$.
This is equivalent to saying that $\alpha-\beta \in \mm\OO[S_\mm^{-1}]$, where $\OO[S_\mm^{-1}]$ is the semilocal ring obtained by inverting the elements of $\OO$ coprime to $\mm$. (See also \Cref{defn:localization}.)
\end{defn}

\begin{defn}[Principal ray ideal group]\label{def:rayclasses} 
Given an integral ideal $\mm$ in $\OO$ and a subset $\rS$ of the real places of $K$ (possibly empty), define the group of 
\textit{principal ray ideals of $\OO$ modulo $(\mm,\rS)$}, denoted $\rP_{\mm,\rS}(\OO)$, by:
\begin{equation}
\rP_{\mm,\rS}(\OO) = \{\mbox{$\alpha\OO$ : $\alpha \in K^\times$ such that $\alpha \Con 1 \Mod{\mm}$ and $\rho(\alpha) > 0$ for all $\rho \in \rS$}\}.
\end{equation}
\end{defn}

\begin{defn}\label{defn:ray-clas-order}
The \textit{ray class group} of $\OO$ modulo $(\mm,\rS)$ is 
\begin{equation}
\Cl_{\mm,\rS}(\OO) = \frac{\rJ_{\mm}^{\ast}(\OO)}{\rP_{\mm,\rS}(\OO)}.
\end{equation}
That is, 
\begin{equation}\label{eq:ray-class-group-order}
\Cl_{\mm,\rS}(\OO) = \frac{\{\mbox{invertible fractional ideals of $\OO$ coprime to $\mm$}\}}{\{\mbox{$\alpha\OO$ : $\alpha \in K^{\times}$ such that $\alpha \Con 1 \Mod{\mm}$ and $\rho(\alpha) > 0$ for all $\rho \in \rS$}\}}.
\end{equation}
\end{defn}

This definition of the ray class group for an order $\OO$ parallels the definition of the ray class group for the maximal order. 
To make the definition flexible, we will show in \Cref{subsec:43} that one may add auxiliary congruence conditions without changing the group, for example, requiring coprimality to the conductor ideal $\ff(\OO)$.

\begin{defn}\label{defn:ring-class-order}
The \textit{(wide) ring class group} (or {\em Picard group}) $\Cl(\OO)$ of an order $\OO$ is the special
case of modulus $(\OO,\emptyset)$, so that 
\begin{equation}
\Cl(\OO) := \Cl_{\OO,\emptyset}(\OO) = \frac{\rJ_{\OO}^{\ast}(\OO)}{\rP_{\OO,\emptyset}(\OO)} =  \frac{\{\mbox{invertible fractional ideals of $\OO$}\}}{\{\mbox{$\alpha\OO$: $\alpha \in K^{\times}$}\}}.
\end{equation}
\end{defn}

One can show this definition of ring class group (and the corresponding ring class field described by splitting of degree one prime ideals) is consistent with the classical definitions in the case of quadratic fields, as given in \cite[pp.~114--115]{Cohn94} and \cite[pp.~162--163]{Cox13}.

\subsection{Local behavior of ideals}\label{subsec:42} 

Before proving results about the ray class group, we collect some facts about the interaction of localization with invertibility and ideal inclusion.

We introduce a notation for localization of rings, for later use.

\begin{defn}\label{defn:localization}
For a commutative ring with unity $R$ and an ideal $I$ of $R$, we denote by $S_I$ the multiplicatively closed set of elements coprime to $I$,
\begin{equation}
S_I := \{a \in R : aR+I=R\}.
\end{equation}
We denote by $\loc{R}{I}$ the ring defined by inverting the elements of $S_I$. (We avoid the notation $S_I^{-1}R$ to prevent any confusion with multiplication of ideals.)
\end{defn}

If $R$ is a Noetherian ring of dimension 1 (such as an order in a number field), then for any nonzero ideal $I$, the ring $\loc{R}{I}$ is a \textit{semilocal ring}---a ring with finitely many maximal ideals. For example, the ring $\loc{\Z}{(6)}$ consists of those rational numbers whose denominators are contain no factors of $2$ or $3$, and the only maximal ideals are $(2)$ and $(3)$. If $\pp$ is a maximal ideal of $R$, then $\loc{R}{\pp} = R[(R \setminus \pp)^{-1}] = R_\pp$ is termed the {\em localization away from} $\pp$. 

The localization of an $R$-modules $M$ may be defined by the extension of scalars $\loc{M}{I} := M \tensor_R \loc{R}{I}$ (or by an equivalent direct construction given in \cite[Ch.\ 3]{AM:69}). We use the notation $M_\pp := \loc{M}{\pp} = M \tensor_R R_\pp$ for the localization of an $\OO$-module $M$ away from a prime ideal $\pp$. In the cases we consider, the natural map $M \to M_\pp$ is injective, so we will drop the tensor product notation, simply writing $\loc{M}{\pp} = M R_\pp = R_\pp M$.

The following  proposition recalls a basic fact from commutative algebra: Fractional ideal containment (and more generally, injectivity of module maps) is a local property.
\begin{prop}\label{prop:locinj}
For any commutative ring $R$ and a map of $R$-modules $\phi : M \to N$, the following are equivalent:
\begin{itemize}
\item[(1)] $\phi$ is injective.
\item[(2)] The induced map $\phi_\pp : M_\pp \to N_\pp$ is injective for every prime ideal $\pp$ of $R$.
\item[(3)] The induced map $\phi_\pp : M_\pp \to N_\pp$ is injective for every maximal ideal $\pp$ of $R$.
\end{itemize}
In particular, for an order $\OO$ and fractional ideals $\aa, \bb \in \rJ(\OO)$,
\begin{equation}
\aa \subseteq \bb \iff \aa_\pp \subseteq \bb_\pp \mbox{ for all nonzero prime ideals } \pp \subseteq \OO.
\end{equation}
\end{prop}
\begin{proof}
The statement for $R$-modules is \cite[Prop.\ 3.9]{AM:69}. The statement for fractional $\OO$-ideals follows by taking $\phi: \aa \to \bb$ to be the inclusion map.
\end{proof}

We next recall a sharpening of \Cref{prop:invertible-ops3}, valid for orders of number fields.

\begin{prop}\label{prop:lociso}
Let $\OO$ be an order in a number field $K$. Invertible fractional ideals of $\OO$ are locally principal:~If $\aa \in \rJ^\ast(\OO)$, then $\aa_\pp = \aa \OO_p$ is a principal ideal of $\OO_\pp$. Moreover, the correspondence $\aa \mapsto (\aa_\pp) := (\aa\OO_\pp)$ defines an isomorphism
\begin{equation}
\rJ^\ast(\OO) \isom \bigoplus_\pp \rP(\OO_\pp),
\end{equation}
where $\pp$ varies over the nonzero prime ideals of $\OO$, and
$\rP(\OO_{\pp})= \{\alpha\OO_{\pp} : \alpha \in K_{\pp}\}$
is the group of nonzero principal fractional ideals of $\OO_{\pp}$ in its quotient field $K_{\pp}$.
\end{prop}
\begin{proof}
This is \cite[Ch.\ I, Prop.\ 12.6]{Neukirch13} and is also proved in \cite[Thm.~2.14]{Stevenhagen19}.
\end{proof}

We now show that the group of invertible fractional ideals coprime to $\mm$ is determined by the set of nonzero prime ideals containing $\mm$.
\begin{lem}\label{lem:mtilde}
Let $\OO$ be an order in a number field $K$, and let $\mm_1, \mm_2$ be nonzero integral ideals of $\OO$. Then:
\begin{itemize}
\item[(1)] One has
$\{\pp : \pp \mbox{ a prime $\OO$-ideal with } \mm_1 \subseteq \pp\} \subseteq \{\pp : \pp \mbox{ a prime $\OO$-ideal with } \mm_2 \subseteq \pp\}$ if and only if $\rI_{\mm_2}(\OO) \subseteq \rI_{\mm_1}(\OO)$.
\item[(2)] If $\rI_{\mm_2}(\OO) \subseteq \rI_{\mm_1}(\OO)$, then $\rJ^{\ast}_{\mm_2}(\OO) \subseteq \rJ^{\ast}_{\mm_1}(\OO).$
\item[(3)] For any nonzero $\OO$-ideal $\mm$, there exists an invertible ideal $\widetilde{\mm} \in \rI^\ast(\OO)$ such that $\widetilde{\mm} \subseteq \mm$ and $\rI_\mm(\OO) = \rI_{\widetilde{\mm}}(\OO)$ (and thus $\rJ^\ast_\mm(\OO) = \rJ^\ast_{\widetilde{\mm}}(\OO)$).
\end{itemize}
\end{lem}
\begin{proof}
If $\mm_1 \subseteq \pp$ but $\mm_2 \not\subseteq \pp$, then $\pp \in \rI_{\mm_2}(\OO)$ (because $\pp$ is maximal and hence $\pp+\mm_2=\OO$) but $\pp \nin \rI_{\mm_1}(\OO)$ (because $\pp+\mm_1=\pp\neq\OO$). This proves the ``if'' direction of (1) (by proving the contrapositive).
To prove the ``only if'' direction of (1), suppose 
\begin{equation}
\{\pp_1, \ldots, \pp_m\} = \{\mbox{nonzero prime ideals } \pp \supseteq \mm_1\} \subseteq \{\mbox{nonzero prime ideals } \pp \supseteq \mm_2 \}.
\end{equation}
For any \textit{integral} ideal $\bb \in \rI(\OO)$, we have
\begin{align}
\bb + \mm_2 = \OO 
&\implies \bb\OO_{\pp_j} + \mm_2\OO_{\pp_j} = \OO_{\pp_j} \mbox{ for } 1 \leq j \leq m 
\\
&\implies \bb\OO_{\pp_j} = \OO_{\pp_j} \mbox{ for } 1 \leq j \leq m, \mbox{ because } \mm_2\OO_{\pp_j} \subseteq \pp_j\OO_{\pp_j} \nn\\
&\ \ \ \ \ \ \ \ \ \ \ \mbox{and } \pp_j\OO_{\pp_j} \mbox{ is the unique maximal ideal of } \OO_{\pp_j} \\
&\implies \bb\OO_\pp + \mm_1\OO_\pp = \OO_\pp \mbox{ for all nonzero prime $\OO$-ideals } \pp, \nn\\
&\ \ \ \ \ \ \ \ \ \ \ \mbox{because } \mm_1\OO_\pp=\OO_\pp \mbox{ for } \pp \nin \{\pp_1, \ldots, \pp_m\} \\
&\implies \bb+\mm_1 = \OO \mbox{ by \Cref{prop:locinj}}.
\end{align}
In other words, $\rI_{\mm_2}(\OO) \subseteq \rI_{\mm_1}(\OO)$.

To prove (2), by definition $\rJ^\ast_{\mm}(\OO)$ consists of fractional ideals of the form $\aa\bb^{-1}$ where $\aa, \bb \in \rI^\ast_{\mm}(\OO)$. 
Thus, $\rI_{\mm_2}(\OO) \subseteq \rI_{\mm_1}(\OO) \implies \rI^\ast_{\mm_2}(\OO) \subseteq \rI^\ast_{\mm_1}(\OO) \implies \rJ^\ast_{\mm_2}(\OO) \subseteq \rJ^\ast_{\mm_1}(\OO)$.

To prove (3), consider a nonzero $\OO$-ideal $\mm$. 
Let $\{\pp_1, \ldots, \pp_m\}$ be the set of all prime ideals containing $\mm$.
For each $j$ with $1 \leq j \leq m$, choose some nonzero element $\mu_j \in \mm\OO_{\pp_j}$.
By \Cref{prop:lociso}, there exists a unique invertible ideal $\widetilde{\mm} \in \rJ^\ast(\OO)$ such that 
$\widetilde{\mm}\OO_{\pp_j} = \mu_j\OO_{\pp_j}$ for $1 \leq j \leq m$ and $\widetilde{\mm}\OO_\pp = \OO_\pp$ for $\pp \nin \{\pp_1, \ldots, \pp_m\}$. 
We have 
\begin{equation}
\{\mbox{nonzero prime ideals $\pp$ with } \mm \subseteq \pp\} = \{\mbox{nonzero prime ideals $\pp$ with } \widetilde{\mm} \subseteq \pp \}.
\end{equation}
Thus, by (1), both $\rI_{\widetilde{\mm}}(\OO) \subseteq \rI_{\mm}(\OO)$ and $\rI_{\mm}(\OO) \subseteq \rI_{\widetilde\mm}(\OO)$.
Then, by (2), $\rJ^\ast_\mm(\OO) = \rJ^\ast_{\widetilde{\mm}}(\OO)$.
\end{proof}

\subsection{Auxiliary coprimality conditions on ray class groups of orders}\label{subsec:43} 

In the definition of the ray class group for the maximal order $\OO_K$, it is well known that restricting  the 
fractional ideal group to fractional ideals coprime to some ideal $\dd$, and restricting the principal ideals
similarly, yields the same group. 
In this subsection we show this 
is true for arbitrary orders. 

Being able to impose an auxiliary coprimality
condition to some ideal $\dd$ on the fractional ideal groups,
done in \Cref{lem:reldok},
makes possible the comparison of ray class groups for different orders
and different ray class moduli, 
particularly the comparison of arbitrary ray class groups with certain ray class groups of the maximal order. 

\begin{defn}[Principal ray ideal group coprime to $\dd$]\label{def:rayclasses2} 
Given an integral $\OO$-ideal $\dd$, the group of \textit{principal ray ideals (modulo $\mm$) coprime to $\dd$}, denoted $\rP_{\mm,\rS}^{\dd}(\OO)$, 
is given by: 
\begin{equation}
\rP_{\mm,\rS}^{\dd}(\OO) = \{\mbox{$\alpha\OO$ : $\alpha \in K^\times$, $\alpha \Con 1 \Mod{\mm}$, $\alpha\OO$ coprime to $\dd$,
$\rho(\alpha) > 0$ for all $\rho \in \rS$} \}.
\end{equation}
\end{defn}

The following lemma shows that, without loss of generality, we may suppose $\dd \subseteq \mm$. 

\begin{lem}\label{lem:WOLOG}
If $\dd$ is an integral $\OO$-ideal, then $\rP_{\mm,\rS}^\dd(\OO) = \rP_{\mm,\rS}^{\dd\cap\mm}(\OO) = \rP_{\mm,\rS}^{\dd\mm}(\OO)$.
In particular, $\rP_{\mm,\rS}(\OO) = \rP_{\mm,\rS}^\mm(\OO)$.
\end{lem}
\begin{proof}
Clearly $\rP_{\mm,\rS}^{\dd\mm}(\OO) \subseteq \rP_{\mm,\rS}^{\dd\cap\mm}(\OO) \subseteq \rP_{\mm,\rS}^{\dd}(\OO)$. To show the reverse inclusions, suppose $\aa \in \rP_{\mm,\rS}^\dd(\OO)$. Write $\aa = \alpha\OO$ for $\alpha \in K^\times$ such that $\alpha \Con 1 \Mod{\mm}$, $\alpha$ coprime to $\dd$, and
$\rho(\alpha) > 0$ for all $\rho \in \rS$. Write $\alpha = \frac{\alpha_1}{\alpha_2}$ for $\alpha_1, \alpha_2 \in \OO$ with $\alpha_2$ coprime to $\mm$ (which is possible by \Cref{defn:RC-congruence}). Then $\alpha_1 \Con \alpha_2 \Mod{\mm}$, so $\alpha_1$ is also coprime to $\mm$, so $\alpha$ is coprime to $\mm$. Since $\alpha$ is coprime to $\dd$ and to $\mm$, $\alpha$ is also coprime to $\dd\mm$ and to $\dd\cap\mm$.

The equality $\rP_{\mm,\rS}(\OO) = \rP_{\mm,\rS}^\mm(\OO)$ follows by taking $\dd=\OO$.
\end{proof}

The following key lemma is the most technically tricky result in this paper.  
\begin{lem}\label{lem:reldok}
For any nonzero ideals $\dd \subseteq \mm \subseteq \OO$ and any set
$\Sigma$ of real places of $K$ (possibly empty), 
\begin{equation}\label{eq:45}
\Cl_{\mm,\rS}(\OO) = \frac{\rJ_{\dd}^{\ast}(\OO)}{\rP^{\dd}_{\mm,\rS}(\OO)}.
\end{equation}
\end{lem}
\begin{proof}
Consider the inclusion $\rJ_{\dd}^{\ast} (\OO) \inj \rJ_{\mm}^{\ast}(\OO)$, and compose with the quotient map to get a homomorphism
\begin{equation}\label{eqn:surjective}
\phi : \rJ_{\dd}^{\ast}(\OO) \to \frac{\rJ_{\mm}^{\ast}(\OO)}{\rP_{\mm,\rS}(\OO)} = \Cl_{\mm,\rS}(\OO).
\end{equation}
The kernel $\ker \phi = \rJ_{\dd}^{\ast}(\OO) \cap \rP_{\mm,\rS}(\OO) = 
\rP^{\dd}_{\mm,\rS}(\OO)$. 
For the equality \eqref{eq:45}, it suffices to prove that $\phi$ is surjective. 

Given $\bb \in \rJ_\mm^{\ast}(\OO)$, to show surjectivity we must find some 
$\beta\OO \in \rP_{\mm,\rS}(\OO)$  (equivalently, $\beta^{-1}\OO \in \rP_{\mm,\rS}(\OO) $) such that
$\aa= \beta^{-1} \bb \in \rJ_\dd^{\ast} (\OO)$.
To construct a suitable element $\beta$, the argument will move to the maximal order $\OO_K$ and then back to $\OO$.

For the integral ideal $\mm$, by \Cref{lem:mtilde}(3), there exists some invertible ideal $\widetilde{\mm} \subseteq \mm$ with the property that the set $\{\qq_1, \ldots \qq_n\}$ of all nonzero prime ideals of $\OO$ for which $\widetilde{\mm} \subseteq \qq_j$ is identical to  the set of nonzero prime ideals with  $\mm \subseteq \qq_j$. By   \Cref{prop:lociso}, the invertible fractional ideals $\widetilde{\mm}, \bb$ are locally principal: For each nonzero prime ideal $\pp$ of $\OO$, 
\begin{enumerate}
\item $\widetilde{\mm}\OO_\pp = \mu_\pp\OO_\pp$ for some $\mu_\pp \in K^\times$, and we may  choose $\mu_\pp=1$ whenever $\pp \nin \{\qq_1, \ldots, \qq_n\}$;
\item $\bb\OO_\pp = \beta_{\pp}\OO_\pp$ for some $\beta_\pp \in K^\times$, and we may choose $\beta_\pp=1$ whenever $\bb \not\subseteq \pp$.
\end{enumerate}
The finite set of prime ideals $\pp$ with $\bb \not\subseteq \pp$ is a union of two subsets:
\begin{enumerate}
\item
a subset (possibly empty) of the set of prime ideals $\{\pp_i : 1 \le i \le m\}$ having $\dd \subseteq \pp_i$ and $\mm \not\subseteq \pp_i$;
\item
a set (possibly empty) of prime ideals $\{\rr_k : 1 \leq j \leq \ell\}$ having $\dd \not\subseteq \rr_k$.
\end{enumerate} 
(Note that the prime ideals $\pp_i$, $\qq_j$, $\rr_k$ need not be invertible.)

The condition that $\bb + \mm = \OO$ is equivalent to the condition that $\bb \not\in \qq_j$ for $1 \leq j \leq n$, and, in turn, to the condition that $\bb + \widetilde{\mm} = \OO$. Additionally, it follows that the sets $\{\pp_1, \ldots, \pp_m\}$, $\{\qq_1, \ldots, \qq_n\}$, and  $\{\rr_1, \ldots, \rr_\ell\}$ are all disjoint. 

To show surjectivity \eqref{eqn:surjective} our object is to multiply $\bb$ by an element in $\rP_{\mm,\rS}(\OO)$ to force an additional coprimality condition with respect to the prime ideals $\{\pp_1, \cdots \pp_m\}$ without changing its behavior locally at the prime ideals $\qq_i$. 
The set of primes $\{\rr_1, \ldots, \rr_\ell\}$, being disjoint from both the set of $\pp_i$ and the set of $\qq_j$, play no role in the following argument.

We now  move to the maximal order $\OO_K$.
For $1 \leq i \leq m$, write $\pp_i = \con(\pp_i')$ for some nonzero prime ideal $\pp_i'$ of $\OO_K$; this is possible by \Cref{lem:conextmax}. Similarly, for $1 \leq j \leq n$, write $\qq_j = \con(\qq_j')$ for some nonzero prime ideal $\qq_j'$ of $\OO_K$. These primes $\pp_i', \qq_j'$ are all distinct, because the primes $\pp_i, \qq_j$ are all distinct. Thus, there are pairwise independent multiplicative valuations (absolute values) $\abs{\cdot}_{v_1}, \ldots, \abs{\cdot}_{v_m}, \abs{\cdot}_{w_1}, \ldots, \abs{\cdot}_{w_n}$ on $K$ corresponding to $\pp_1', \ldots, \pp_m', \qq_1', \ldots \qq_n'$, respectively.

Let $\ff_\pp$ be the ``local conductor'' at $\pp$, that is,
\begin{equation}
\ff_\pp = \{x \in \OO_\pp : x\overline{\OO}_\pp \subseteq \OO_\pp\},
\end{equation}
where $\overline{\OO}_\pp$ is the integral closure of $\OO_\pp$ in its fraction field $K_\pp=K$.
Let $\tilde{f}_\pp$ be any nonzero element of $\ff_\pp$,
so $\tilde{f}_\pp\OO_\pp$ is a nonzero principal ideal contained in $\ff_\pp$.

By the Approximation Theorem \cite[Thm.\ 3.4]{Neukirch13}, we can find $\beta \in K$ such that 
\begin{itemize}
\item[(1)] $\abs{\beta - \beta_{\pp_i}}_{v_i} < \abs{\beta}_{v_i}$ for $1 \leq i \leq m$,
\item[(2)] $\abs{\beta - 1}_{w_j} \leq \abs{\tilde{f}_{\qq_j}\mu_{\qq_j}}_{w_j}^{-1}$ for $1 \leq j \leq n$, and
\item[(3)] $\rho(\beta) > 0$ for $\rho \in \rS$, via the Archimedean bound $\abs{\beta-1}_{\rho} < 1$. 
\end{itemize}
We now define the invertible fractional $\OO$-ideal $\aa$ by $\aa := \beta^{-1} \bb$.

We have $\aa\OO_{\pp} = \beta^{-1}\beta_{\pp}\OO_{\pp}$ for all prime $\OO$-ideals $\pp$. 
By (1), for $1 \leq i \leq m$, $\abs{\beta^{-1}\beta_{\pp_i}-1}_{v_i} < 1$, so in particular, $\beta^{-1}\beta_{\pp_i}$ is a unit in $\OO_{\pp_i'} = \overline{\OO}_{\pp_i}$. But $\overline{\OO}_{\pp_i}^\times \cap \OO_{\pp_i} = \OO_{\pp_i}^\times$, so $\beta^{-1}\beta_{\pp_i}$ is a unit in $\OO_{\pp_i}$, and thus $\aa\OO_{\pp_i} = \OO_{\pp_i}$.
On the other hand, for $1 \leq j \leq n$, $\aa\OO_{\qq_j} = \beta^{-1}\OO_{\qq_j}$ 
because $\qq_j \supseteq \mm$, and thus $\beta_{\qq_j}=1$. It follows from (2) that $\abs{\beta^{-1}-1}_{w_j} < 1$, so in particular, $\beta^{-1}$ is a unit in $\OO_{\qq_j'} = \overline{\OO}_{\qq_j}$ 
and thus a unit in $\OO_{\qq_j}$, so $\aa\OO_{\qq_j} = \OO_{\qq_j}$. 
Since $\aa\OO_{\pp} = \OO_{\pp}$ for all $\pp \in \{\pp_1, \ldots, \pp_m\}\cup\{\qq_1, \ldots, \qq_n\}$, it follows that $\aa$ is coprime to $\dd$, so $\aa \in \rJ^\ast_\dd(\OO)$.

Moreover, for $1 \leq j \leq n$, $\abs{\beta^{-1} - 1}_{w_j} \leq \abs{\tilde{f}_{\qq_j}\mu_{\qq_j}}_{w_j}^{-1}$ by (2), so 
$\beta^{-1} - 1 \in \tilde{f}_{\qq_j}\mu_{\qq_j}\OO_{\qq_j'}$. We have the inclusion of ideals
\begin{equation}
\tilde{f}_{\qq_j}\mu_{\qq_j}\OO_{\qq_j'} = \mu_{\qq_j}(\tilde{f}_{\qq_j}\overline{\OO}_{\qq_j}) \subseteq \mu_{\qq_j}\OO_{\qq_j} = \widetilde{\mm}\OO_{\qq_j},
\end{equation}
so $\beta^{-1}-1 \in \widetilde{\mm}\OO_{\qq_j}$ for each $j$.
It follows that $\beta^{-1} \equiv 1 \Mod{\widetilde{\mm}}$, so in particular, $\beta^{-1} \equiv 1 \Mod{\mm}$, 
since $\widetilde{\mm} \subseteq \mm$. Combining this congruence with the positivity condition (3), we have $\aa\bb^{-1} = \beta^{-1}\OO \in \rP_{\mm,\rS}(\OO_K)$.
\end{proof}

\subsection{Effect of change of order and modulus on ray class groups of orders}\label{subsec:44} 

While it is clear that \Cref{lem:reldok} implies the existence of surjective change-of-modulus maps $\Cl_{\mm,\rS}(\OO) \to \Cl_{\mm', \rS'}(\OO)$ whenever $\mm \subseteq \mm'$ and $\rS \supseteq \rS'$, it is also true (but not immediately obvious) that one can change the order $\OO$. For orders $\OO \subseteq \OO'$, the change-of-order map between ray class groups is induced by the extension map $\ext: \rJ^\ast(\OO) \to \rJ^\ast(\OO')$ on fractional ideals. We thus call it the \textit{induced extension map} $\ol{\ext}$.

We show that the induced extension map is well-defined and surjective. This result will be applied and refined in \Cref{prop:exseq1} to explicitly describe the kernel of $\ol{\ext}$.

\begin{lem}\label{lem:surjtook}
Let $K$ be a number field, and consider level data $\ddD = (\OO;\mm,\rS)$ and $\ddD' = (\OO';\mm',\rS')$ for $K$ such that $\OO \subseteq \OO'$, $\mm\OO' \subseteq \mm'$, and $\rS \supseteq \rS'$.
There exists a unique homomorphism
\begin{equation}\label{eq:57}
\ol{\ext} = \ol{\ext}_{\ddD}^{\ddD'} : \Cl_{\mm,\rS}(\OO) \to \Cl_{\mm',\rS'}(\OO')
\end{equation}
satisfying $\ol{\ext}\!\left([\aa]\right) = [\ext(\aa)]$ for $\aa \in \rJ_\mm^\ast(\OO)$, and this map is surjective. For another level datum $\ddD'' = (\OO'',\mm'',\rS'')$ for $K$ such that $\OO' \subseteq \OO''$, $\mm'\OO'' \subseteq \mm''$, and $\rS' \supseteq \rS''$, the compatibility relation $\ol{\ext}_{\ddD'}^{\ddD''} \!\circ \ol{\ext}_{\ddD}^{\ddD'} = \ol{\ext}_{\ddD}^{\ddD''}$ holds.
\end{lem}
\begin{proof}
Uniqueness follows from the formula $\ol{\ext}\!\left([\aa]\right) = [\ext(\aa)]$, provided that this formula defines a well-defined homomorphism. 
 
To see that $\ol{\ext}$ is a well-defined, it suffices to show the image of $\rP_{\mm, \rS}(\OO)$ under ring extension from $\OO$ to $\OO'$ is contained in $\rP_{\mm', \rS'}(\OO')$. 
So consider an ideal $\bb \in \rP_{\mm, \rS}(\OO)$. 
Then $\bb = b\OO$ for some $b \in K$, $b - 1 \in \mm\OO[S_\mm^{-1}]$, and $\rho(b) >0$ for $\rho \in \rS$. 
Since $\mm\OO[S_\mm^{-1}] \subseteq \mm\OO'[S_\mm^{-1}] \subseteq \mm'\OO'[S_{\mm'}^{-1}]$, we have $b - 1 \in \mm'\OO'[S_{\mm'}^{-1}]$, that is, $b \equiv 1 \Mod{\mm'}$. 
Additionally, since $\rS \supseteq \rS'$, we have $\rho(b)>0$ for $\rho \in \rS'$.
Thus, $\ext(\bb) \in \rP_{\mm', \rS'}(\OO')$.

Let $\ff = \ff_{\OO'}(\OO)$ be the relative conductor of $\OO$ in $\OO'$. Let $\dd$ be an ideal of $\OO'$ contained in both $\mm$ and $\ff$. (For example, take $\dd = \colonideal{\mm}{\OO'}$.) Using \Cref{lem:reldok}, under the inclusion map on ideals, we have
\begin{equation}
\Cl_{\mm',\rS'}(\OO') := 
\frac{\rJ_{\mm'}^{\ast}(\OO')}{\rP_{\mm',\rS'}(\OO')} 
= \frac{\rJ_{\dd}^{\ast}(\OO')}{\rP^{\dd}_{\mm',\rS'}(\OO')}.
\end{equation}
To see that the map $\ol{\ext}$ is surjective, let $A \in \Cl_{\mm',\rS'}(\OO')$, and let $\aa \in \rJ_{\dd}^{\ast}(\OO')$ be a representative of the class $A$. 
Because $\aa$ is coprime to the conductor $\ff$, we have $\ext(\con(\aa))=\aa$ by \Cref{lem:conext}, so $\ol{\ext}\!\left([\con(\aa)]\right) = [\aa] = A$.

The compatibility relation $\ol{\ext}_{\ddD'}^{\ddD''} \!\circ \ol{\ext}_{\ddD}^{\ddD'} = \ol{\ext}_{\ddD}^{\ddD''}$ is a direct consequence of the formula $\ol{\ext}\!\left([\aa]\right) = [\ext(\aa)]$.
\end{proof}

\section{Exact sequences for ray class groups of orders}\label{sec:5} 

In this section, we compare ray class groups for pairs of level data for the same field. So as to have a convenient shorthand to indicate which pairs of level data are comparable, we introduce a partial order on level data.
\begin{defn}\label{defn:levelleq}
Let $\ddD = (\OO; \mm, \rS)$ and $\ddD' = (\OO'; \mm', \rS')$ be level data for the same number field $K$. (That is, $\OO, \OO'$ are $K$-orders, $\mm$ is an integral $\OO$-ideal, $\mm'$ is an integral $\OO'$-ideal, and $\rS, \rS' \subseteq \{\text{embeddings } K \inj \R\}$.) We say that
\begin{equation}
\ddD \leq \ddD'
\end{equation}
if and only if
\begin{align}
\OO \subseteq \OO', \, \mm\OO' \subseteq \mm', \mbox{ and } \rS \supseteq \rS'.
\end{align}
\end{defn}
When $\ddD \leq \ddD'$, the induced extension map $\ol{\ext}_{\ddD}^{\ddD'}$ is shown in \Cref{lem:surjtook} to be a surjective group homomorphism from $\Cl_{\mm,\rS}(\OO)$ to $\Cl_{\mm',\rS'}(\OO')$. In this section, we completely describe the kernel of $\ol{\ext}$ in terms of local and global unit groups.

\subsection{Exact sequences relating unit groups and principal ideals for varying orders}\label{subsec:51} 

We describe unit groups that we will relate by exact sequences to groups of principal ideals of varying orders.

\begin{defn}\label{defn:ugroups}
For a commutative ring with unity $R$ and an ideal $I$ of $R$, define the group
\begin{equation}
\U_I(R) := \{\alpha \in R^\times : \alpha \equiv 1 \Mod{I}\} = (1+I) \cap R^\times.
\end{equation}

If $R$ has real embeddings, and $\rS$ is a subset of the real embeddings of $R$, define
\begin{equation}
\U_{I,\rS}(R) := \{\alpha \in R^\times : \alpha \equiv 1 \Mod{I} \mbox{ and } \rho(\alpha)>0 \mbox{ for } \rho \in \rS\}.
\end{equation}
\end{defn}

We also make use of the following extension of this notation: If there is an obvious map $\phi : R_1 \to R_2$ implicit in the discussion, and if $I$ is an ideal of $R_1$, then
we will let $\U_I(R_2) := \U_{\phi(I)R_2}(R_2)$ and $\U_{I,\rS}(R_2) := \U_{\phi(I)R_2,\rS}(R_2)$.

\begin{prop}\label{prop:uupseq}
For any level datum $(\OO;\mm,\rS)$,
we have an exact sequence
\begin{equation}
1 \to \U_{\mm,\rS}(\OO) \to \U_{\mm,\rS}(\loc{\OO}{\dd}) \to \rP^{\dd}_{\mm,\rS}(\OO) \to 1,
\end{equation}
where $S_{\dd} = \{ \alpha \in \OO : \alpha \OO + \dd = \OO\}$.
\end{prop}
\begin{proof}
By definition,
\begin{equation}
\rP^{\dd}_{\mm,\rS}(\OO) = \{\alpha\OO : \alpha \in \loc{\OO}{\dd}, \ \alpha \equiv 1 \Mod{\mm}, \ \rho(\alpha)>0 \mbox{ for all } \rho \in \rS\},
\end{equation}
so $\phi(\alpha) := \alpha\OO$ defines a surjective map $\phi: \U_{\mm,\rS}(\loc{\OO}{\dd}) \to \rP^{\dd}_{\mm,\rS}(\OO)$.
Moreover,
\begin{align}
\ker(\phi) &= \{\alpha \in \loc{\OO}{\dd} : \alpha\OO = \OO, \ \alpha \equiv 1 \Mod{\mm}, \ \rho(\alpha)>0 \mbox{ for all } \rho \in \rS\} \\
&= \U_{\mm,\rS}(\OO).
\end{align}
The proposition follows.
\end{proof}

We now relate unit groups and principal ideal groups of varying orders.
\begin{prop}\label{prop:exseq2}
Let $K$ be a number field, and consider level data $\ddD = (\OO; \mm, \rS)$ and $\ddD' = (\OO'; \mm',\rS')$ for $K$ such that $\ddD \leq \ddD'$.
Let $\dd$ be any $\OO'$-ideal such that $\dd \subseteq \colonideal{\mm}{\OO'}$.
Then we have a short exact sequence of the form
\begin{equation}
1 \to \frac{\U_{\mm',\rS'}\left(\OO'\right)}{\U_{\mm,\rS}\left(\OO\right)} \to \frac{\U_{\mm'}\left(\OO'/\dd\right)}{\U_{\mm}\left(\OO/\dd\right)} \times \{\pm1\}^{|\rS \setminus \rS'|}
\to \frac{\rP_{\mm',\rS'}^{\dd}(\OO')}{\rP_{\mm,\rS}^{\dd}(\OO)} \to 1.
\end{equation}
\end{prop}

\begin{proof}
Here $\dd$ is an integral $\OO$-ideal because $\dd \subseteq \colonideal{\mm}{\OO'}=\{ a \in K : a\OO' \subseteq \mm\} \subseteq \mm \subseteq \OO$,
and $\dd\OO \subseteq \dd\OO' = \dd$. Thus $\OO/\dd$ is well-defined. We localize away from $\dd$ by inverting $S_{\dd}$,
for the two rings $\OO$ and $\OO'$ separately.
By \Cref{prop:uupseq}, we have short exact sequences
\begin{equation}
\begin{array}{ccccccccc}
1 & \to & \U_{\mm,\rS}(\OO) & \to & \U_{\mm,\rS}(\loc{\OO}{\dd}) & \to & \rP_{\mm,\rS}^{\dd}(\OO) & \to & 1 \\
& & \downarrow & & \downarrow & & \downarrow & & \\
1 & \to & \U_{\mm',\rS'}(\OO') & \to & \U_{\mm',\rS'}(\loc{\OO'}{\dd}) & \to & \rP_{\mm',\rS'}^{\dd}(\OO') & \to & 1.
\end{array}
\end{equation}
The downward maps are all injective---in particular, the rightmost map is---so, by the snake lemma, the sequence of cokernels is exact:
\begin{equation}
1 \to \frac{\U_{\mm',\rS'}(\OO')}{\U_{\mm,\rS}(\OO)}
 \to \frac{\U_{\mm',\rS'}(\loc{\OO'}{\dd})}{\U_{\mm,\rS}(\loc{\OO}{\dd})} 
 \to \frac{\rP_{\mm',\rS'}^{\dd}(\OO')}{\rP_{\mm,\rS}^{\dd}(\OO)}
 \to 1.
\end{equation}
This is the short exact sequence in the proposition statement, except for the middle group.
We now must show the middle group is isomorphic to the group in the proposition statement. 

First of all, note that the localization maps induce compatible isomorphisms
\begin{equation}
\begin{array}{cccc}
\iota : & \U_{\mm}(\OO/\dd) & \isomto & \U_{\mm}(\loc{\OO}{\dd}/\dd\loc{\OO}{\dd}) \\
& \downarrow & & \downarrow \\
\iota': & \U_{\mm'}(\OO'/\dd) & \isomto & \U_{\mm'}(\loc{\OO'}{\dd}/\dd\loc{\OO'}{\dd}).
\end{array}
\end{equation}
Let $\Upsilon = \{\mbox{embeddings } K \inj \R\}$. Consider the map
\begin{equation}
\phi : \U_{\mm',\rS'}(\loc{\OO'}{\dd}) \to \U_{\mm'}\left(\OO'/\dd\right) \times \{\pm 1\}^{\Upsilon \setminus \rS'}
\end{equation}
given by $\phi(u) = \left({\iota'}^{-1}\left(u \Mod{\dd\loc{\OO'}{\dd}}\right), (\sign(\rho(u)))_{\rho \in {\Upsilon \setminus \rS'}}\right)$.

We see that $\phi$ is surjective as follows: Consider $(u,\e) \in \U_{\mm'}\!\left(\OO'/\dd\right) \times \{\pm 1\}^{\Upsilon \setminus \rS'}$. 
Choose a lift $\widetilde{u} \in \OO'$ of $u$; $\widetilde{u} \in \U_{\mm'}(\loc{\OO'}{\dd})$ because it is coprime to $\dd$ and congruent to 1 modulo $\mm'$. 
We may replace $\widetilde{u}$ with $\widetilde{u}+\lambda$ for any $\lambda\in\dd\cap\mm'=\dd$
without affecting its image in $\U_{\mm'}\!\left(\OO'/\dd\right)$. Since $\dd$ forms a lattice in $K \tensor \R$, 
we may choose $\lambda$ appropriately so that $\sign(\rho(\widetilde{u}+\lambda)) = +1$ for 
$\rho \in \rS'$ and $\sign(\rho(\widetilde{u}+\lambda)) = \e_\rho$ for $\rho \in \Upsilon \setminus \rS$. 
Thus, $\widetilde{u}+\lambda \in \U_{\mm',\rS'}(\loc{\OO'}{\dd})$, and $\phi(\widetilde{u}+\lambda) = (u,\e)$.

Now, define
\begin{equation}
\ol{\phi} : \U_{\mm',\rS'}(\loc{\OO'}{\dd}) \to \frac{\U_{\mm'}\left(\OO'/\dd\right) \times \{\pm 1\}^{\Upsilon \setminus \rS'}}{\U_{\mm}\left(\OO/\dd\right) \times \{\pm 1\}^{\Upsilon \setminus \rS}}
\end{equation}
by $\ol{\phi}(u) := \phi(u) \Mod{\U_{\mm}\left(\OO/\dd\right) \times \{\pm 1\}^{E \setminus \rS}}$; $\ol{\phi}$ is surjective because $\phi$ is.
Furthermore, we compute the kernel of $\ol{\phi}$ as follows.
\begin{align}
\ol{\phi}(u) = 1
& \iff {\iota'}^{-1}\left(u \Mod{\dd\loc{\OO'}{\dd}}\right) \in \U_{\mm}(\OO/\dd)
\mbox{ and } \sign(\rho(u))=1 \mbox{ for all } \rho \in \rS \setminus \rS'.
\end{align}
We already knew that $\sign(\rho(u)) = 1$ for all $\rho \in \rS'$, so in fact, the second condition tells us that $\rho(u) > 0$ for all $\rho \in \rS$. 
We now reformulate the first condition.
\begin{align}
{\iota'}^{-1}&\!\left(u \Mod{\dd\loc{\OO'}{\dd}}\right) \in \U_{\mm}(\OO/\dd) \\
& \iff u \Mod{\dd\loc{\OO'}{\dd}} \in \iota'\!\left(\U_{\mm}(\OO/\dd)\right) = \iota\!\left(\U_{\mm}(\OO/\dd)\right) = \U_{\mm}\!\left(\loc{\OO}{\dd}/\dd\loc{\OO}{\dd}\right) \\
& \iff u \equiv 1 \Mod{\mm \loc{\OO}{\dd}}.
\end{align}
This last condition also implies that $u \in \loc{\OO}{\dd}$ (rather than simply being in $\loc{\OO'}{\dd}$), since $\dd$ is an integral $\OO$-ideal. Thus, we have shown that
\begin{align}
\ker\left(\ol{\phi}\right) 
&= \{u \in \U_{\mm',\rS'}(\loc{\OO'}{\dd}) : u \in \loc{\OO}{\dd}, \ u \equiv 1 \Mod{\mm \loc{\OO}{\dd}}, \ \rho(u) > 0 \mbox{ for } \rho \in \rS\} \\
&= \U_{\mm,\rS}(\loc{\OO}{\dd}).
\end{align}
Therefore, $\ol{\phi}$ induces an isomorphism
\begin{equation}
\frac{\U_{\mm',\rS'}(\loc{\OO'}{\dd})}{\U_{\mm,\rS}(\loc{\OO}{\dd})} \isomto
\frac{\U_{\mm'}\left(\OO'/\dd\right) \times \{\pm 1\}^{E \setminus \rS'}}{\U_{\mm}\left(\OO/\dd\right) \times \{\pm 1\}^{E \setminus \rS}} \isom
\frac{\U_{\mm'}\left(\OO'/\dd\right)}{\U_{\mm}\left(\OO/\dd\right)} \times \{\pm 1\}^{\abs{\rS \setminus \rS'}},
\end{equation}
proving the proposition.
\end{proof}

\subsection{Exact sequences relating ray class groups of varying orders}\label{subsec:52} 

We relate class groups of varying orders and moduli; the formula \eqref{eq:exseq} is important for applications. The following results generalize results presented in \cite[Ch.~1, Sec.~10]{Neukirch13} corresponding to the special case $\ddD = (\OO;\OO,\emptyset)$, $\ddD' = (\OO_K; \OO_K,\emptyset)$. More loosely speaking, they generalize results presented in \cite[Sec.~7.2]{Cohn94} corresponding to the special case $\ddD = (\OO_K;\mm,\rS)$, $\ddD' = (\OO_K; \OO_K,\emptyset)$.

\begin{thm}\label{thm:exseq}
Let $K$ be a number field, and consider level data $\ddD= (\OO; \mm, \rS)$ and $\ddD'=(\OO'; \mm',\rS')$ for $K$ such that $\ddD \leq \ddD'$ in the sense of \Cref{defn:levelleq}.
Let $\dd$ be any $\OO'$-ideal such that $\dd \subseteq \colonideal{\mm}{\OO'}$.
With the $\U$-groups defined as in \Cref{defn:ugroups}, we have the following exact sequence.
\begin{equation}\label{eq:exseq}
1 \to \frac{\U_{\mm',\rS'}\!\left(\OO'\right)}{\U_{\mm,\rS}\!\left(\OO\right)} \to \frac{\U_{\mm'}\!\left(\OO'/\dd\right)}{\U_{\mm}\!\left(\OO/\dd\right)} \times \{\pm 1\}^{|\rS \setminus \rS'|} \to \Cl_{\mm,\rS}(\OO) \to \Cl_{\mm',\rS'}(\OO') \to 1.
\end{equation}
\end{thm}

To prove this result, we use the following proposition.

\begin{prop}\label{prop:exseq1}
Let $K$ be a number field, and consider level data $\ddD= (\OO; \mm, \rS)$ and $\ddD'=(\OO'; \mm',\rS')$ for $K$ such that $\ddD \leq \ddD'$.
Let $\dd$ be any $\OO'$-ideal such that $\dd \subseteq \colonideal{\mm}{\OO'}$.
Then we have an exact sequence of the form
\begin{equation}
1 \to \frac{\rP^{\dd}_{\mm',\rS'}(\OO')}{\rP^{\dd}_{\mm,\rS}(\OO)} \to \Cl_{\mm,\rS}(\OO) \to \Cl_{\mm',\rS'}(\OO') \to 1.
\end{equation}
\end{prop}
\begin{proof}
The extension map $\ol{\ext} : \Cl_{\mm,\rS}(\OO) \to \Cl_{\mm',\rS'}(\OO')$ is surjective by \Cref{lem:surjtook}, so the sequence is exact at $\Cl_{\mm',\rS'}(\OO')$.

Note that $\dd$ is both an $\OO'$-ideal and $\OO$-ideal, because it is an $\OO'$-ideal and
$\dd \subseteq \colonideal{\mm}{\OO'} \subseteq \mm \subseteq \OO$.
In addition $\dd \subseteq \mm$, so that $\dd \subseteq \mm'$ (because $\mm \subseteq \mm'$).
By \Cref{lem:reldok}, 
\begin{equation}
\Cl_{\mm,\rS}(\OO) = \frac{\rI_{\dd}(\OO)}{\rP^{\dd}_{\mm,\rS}(\OO)} \mbox{ and } \Cl_{\mm',\rS'}(\OO') = \frac{\rI_{\dd}(\OO')}{\rP^{\dd}_{\mm',\rS'}(\OO')}.
\end{equation}
The kernel of the extension map is
\begin{equation}
\ker(\ol{\ext}) = \{[\aa] \in \Cl_{\mm,\rS}(\OO) : \aa\OO' = \alpha\OO', \ \alpha \in \OO', \ \alpha \equiv 1 \Mod{\mm'}, \ \rho(\alpha)>0 \mbox{ for all } \rho \in \rS'\}.
\end{equation}
The ideal $\dd \subseteq \colonideal{\OO}{\OO'} = \ff_{\OO'}(\OO)$, so fractional ideals (of $\OO$ or $\OO'$) coprime to $\dd$ are also coprime to the conductor $\ff_{\OO'}(\OO)$.
By \Cref{prop:conext}, $\con$ and $\ext$ act as inverses on ideals coprime to the conductor, so setting $\phi(\alpha \OO') = [\con(\alpha\OO')]$ defines a surjective map 
\begin{equation}
\phi : \rP^{\dd}_{\mm',\rS'}(\OO') \to \ker(\ol{\ext}).
\end{equation}
Moreover,
\begin{equation}
\ker(\phi) = \{\alpha\OO' : \alpha \in \loc{\OO}{\dd}, \ \alpha \equiv 1 \Mod{\mm}, \ \rho(\alpha) > 0 \mbox{ for } \rho \in \rS\} = \rP^{\dd\cap\mm}_{\mm,\rS}(\OO) = \rP^{\dd}_{\mm,\rS}(\OO).
\end{equation}
The proposition follows.
\end{proof}

We now prove the main exact sequence.

\begin{proof}[Proof of \Cref{thm:exseq}]
The result follows by gluing together the two short exact sequences in \Cref{prop:exseq2} and \Cref{prop:exseq1}.
\end{proof}

We also give the specialization of \Cref{thm:exseq} to  
$\ddD=(\OO; \mm, \rS)$ and $\ddD'=(\OO; \OO, \emptyset)$.
This case will be important for the application to SIC-POVMs \cite{AFK,KLsic}.
\begin{cor}\label{cor:raystructure}
Let $K$ be a number field and $\OO$ an order of $K$.
Let $\mm$ be an ideal of $\OO$ and $\rS \subseteq \{\mbox{embeddings } K \inj \R\}$. 
With the $\U$-groups defined as in \Cref{defn:ugroups}, we have the following exact sequence.
\begin{equation}\label{eq:raystructure}
1 \to \frac{\OO^\times}{\U_{\mm,\rS}\!\left(\OO\right)} \to \left(\OO/\mm\right)^\times \times \{\pm 1\}^{|\rS|} \to \Cl_{\mm,\rS}(\OO) \to \Cl(\OO) \to 1.
\end{equation}
In particular, the kernel
\begin{equation}
\ker\!\left(\Cl_{\mm,\rS}(\OO) \surj \Cl(\OO)\right) \isom \frac{\left(\OO/\mm\right)^\times \times \{\pm 1\}^{|\rS|}}{{\rm im}(\OO^\times)},
\end{equation}
where ${\rm im}(\OO^\times)$ is the image of global units under 
the map $\e \mapsto \left(\ol{\e}, (\sgn(\rho(\e)))_{\rho \in \rS}\right)$.
\end{cor}

\begin{proof}
Take ray class level data $\ddD = (\OO; \mm, \rS)$ and $\ddD'=(\OO; \OO, \emptyset)$,
and choose $\dd=\mm$ in \Cref{thm:exseq}. All the hypotheses are satisfied,
noting  $\dd =\mm =  \colonideal{\mm}{\OO}$. The terms in the exact sequence simplify, 
with  $\U_{\mm}\!\left(\OO/\mm\right)= \U_{\dd}\!\left(\OO/\dd\right)$ being the trivial group,
 $\U_{\OO}\!\left(\OO/\dd\right)= (\OO/\mm)^{\times}$, and $\U_{\OO, \emptyset}(\OO)= \OO^{\times}$.
\end{proof}

\begin{rmk}\label{rmk:raystructure}
\Cref{cor:raystructure} parallels a result of Campagna and Pengo \cite[Thm.\ 4.6]{CampagnaP22} for their id\`elic formulation of class field theory for orders.
The kernel $\ker\!\left(\Cl_{\mm,\rS}(\OO) \surj \Cl(\OO)\right)$ is isomorphic to the Galois group $\Gal\!\left(H_{\mm,\rS}^\OO/H^\OO\right)$ by \Cref{thm:main3}, so \Cref{cor:raystructure} describes the structure of this Galois group. More generally, \Cref{thm:exseq} describes the structure of the Galois group $\Gal\!\left(H_{\mm,\rS}^\OO/H_{\mm',\rS'}^{\OO'}\right)$. 
\end{rmk}

\subsection{Cardinality of ray class groups of orders} \label{subsec:53} 

The following result gives a formula for the ``class number'' of the ray class group with level datum $(\OO; \mm, \rS)$. It generalizes a formula given by Neukirch \cite[Thm.\ I.12.12]{Neukirch13} for the cardinality of the Picard group of an order (corresponding to level datum $(\OO; \OO, \emptyset)$) as well as a formula given by Cohn \cite[Thm.\ 7.2.7]{Cohn94} for the ``unit ray class number'' (corresponding to level datum $(\OO_K; \mm,\rS)$).

\begin{thm}\label{thm:neukirch2}
Let $K$ be an algebraic number field having $r$ real places and $s$ conjugate pairs of complex places.
Let $\OO_K$ be the maximal order, and let $\left(\OO;\mm,\rS\right)$ be a level datum of $K$.

The groups $\frac{\OO_K^\times}{\U_{\mm,\rS}\left(\OO\right)}$ and $\Cl_{\mm, \rS}(\OO)$ are finite, and one has
\begin{equation}\label{eq:neukirch2.1}
\# \Cl_{\mm, \rS} (\OO) = \frac{h_K}{[\OO_K^\times : \U_{\mm,\rS}\left(\OO\right)]} \cdot \frac{2^{\abs{\rS}}\#\!\left(\OO_K/\!\colonideal{\mm}{\OO_K}\right)^\times}{\#\U_{\mm}\!\left(\OO/\!\colonideal{\mm}{\OO_K}\right)}
\end{equation}
where $h_{K}$ is the class number of $K$. In particular, one has
\begin{equation}\label{eq:neukirch2.2}
\rank\!\left(\U_{\mm,\rS}\left(\OO\right)\right) = \rank(\OO_K^{\times}) = r+s-1.
\end{equation} 
\end{thm} 
\begin{proof}
We specialize the short exact sequence in \Cref{thm:exseq}, given $\mm \subseteq \OO$, 
for the level data  $\ddD = (\OO; \mm, \rS)$ and $\ddD'= (\OO_K; \OO_K, \emptyset)$.
We choose $\dd = \colonideal{\mm}{\OO_K}$. All the hypotheses of  \Cref{thm:exseq} are satisfied, since $\mm\OO_K \subseteq \OO_K =\mm'$, and we obtain
\begin{equation}\label{eq:exactspecial}
1 \to \frac{\OO_K^\times}{\U_{\mm,\rS}\!\left(\OO\right)} \to \frac{\left(\OO_K/\!\colonideal{\mm}{\OO_K}\right)^\times}{\U_{\mm}\!\left(\OO/\!\colonideal{\mm}{\OO_K}\right)} \times \{\pm 1\}^{|\rS|} \to \Cl_{\mm,\rS}(\OO) \to \Cl(\OO_K) \to 1.
\end{equation}
The second nontrivial term in this exact sequence is clearly finite, and the fourth term is finite of order $h_K$ by the finiteness of the class group. It follows that the other two terms are finite. 
Equation \eqref{eq:neukirch2.2} follows from the finiteness of the first term and Dirichlet's Unit Theorem.
Moreover, the alternating product of the cardinality of the terms in an exact sequence of finite groups is $1$. Writing this product and solving for $\# \Cl_{\mm, \rS} (\OO)$ yields \eqref{eq:neukirch2.1}.
\end{proof}

\subsection{Ring class groups of orders} \label{subsec:54}

\Cref{thm:exseq} allows us to express the (wide) ring class group $\Cl(\OO) = \Cl_{\OO,\emptyset}(\OO)$
of an order $\OO$, as defined in \Cref{defn:ring-class-order}, 
as a quotient of the Takagi ray class group $\Cl_{\ff}(\OO_k)=\Cl_{\ff(\OO), \emptyset}(\OO_K)$ (of the maximal order)
for the conductor ideal $\ff(\OO)$, permitting us to quantify the difference between them.

To understand the structure of the (wide) ring class group $\Cl(\OO)$, take $\ddD=(\OO; \OO, \emptyset)$ and $\ddD'=(\OO_K; \OO_K, \emptyset)$. 
Now choose $\dd = \ff = \ff(\OO)$. In particular, since $\dd = \colonideal{\OO}{\OO_K}$, 
this is the case already used in \eqref{eq:exactspecial} with the further specialization $(\mm,\rS)=(\OO,\emptyset)$.
\Cref{thm:exseq} applies to give the exact sequence
\begin{equation}
1 \to \frac{\OO_K^\times}{\OO^\times} \to \frac{\left(\OO_K/\ff\right)^\times}{\left(\OO/\ff\right)^\times} \to \Cl(\OO) \to \Cl(\OO_K) \to 1.
\end{equation}
For the Takagi ray class group, choose instead
$\ddD = (\OO_K; \ff(\OO), \emptyset)$ and $\ddD' = (\OO_K; \OO_K, \emptyset)$. 
Again set $\dd = \ff= \ff(\OO)$, and note  $\dd \subseteq (\ff(\OO): \OO_K)=\ff(\OO)$. Note that the unit group $\U_\ff(\OO_K/\ff) = \{1\}$, so we obtain the exact sequence
\begin{equation}
1 \to \frac{\OO_K^\times}{\U_\ff(\OO_K)} \to \left(\OO_K/\ff\right)^\times \to \Cl_{\ff}(\OO_K) \to \Cl(\OO_K) \to 1.
\end{equation}
There are natural quotient maps making the following diagram commute
(because $\U_\ff(\OO_K)$ is a subgroup of $\OO^\times$), 
so there is a surjective induced map $\psi$ from the ray class group to the ring class group.
\begin{equation}\label{eqn:ring-class-group-seq} 
\begin{tikzcd}
1 \arrow[r] & \frac{\OO_K^\times}{\U_\ff(\OO_K)} \arrow[r] \arrow[d, two heads] & \left(\OO_K/\ff\right)^\times \arrow[r] \arrow[d, two heads] & \Cl_{\ff}(\OO_K) \arrow[r] \arrow[d, two heads, dashed, "\psi"] & \Cl(\OO_K) \arrow[r] \arrow[d, equal] & 1 \\
1 \arrow[r] & \frac{\OO_K^\times}{\OO^\times} \arrow[r] & \frac{\left(\OO_K/\ff\right)^\times}{\left(\OO/\ff\right)^\times} \arrow[r] & \Cl(\OO) \arrow[r] & \Cl(\OO_K) \arrow[r] & 1
\end{tikzcd}
\end{equation}

In the next section, we generalize this comparison by introducing a ray class modulus.

\section{Ray class fields of orders}\label{sec:6}

In this section, we define and prove existence of class fields associated to the ray class groups of \Cref{sec:group}.
As usual, we fix a number field $K$. We will attach a ray class field $H^{\OO}_{\mm,\rS}$ to the level datum $\ddD = (\OO; \mm, \rS)$ for $K$; this field will be an abelian extension of $K$.

%
%
\subsection{Ray class fields of orders defined via Takagi ray class groups}\label{subsec:61}

We define ray class fields of orders by the following recipe. 
We are given the level datum $(\OO; \mm, \rS)$,
and we also consider the level datum $(\OO_K; \colonideal{\mm}{\OO_K}\!, \rS)$.
We define a homomorphism $\psi$ from the group $\Cl_{\colonideal{\mm}{\OO_K},\rS}(\OO_K)$ to the given ray class group $\Cl_{\mm,\rS}(\OO)$
and show that it is surjective.
We let $J = J(\OO; \mm, \rS) = \ker\psi$ be the kernel.
As a subgroup of $\Cl_{\colonideal{\mm}{\OO_K},\rS}(\OO_K)$, $J$ has an associated Takagi ray class (sub)field $L = L_J$ (of $H_{\colonideal{\mm}{\OO_K},\rS}$), 
given in \Cref{thm:what}(1) below. 
We define the field $L$ obtained this way to be the ray class field $H_{\mm, \rS}^{\OO}$ 
assigned to the level datum $(\OO; \mm, \rS)$ for the order $\OO$.

To construct $\psi$, we start from the exact sequence from \Cref{thm:exseq} for $\ddD=(\OO; \mm, \rS)$ and $\ddD' = (\OO_K; \OO_K, \emptyset)$ with and $\dd :=\colonideal{\mm}{\OO_K}$. It is
\begin{equation}
1 
\to \frac{\OO_K^\times}{\U_{\mm,\rS}\left(\OO\right)} 
\to \frac{\left(\OO_K/\colonideal{\mm}{\OO_K})\right)^\times}{\U_{\mm}\left(\OO/\colonideal{\mm}{\OO_K}\right)} \times \{\pm 1\}^{|\rS|} 
\to \Cl_{\mm,\rS}(\OO) 
\to \Cl(\OO_K) 
\to 1,
\end{equation}
since 
$\U_{\mm'}( \OO'/ \dd) = \left(\OO_K/\colonideal{\mm}{\OO_K}\right)^\times$.
Consider also a second exact sequence given by \Cref{thm:exseq}, taking $\ddD = (\OO_K; \colonideal{\mm}{\OO_K}\!, \rS)$, 
$\ddD' = (\OO_K; \OO_K, \emptyset)$, and $\dd = \colonideal{\mm}{\OO_K}$. In this exact sequence, the denominator in the second term is $\U_{\colonideal{\mm}{\OO_K}}(\OO_K/\colonideal{\mm}{\OO_K})$
(since $\colonideal{\colonideal{\mm}{\OO_K}}{\OO_K} = \colonideal{\mm}{\OO_K}$), which is the trivial group.
We obtain
\begin{equation}
1 \to \frac{\OO_K^\times}{\U_{\colonideal{\mm}{\OO_K},\rS}\!\left(\OO_K\right)} 
\to \left(\OO_K/\colonideal{\mm}{\OO_K}\right)^\times \times \{\pm 1\}^{|\rS|} 
\to \Cl_{\colonideal{\mm}{\OO_K},\rS}(\OO_K) 
\to \Cl(\OO_K) 
\to 1.
\end{equation}

As in \Cref{subsec:53}, 
there are natural quotient maps between the objects in these two exact sequences corresponding to 
the downward maps labeled $\kappa$, $\pi$, and $\id$
in the following diagram. The map $\kappa$ exists and is surjective because $\U_{\colonideal{\mm}{\OO_K},\rS}\left(\OO_K\right)$ is a subgroup of $\U_{\mm,\rS}\left(\OO\right)$; the map $\pi$ is a quotient in the first coordinate; and the map $\id$ is the identity map. 
Moreover, it is straightforward to check that the diagram commutes, which implies that there is an induced surjective map $\psi$ in the position shown.
\begin{equation}\label{eqn:psi}
\begin{tikzcd}[column sep=small]
1 \arrow[r] 
& \frac{\OO_K^\times}{\U_{\colonideal{\mm}{\OO_K},\rS}\left(\OO_K\right)} \arrow[r] \arrow[d, two heads, "\kappa"] 
& \left(\OO_K/\colonideal{\mm}{\OO_K}\right)^\times \times \{\pm 1\}^{|\rS|} \arrow[r] \arrow[d, two heads, "\pi"] 
& \Cl_{\colonideal{\mm}{\OO_K},\rS}(\OO_K) \arrow[r] \arrow[d, two heads, dashed, "\psi"] 
& \Cl(\OO_K) \arrow[r] \arrow[d, equal, "\id"] 
& 1
\\
1 \arrow[r] 
& \frac{\OO_K^\times}{\U_{\mm,\rS}\left(\OO\right)} \arrow[r] 
& \frac{\left(\OO_K/\colonideal{\mm}{\OO_K}\right)^\times}{\U_{\mm}\left(\OO/\colonideal{\mm}{\OO_K}\right)} \times \{\pm 1\}^{|\rS|} \arrow[r] 
& \Cl_{\mm,\rS}(\OO) \arrow[r] 
& \Cl(\OO_K) \arrow[r] 
& 1 
\end{tikzcd}
\end{equation}
We have thus constructed a surjective map $\psi : \Cl_{\colonideal{\mm}{\OO_K},\rS}(\OO_K) \surj \Cl_{\mm,\rS}(\OO)$.

We make the following definition.
\begin{defn}\label{defn:61} 
The \textit{ray class field of the order $\OO$ with modulus $(\mm,\rS)$} is the subfield $H_{\mm,\rS}^\OO$ of
the Takagi ray class field $H_{\colonideal{\mm}{\OO_K},\rS}^{\OO_K}$ associated to $J(\OO; \mm, \rS) := \ker \psi$
in \eqref{eqn:psi} 
 under the Galois correspondence between subgroups of the Galois group $\Gal(H_{\colonideal{\mm}{\OO_K},\rS}^{\OO_K}/K) \isom \Cl_{\colonideal{\mm}{\OO_K},\rS}(\OO_K)$ and subfields of $H_{\colonideal{\mm}{\OO_K},\rS}^{\OO_K}$ containing $K$.
\end{defn}

\Cref{thm:main1} will identify $H_{\mm, \rS}^{\OO}$ in terms of data associated to
the splitting of primes in $\OO_K$, their contractions to $\OO$, and a ray class congruence condition on those contractions.
Namely, we will show the field $L=H_{\mm, \rS}^{\OO}$ produced by this definition is the unique extension field of $K$ whose set of prime ideals $\pp$ over $\OO_K$ that split completely in $L/K$
agrees (with symmetric difference a finite set) with the set of prime ideals $\pp$ of $\OO_K$ whose contraction $\pp \cap \OO$
to $\OO$ is a principal prime ideal $\pi\OO$ having a generator $\pi \equiv_{\OO} 1 \Mod{\mm}$ and $\rho(\pi) >0$
for $\rho \in \rS$.

%
%
\subsection{The classical existence theorem}\label{subsec:62}
We state the classical existence theorem of class field theory, called the Weber--Hilbert--Artin--Takagi [WHAT] correspondence by Cohn \cite[Chap.\ 7]{Cohn94}, in our notation.

\begin{thm}[WHAT correspondence]\label{thm:what}
Let $K$ be a number field, $\mm$ an ideal of $\OO_K$, and $\rS$ a subset of the set of real embeddings of $K$. 
\begin{enumerate}
\item[(1)]
{\rm (Weber--Takagi)} Let $J$ be a subgroup of the (Takagi) ray class group $\Cl_{\mm,\rS}(\OO_K)$. Then, there is a unique abelian extension $L_J/K$ with the property that a prime
ideal $\pp$ of $\OO_K$ splits completely in $L_J$ if and only if the ray class $[\pp]$ lies in $J$, with finitely many exceptions.
(The exceptions are among the prime ideals dividing $\mm$.) 
\begin{enumerate}
\item
If $J_1 \subseteq J_2$, then $L_{J_2} \subseteq L_{J_1}$.
\item
For $J =\{\II\}$, where $\II = [\OO_K]$ is the principal ray class modulo $(\mm, \rS)$, the field $L= L_{\{\II\} }= H_{\mm, \rS}^{\OO_K}$
is the principal ray class field. 
\end{enumerate}
\item[(2)]
{\rm (Artin)} Under the correspondence (1), for $L=L_{\{\II\}}= H_{\mm, \rS}^{\OO_K}$ 
there is an isomorphism $\Art: \Cl_{\mm,\rS}(\OO_K) \to \Gal (L/K)$, 
the so-called Artin isomorphism $\Art= \Art_{\mm, \rS}$, which is 
determined by sending prime ideals $[\pp]$ of $\OO_K$ (coprime to $\mm$)
to $ \left[ \frac{L/K}{\pp} \right] \in \Gal(L/K)$, 
and extending this map multiplicatively to all ray ideals $[\aa]$ coprime to $\mm$.
Under this isomorphism, $L_J$ is the fixed field of the principal ray class field $L=L_{\{\II\}}$
under the action of the group of automorphisms $\Art(J) \subseteq \Gal(L/K)$; that is,
\begin{equation}\label{eqn:WHAT-2} 
L_J := \left(H_{\mm,\rS}^{\OO_K} \right)^{\Art(J)}.
\end{equation} 
\end{enumerate}
In the statement of (2), the Artin symbol 
$\left[ \frac{L/K}{\pp} \right]$ denotes the Frobenius automorphism $\sigma_{\pp} \in \Gal(L/K)$
computed for a prime ideal $\PP$ of $\OO_L$ lying over $\pp$ as $\sigma_p(x) \equiv x^{q}  \Mod{\PP}$
for all $x \in \OO_L$, where $q=p^j$ is the number of elements in the finite field $\OO_K/\pp$. 
\end{thm}

\begin{proof}
This statement (1) is extracted from \cite[Thm.\ 7.4.1]{Cohn94}, whereas (2) follows from \cite[Thm.\ 7.4.2]{Cohn94}.
See also \cite[Thm.~6.6.8]{Halter-Koch:22} for (1) and \cite[Thm.~and Defn.~6.6.2]{Halter-Koch:22} for (2).
\end{proof} 

\begin{rmk}\label{rmk:artin}
For $J$ a subgroup of $\Cl_{\mm,\rS}(\OO_K)$, with $L = H^{\OO_K}_{\mm,\rS}$ and $L_J = L^{\Art(J)}$ as in \Cref{thm:what}, one can also define an Artin isomorphism $\Art_J$ making the diagram
\begin{equation}
\begin{tikzcd}
\Cl_{\mm,\rS}(\OO_K) \ar[r, "\Art_{\mm,\rS}"] \ar[d, two heads] & \Gal(L/K) \ar[d, two heads] \\
\Cl_{\mm,\rS}(\OO_K)/J \ar[r, "\Art_J"] & \Gal(L_J/K)
\end{tikzcd}
\end{equation}
commute, as a direct consequence of the Galois correspondence. Moreover, the Artin maps with different class field moduli are compatible because of their local definition via the Frobenius map. That is, if $\phi : \Cl_{\mm,\rS}(\OO_K) \to \Cl_{\mm',\rS'}(\OO_K)$ is the natural quotient map, then $\Art_{\ker \phi} = \Art_{\mm',\rS'}$.
\end{rmk}

\subsection{Proof of \Cref{thm:main1}}\label{subsec:63}

The main step in the proof is to identify the map $\psi$ introduced in \Cref{subsec:61}
with contraction to $\OO$
on the set of fractional ideals of $\OO_K$ coprime to $\colonideal{\mm}{\OO_K}$ (or equivalently, coprime to $\mm\OO_K \cap \ff(\OO_K)$).

\begin{lem}\label{lem:psicon}
For a class $[\aa] \in \Cl_{\colonideal{\mm}{\OO_K}, \rS}(\OO_K)$ represented by some $\aa \in \rJ_{\colonideal{\mm}{\OO_K}, \rS}(\OO_K)$, the map ${\psi: \Cl_{\colonideal{\mm}{\OO_K},\rS}(\OO_K) \to \Cl_{\mm,\rS}(\OO)}$ 
defined by \eqref{eqn:psi} may be explicitly written $\psi([\aa]) = [\con(\aa)]$. 
\end{lem}
\begin{proof}
We apply \Cref{prop:conext} for the case $\OO \subseteq \OO_K$, taking $\mm':= \colonideal{\mm}{\OO_K}$ and noting that $\colonideal{\mm}{\OO_K} \subseteq \ff(\OO)$; see \Cref{lem:basicinclusions}. \Cref{prop:conext} constructed an isomorphism
$\con : \rJ_{\colonideal{\mm}{\OO_K}}(\OO_K) \to \rJ_{\colonideal{\mm}{\OO_K}}(\OO)$ 
and showed that its inverse map was $\ext : \rJ_{\colonideal{\mm}{\OO_K}}(\OO) \to \rJ_{\colonideal{\mm}{\OO_K}}(\OO_K)$.
For a principal ideal $\aa \in \rP_{\colonideal{\mm}{\OO_K},\rS}(\OO_K)$ with $\aa = \alpha\OO_K$, note that $\aa = \ext(\alpha\OO)$; thus, $\con(\aa) = \con(\ext(\alpha\OO)) = \alpha\OO$.
Thus, the composition 
\begin{equation}
\begin{tikzcd}
\rJ_{\colonideal{\mm}{\OO_K}}(\OO_K) \arrow[r,"\con"] & \rJ_{\colonideal{\mm}{\OO_K}}(\OO) \arrow[r,hook] & \rJ_{\mm}(\OO)
\end{tikzcd}
\end{equation}
sends the subgroup $\rP_{\colonideal{\mm}{\OO_K},\rS}(\OO_K)$ to a subgroup of $\rP_{\mm,\rS}(\OO)$ and thus defines a map
\begin{equation}
\widetilde{\psi} : \Cl_{\colonideal{\mm}{\OO_K},\rS}(\OO_K) \to \Cl_{\mm,\rS}(\OO).
\end{equation}
To show that $\widetilde{\psi} = \psi$, it suffices to show that $\widetilde{\psi}$ makes the diagram in \eqref{eqn:psi} commute. On the left, for a pair $(\alpha, \e) \in \left(\OO_K/(\colonideal{\mm}{\OO_K})\right)^\times \times \{\pm 1\}^{|\rS|}$, the square looks like
\begin{equation}
\begin{tikzcd}
(\alpha,\e) \arrow[r, mapsto] \arrow[d, mapsto] & \{\beta\OO_K : \beta \equiv \alpha \Mod{\colonideal{\mm}{\OO_K}} \mbox{ and } \rho(\beta)=\e_\rho\} \arrow[d, mapsto, "\widetilde{\psi}"] \\
(\alpha \Mod{\U_{\mm}\!\left(\OO/(\colonideal{\mm}{\OO_K})\right)},\e) \arrow[r, mapsto] & \{\beta\OO : \beta \equiv \alpha \Mod{\mm} \mbox{ and } \rho(\beta)=\e_\rho\}, \\
\end{tikzcd}
\end{equation}
and we observe that it commutes. On the right, $\widetilde{\psi}$ clearly does not change the class of $\aa$ in $\Cl(\OO_K)$, so the right square commutes as well. So $\widetilde{\psi} = \psi$, and the lemma is proved.
\end{proof}

\begin{proof}[Proof of \Cref{thm:main1}]
Consider the map $\psi : \Cl_{\colonideal{\mm}{\OO_K},\rS}(\OO_K) \to \Cl_{\mm,\rS}(\OO) $ defined by \eqref{eqn:psi}, and let $J = \ker\psi$.
By \Cref{thm:what}, $H_{\mm,\rS}^{\OO} := \left(H_{\colonideal{\mm}{\OO_K},\rS}\right)^{\Art(J)}$ is the unique abelian extension of $K$ such that a prime $\pp$ of $\OO_K$ splits completely in $H_{\mm,\rS}^{\OO}$ if and only if $[\pp]$ lies in $J$. But $[\pp] \in J$ if and only if $\psi([\pp])=0$, and by \Cref{lem:psicon}, $\psi([\pp]) = [\con(\pp)] = [\pp \cap \OO]$. (Since $\pp$ is a maximal ideal, $\pp \cap \OO$ is also a maximal ideal by \Cref{lem:conextmax}(2).) 
\end{proof}

%
%
\subsection{Proof of \Cref{thm:main2}}\label{subsec:64}
 
We prove a more general result.
\begin{thm}\label{thm:main2b} 
For two orders $\OO \subseteq \OO'$ in a number field $K$ and any level datum $(\OO; \mm, \rS)$, there are inclusions of ray class fields $H_{\mm\OO',\rS}^{\OO'} \subseteq H_{\mm,\rS}^{\OO} \subseteq H_{\colonideal{\mm}{\OO'},\rS}^{\OO'}$. 
\end{thm}
\begin{proof}
Consider the following (commutative) diagram, with the dotted lines denoting maps induced by the others. 
In this diagram, rows are exact, but the columns are not exact; all vertical maps are surjective.
\begin{equation}\label{eq:big-diagram2}
\begin{tikzcd}[column sep=small]
1 \arrow[r] 
& \frac{(\OO')^\times}{\U_{\colonideal{\mm}{\OO'},\rS}\left(\OO'\right)} \arrow[r] \arrow[d, two heads] 
& 
\left(\OO'/\colonideal{\mm}{\OO'}\right)^\times 
\times \{\pm 1\}^{|\rS|} \arrow[r] \arrow[d, two heads] 
& \Cl_{\colonideal{\mm}{\OO'},\rS}(\OO') \arrow[r] \arrow[d, two heads, dashed, "\psi"] 
& \Cl(\OO') \arrow[r] \arrow[d, equal] 
& 1
\\
1 \arrow[r] 
& \frac{(\OO')^\times}{\U_{\mm,\rS}\left(\OO\right)} \arrow[r] \arrow[d, two heads] 
& \frac{\left(\OO'/\colonideal{\mm}{\OO'}\right)^\times}{\U_{\mm}\left(\OO/\colonideal{\mm}{\OO'}\right)} \times \{\pm 1\}^{|\rS|} \arrow[r] \arrow[d, two heads] 
& \Cl_{\mm,\rS}(\OO) \arrow[r] \arrow[d, two heads, dashed, "\phi"] 
& \Cl(\OO') \arrow[r] \arrow[d, equal] 
& 1 
\\
1 \arrow[r] 
& \frac{(\OO')^\times}{\U_{\mm\OO',\rS}\left(\OO'\right)} \arrow[r] 
& 
\left(\OO'/\mm\OO'\right)^\times 
\times \{\pm 1\}^{|\rS|} \arrow[r] 
& \Cl_{\mm\OO',\rS}(\OO') \arrow[r] 
& \Cl(\OO') \arrow[r] 
& 1
\end{tikzcd}
\end{equation}
The horizontal rows are exact sequences from \Cref{thm:exseq} 
with data given in the following table.
\begin{center}
\begin{tabular}{|c||c|c|c|}
\hline
\Cref{thm:exseq} & $\ddD=(\OO; \mm, \rS)$ & $\ddD'=(\OO'; \mm', \rS')$ & $\dd$ \\
\hhline{|=||=|=|=|}
top row & $(\OO'; \colonideal{\mm}{\OO'}, \rS)$ & $(\OO'; \OO', \emptyset)$ & $\colonideal{\mm}{\OO'}$ \\
\hline
middle row & $(\OO; \mm, \rS)$ & $(\OO'; \OO', \emptyset)$ & $\colonideal{\mm}{\OO'}$ \\
\hline
bottom row & $(\OO'; \mm\OO', \rS)$ & $(\OO'; \OO', \emptyset)$ & $\mm\OO'$ \\
\hline
\end{tabular}
\end{center}
In the second nontrivial column,
we have used in all rows that $\U_{\OO'}(\OO'/\dd) = (\OO'/\dd)^\times$ and in the top and bottom rows row that $\U_{\dd}(\OO/\dd) = \{1\}$.

The vertical maps in the first two nontrivial columns are given by quotienting by everything that is $1$ modulo $\mm$ (from the top row to the middle) and then by everything that is $1$ modulo $\mm\OO'$ (from the middle row to the bottom). These maps are well-defined because $\colonideal{\mm}{\OO'} \subseteq \mm \subseteq \mm\OO'$. The commutativity of the leftmost two squares is thus clear. The commutativity of the diagram, excepting the dotted lines, follows by exactness because, on the longer horizontal rectangles, it simply encodes an equality between two zero maps.

The maps denoted by dotted lines are induced, so the whole diagram \eqref{eq:big-diagram2} commutes.
The upper induced map $\psi$ was described in \Cref{subsec:61}.
The lower induced map $\phi$ is 
equal to the map induced by extension of ideals, as can be seen by commutativity of the diagram 
and comparison with the ``change of order'' exact sequence 
\begin{equation}
\begin{tikzcd}[column sep=small]
1 \arrow[r] 
& \frac{\U_{\mm\OO',\rS}(\OO')}{\U_{\mm,\rS}(\OO)} \arrow[r] 
& \frac{\U_{\mm\OO'}\!\left(\OO'/\colonideal{\mm}{\OO'}\right)}{\U_\mm\!\left(\OO/\colonideal{\mm}{\OO'}\right)} \arrow[r] 
& \Cl_{\mm,\rS}(\OO) \arrow[r,"\phi"] 
& \Cl_{\mm\OO',\rS}(\OO') \arrow[r] 
& 1
\end{tikzcd}
\end{equation}
obtained from \Cref{thm:exseq} (taking $\ddD = (\OO; \mm, \rS)$, $\ddD' = (\OO'; \mm\OO', \rS)$, and $\dd :=\colonideal{\mm}{\OO'}$).

Similarly, the composition $\phi \circ \psi$ fits into the ``change of modulus'' exact sequence
\begin{equation}
\begin{tikzcd}[column sep=small]
1 \arrow[r] 
& \frac{\U_{\mm\OO',\rS}(\OO')}{\U_{\colonideal{\mm}{\OO'},\rS}(\OO')} \arrow[r] 
& \U_{\mm\OO'}\!\left(\OO'/\colonideal{\mm}{\OO'}\right) \arrow[r] 
& \Cl_{\colonideal{\mm}{\OO'},\rS}(\OO') \arrow[r,"\phi \circ \psi"] 
& \Cl_{\mm\OO',\rS}(\OO') \arrow[r] 
& 1,
\end{tikzcd}
\end{equation}
also a special case of \Cref{thm:exseq}
(where we take $\ddD = (\OO', \colonideal{\mm}{\OO'}, \rS)$, $\ddD' = (\OO', \mm\OO', \rS)$, and $\dd :=\colonideal{\mm}{\OO'}$).

The diagram \eqref{eq:big-diagram2} establishes (from its third nontrivial column) the sequence of surjections
\begin{equation}
\begin{tikzcd}[column sep=small]
& \Cl_{\colonideal{\mm}{\OO'},\rS}(\OO') \arrow[r,"\psi",two heads] 
& \Cl_{\mm,\rS}(\OO) \arrow[r,"\phi",two heads] 
& \Cl_{\mm\OO',\rS}(\OO').
\end{tikzcd}
\end{equation}
Now by 
\Cref{thm:what}(1),
we have the tower of ray class fields of orders
\begin{equation}
\begin{tikzcd}
H^{\OO'}_{\colonideal{\mm}{\OO'},\rS} \arrow[d, dash, "\ker\psi"] \\
H^{\OO}_{\mm,\rS} \arrow[d, dash, "\ker\phi"] \\
H^{\OO'}_{\mm\OO',\rS} \arrow[d, dash] \\
K
\end{tikzcd}
\end{equation}
with Galois groups as labeled, thus proving the theorem.
\end{proof} 

\begin{proof}[Proof of \Cref{thm:main2}]
This is the special case $\OO'=\OO_K$ of \Cref{thm:main2b}.
\end{proof}

%
%
\subsection{Proof of \Cref{thm:main3}}\label{subsec:65}

We now prove \Cref{thm:main3}, 
giving a form of Artin reciprocity for a ray class group and ray class field of an order. The result is obtained
from the usual Artin reciprocity law together with properties of the map $\psi$ in \eqref{eqn:psi} 
established earlier in this section.

\begin{proof}[Proof of \Cref{thm:main3}]
Let $H_1 = H_{\colonideal{\mm}{\OO_K},\rS}^{\OO_K}$, and let $\Art : \Cl_{\colonideal{\mm}{\OO_K},\rS}(\OO_K) \to \Gal(H_1/K)$ 
be the (usual) Artin map. Let $\psi : \Cl_{\colonideal{\mm}{\OO_K},\rS}(\OO_K) \to \Cl_{\mm,\rS}(\OO)$ 
be the map constructed in \eqref{eqn:psi}. By \Cref{lem:psicon}, for any class 
$[\bb] \in \Cl_{\colonideal{\mm}{\OO_K},\rS}(\OO_K)$, one has $\psi([\bb]) = [\con(\bb)]$.

Let $H_0 = H_{\mm,\rS}^\OO$. By \Cref{defn:61} and the Galois correspondence it describes, 
there is an isomorphism $\Art_\OO$ making the following diagram commute.
\begin{equation}\label{eq:new-Art}
\begin{tikzcd}
\Cl_{\colonideal{\mm}{\OO_K},\rS}(\OO_K) \arrow[r, "\sim" inner sep=.3mm, "\Art" inner sep=1.7mm] \arrow[d, two heads, "\psi"] & \Gal(H_1/K) \arrow[d, two heads] \\
\Cl_{\mm,\rS}(\OO) \arrow[r, "\sim" inner sep=.3mm, "\Art_\OO" inner sep=1.7mm] & \Gal(H_0/K)
\end{tikzcd}
\end{equation}
To give an explicit formula for $\Art_\OO$, consider any fractional ideal $\aa$ of $\OO$ coprime to $\ff(\OO) \cap \mm$ (or, equivalently, coprime to $\colonideal{\mm}{\OO_K}$).
We compute $\Art_\OO([\aa])$ by lifting $[\aa]$ to $\Cl_{\colonideal{\mm}{\OO_K},\rS}(\OO_K)$ along $\psi$.
The coprimality condition implies that $\con(\aa\OO_K) = \con(\ext(\aa)) = \aa$ (by \Cref{prop:conext}), so $\psi([\aa\OO_K])= [\con(\ext(\aa))]=[\aa]$, that is, $[\aa\OO_K]$ is a lift of $[\aa]$. Therefore \eqref{eq:new-Art} gives 
\begin{equation}
\Art_\OO([\aa]) = \left.\Art([\aa\OO_K])\right|_{H_0}.
\end{equation}

Now let $\pp$ be any prime ideal of $\OO$ coprime to $\ff(\OO) \cap \mm$, 
with residue field $\OO/\pp$ having characteristic $p$.
Let $\PP$ be a prime ideal of $\OO_{H_0}$ lying over 
$\pp\OO_K$. We claim that for all $\alpha\in \OO_{H_0}$, 
\begin{equation}\label{eq:artoop}
\Art_\OO([\pp])(\alpha) \equiv \alpha^q  \Mod{\PP},
\end{equation}
where $q=p^j$ is the number of elements in $\OO/\pp$,
identifying $\Art_{\OO}([\pp])$ as a Frobenius automorphism.

To see this, let $\mathcal{P}$ be a prime ideal of $\OO_{H_1}$ lying over $\PP$. By Artin reciprocity (part (2) of \Cref{thm:what}; see also \Cref{rmk:artin}), for any $\alpha \in \OO_{H_0}$,
\begin{equation}
\Art_\OO([\pp])(\alpha) = \Art([\pp\OO_K])|_{H_0}(\alpha)
\end{equation}
Now the condition that $\pp$ is coprime to $\ff(\OO) \cap \mm$ implies that
$\pp\OO_K$ is a prime ideal of $\OO_K$ by \Cref{lem:conextmax} and \Cref{prop:conext}; 
in addition, $\PP$ and $\mathcal{P}$ are unramified over $\pp\OO_K$ because $H_0$ and $H_1$ only have ramification over $K$ at the primes dividing $\ff(\OO) \cap \mm$ (that is, the primes dividing $\colonideal{\mm}{\OO_K}$). 
Now for any $\beta \in \OO_{H_1}$ we have by definition
\begin{equation}
\Art([\pp\OO_K])(\beta))\equiv \beta^{q'} \Mod{\mathcal{P}},
\end{equation}
where $q'$ is the number of elements in $\OO_K/\pp\OO_K$.
We have $q'= q$ since the  hypotheses guarantee that  $\OO/\pp \cong \OO_K/ \pp \OO_K$.
Applying this congruence with $\beta= \alpha \in \OO_{H_0}$, noting that $\mathcal{P} \cap \OO_{H_0} = \PP$
and $\alpha^q \in \OO_{H_0}$,
we may conclude 
\begin{equation}
\Art_\OO([\pp])(\alpha) \equiv \alpha^q \Mod{\PP},
\end{equation}
proving the claim.
 
Finally we note that by \Cref{prop:m-integral-coprime} and \Cref{lem:coloninv}, fractional ideals of $\OO$ coprime to $\ff(\OO) \cap \mm$ factor into powers of prime ideals coprime to $\ff(\OO) \cap \mm$, so
the map $\Art_\OO$ is uniquely determined by \eqref{eq:artoop}.
\end{proof}

\section{Computations of ray class groups of orders}\label{sec:examples}

To illustrate how to compute with the exact sequence for change of order and modulus, we calculate several ray class groups of orders.
Let $\OO$ be an order in a number field $K$, and let $\rho_1, \ldots, \rho_r$ be the real embeddings of $K$.
Consider a ray class modulus $(\mm, \rS)$ for an order $\OO$, 
where $\mm$ is an integral $\OO$-ideal and $\rS = \{\rho_{k_1}, \ldots, \rho_{k_\ell}\}$.
In the following examples, we will abbreviate the pair $(\mm, \rS)$ by the formal product $\mm\infty_{k_1}\cdots\infty_{k_\ell}$.
We will also denote principal ideals such as $\alpha\OO$ by $(\alpha)$.
We consider a second $K$-order $\OO'$ with $\OO \subseteq \OO'$ and a
corresponding modulus $(\mm', \rS')$, requiring $\mm \subseteq \mm'$ and $\rS' \subseteq \rS.$
We apply  \Cref{thm:exseq} with $\ddD=(\OO; \mm, \rS)$ and $\ddD'=(\OO'; \mm', \rS')$,
and make the choice $\dd = \colonideal{\mm}{\OO'}$, to obtain the exact sequence
\begin{equation}\label{eq:71exact} 
1 \to \frac{\U_{\mm',\rS'}\!\left(\OO'\right)}{\U_{\mm,\rS}\!\left(\OO\right)} \to \frac{\U_{\mm'}\!\left(\OO'/\colonideal{\mm}{\OO'}\right)}{\U_{\mm}\!\left(\OO/\colonideal{\mm}{\OO'}\right)} \times \{\pm 1\}^{|\rS \setminus \rS'|} \to \Cl_{\mm,\rS}(\OO) \to \Cl_{\mm',\rS'}(\OO') \to 1.
\end{equation}

The following series of examples treats several different orders in the real quadratic field $K = \Q(\sqrt{2})$ 
and the computation of some of their ray class groups at moduli ramified at 
prime ideals lying over $(7)$ and $\infty$ of $\Z$. We let $\infty_1$, $\infty_2$  denote the real embeddings $\rho_i: \Q(\sqrt{2}) \to \R$
defined by $\rho_1(\sqrt{2}) = \sqrt{2}$ and $\rho_2(\sqrt{2}) = -\sqrt{2},$ respectively.

The first three examples taken together give information on
the corresponding ray class fields.
They show  that a ray class field of an order can be strictly larger than the compositum of the corresponding ray class field of the maximal order with the ring class field. Using \Cref{thm:main3}, the class number calculations imply
\begin{equation}
H^{\Z[\sqrt{2}]}_{(7)\infty_2} \cdot H_{(1)}^{\Z[2\sqrt{2}]} \subsetneq H^{\Z[2\sqrt{2}]}_{(7)\infty_2},
\end{equation}
since
$\left[H^{\Z[\sqrt{2}]}_{(7)\infty_2} H^{\Z[2\sqrt{2}]}_{(1)} : K\right] \le \left[H^{\Z[\sqrt{2}]}_{(7)\infty_2} :K\right] \cdot \left[H^{\Z[2\sqrt{2}]}_{(1)}: K\right] =6$
while $\left[H^{\Z[2\sqrt{2}]}_{(7)\infty_2}: K\right] = 12$. The final example 
presents a calculation in a case where the corresponding ring class field is nontrivial.

\subsection{Example 1}
Take level data $\ddD = \left(\Z[\sqrt{2}]; 7\Z[\sqrt{2}], \{\rho_2\}\right) = \left(\Z[\sqrt{2}]; (7)\infty_2\right)$ and $\ddD' = \left(\Z[\sqrt{2}]; \Z[\sqrt{2}], \emptyset\right) = \left(\Z[\sqrt{2}]; (1)\right)$. 
Then by \eqref{eq:71exact},
\begin{equation}
1 \to \frac{\Z[\sqrt{2}]^\times}{\U_{(7)\infty_2}\!\left(\Z[\sqrt{2}]\right)} \to \left(\Z[\sqrt{2}]/(7)\right)^\times \times \{\pm 1\} \to \Cl_{(7)\infty_2}\!\left(\Z[\sqrt{2}]\right) \to \Cl\!\left(\Z[\sqrt{2}]\right) \to 1.
\end{equation}
The class group $\Cl\!\left(\Z[\sqrt{2}]\right) = 1$.
Thus, the ray class group $\Cl_{(7)\infty_2}\!\left(\Z[\sqrt{2}]\right)$ is isomorphic to the previous term in the sequence modulo the image of global units.
By the Chinese Remainder Theorem,
\begin{align}
\left(\Z[\sqrt{2}]/(7)\right)^\times 
&\isom \left(\Z[\sqrt{2}]/(3+\sqrt{2})\right)^\times \times \left(\Z[\sqrt{2}]/(3-\sqrt{2})\right)^\times \\
&\isom \Z/6\Z \times \Z/6\Z.
\end{align}
Moreover,
\begin{equation}
\frac{\Z[\sqrt{2}]^\times}{\U_{(7)\infty_2}\!\left(\Z[\sqrt{2}]\right)} = \frac{\langle -1, 1+\sqrt{2} \rangle}{\langle (1+\sqrt{2})^6 \rangle} \isom \Z/6\Z \times \Z/2\Z.
\end{equation}
Thus, we see that
\begin{equation}
\Cl_{(7)\infty_2}\!\left(\Z[\sqrt{2}]\right) \isom \Z/6\Z.
\end{equation}

\subsection{Example 2} 
Take $\ddD = \left(\Z[2\sqrt{2}]; 7\Z[2\sqrt{2}], \{\rho_2\}\right) = \left(\Z[2\sqrt{2}]; (7)\infty_2\right)$ and, as above, take $\ddD' = \left(\Z[\sqrt{2}]; \Z[\sqrt{2}], \emptyset\right) = \left(\Z[\sqrt{2}]; (1)\right)$. 
Then \eqref{eq:71exact} gives
\begin{equation}
1 \to \frac{\Z[\sqrt{2}]^\times}{\U_{(7)\infty_2}\!\left(\Z[2\sqrt{2}]\right)} \to \frac{\left(\frac{\Z[\sqrt{2}]}{(14)}\right)^\times}{\U_{(7)}\!\left(\frac{\Z[2\sqrt{2}]}{(14\Z+14\sqrt{2}\Z)}\right)} \times \{\pm 1\} \to \Cl_{(7)\infty_2}\!\left(\Z[2\sqrt{2}]\right) \to \Cl\!\left(\Z[\sqrt{2}]\right) \to 1.
\end{equation}
As in the previous example, $\Cl\!\left(\Z[\sqrt{2}]\right) = 1$, so the ray class group $\Cl_{(7)\infty_2}\!\left(\Z[2\sqrt{2}]\right)$ is isomorphic to the previous term in the sequence modulo the image of global units.
By the Chinese Remainder Theorem, we have
\begin{align}
\left(\Z[\sqrt{2}]/(14)\right)^\times 
&\isom \left(\Z[\sqrt{2}]/(3+\sqrt{2})\right)^\times \times \left(\Z[\sqrt{2}]/(3-\sqrt{2})\right)^\times \times \left(\Z[\sqrt{2}]/(\sqrt{2})^2\right)^\times \\
&\isom \Z/6\Z \times \Z/6\Z \times \Z/2\Z.
\end{align}
We also have
\begin{equation}
\U_{(7)}\!\left(\Z[2\sqrt{2}]/(14\Z+14\sqrt{2}\Z)\right) = 1
\end{equation}
and
\begin{equation}
\frac{\Z[\sqrt{2}]^\times}{\U_{(7)\infty_2}\!\left(\Z[2\sqrt{2}]\right)} = \frac{\langle -1, 1+\sqrt{2} \rangle}{\langle (1+\sqrt{2})^6 \rangle} \isom \Z/6\Z \times \Z/2\Z.
\end{equation}
We see that 
\begin{equation}
\abs{\Cl_{(7)\infty_2}\!\left(\Z[2\sqrt{2}]\right)} = 12.
\end{equation}
With more careful accounting, we can obtain an isomorphism
\begin{equation}
\Cl_{(7)\infty_2}\!\left(\Z[2\sqrt{2}]\right) \isom \Z/6\Z \times \Z/2\Z.
\end{equation}

\subsection{Example 3} 
We compute the same group $\Cl_{(7)\infty_2}\!\left(\Z[2\sqrt{2}]\right)$ a different way.
Take $\ddD = \left(\Z[2\sqrt{2}]; (7)\infty_2\right)$ as above, but instead take $\ddD' = \left(\Z[2\sqrt{2}]; (1)\right)$. 
Then \eqref{eq:71exact} gives
\begin{equation}
1 \to \frac{\Z[2\sqrt{2}]^\times}{\U_{(7)\infty_2}\!\left(\Z[2\sqrt{2}]\right)} \to \left(\Z[2\sqrt{2}]/(7)\right)^\times \times \{\pm 1\} \to \Cl_{(7)\infty_2}\!\left(\Z[2\sqrt{2}]\right) \to \Cl\!\left(\Z[2\sqrt{2}]\right) \to 1.
\end{equation}
The ring class group $\Cl\!\left(\Z[2\sqrt{2}]\right) = 1$.
By the Chinese Remainder Theorem,
\begin{align}
\left(\Z[2\sqrt{2}]/(7)\right)^\times 
&\isom \left(\Z[2\sqrt{2}]/(1+2\sqrt{2})\right)^\times \times \left(\Z[2\sqrt{2}]/(1-2\sqrt{2})\right)^\times \\
&\isom \Z/6\Z \times \Z/6\Z.
\end{align}
The quotient of global unit groups is
\begin{equation}
\frac{\Z[2\sqrt{2}]^\times}{\U_{(7)\infty_2}\!\left(\Z[2\sqrt{2}]\right)} = \frac{\langle -1, (1+\sqrt{2})^2 \rangle}{\langle (1+\sqrt{2})^6 \rangle} \isom \Z/2\Z \times \Z/3\Z \isom \Z/6\Z.
\end{equation}
It follows that
\begin{equation}
\Cl_{(7)\infty_2}\!\left(\Z[2\sqrt{2}]\right) \isom \Z/6\Z \times \Z/2\Z.
\end{equation}

\subsection{Example 4} 
Take $\ddD = \left(\Z[5\sqrt{2}]; (7)\infty_2\right)$ and $\ddD' = \left(\Z[5\sqrt{2}]; (1)\right)$. 
Then \eqref{eq:71exact} gives
\begin{equation}
1 \to \frac{\Z[5\sqrt{2}]^\times}{\U_{(7)\infty_2}\!\left(\Z[5\sqrt{2}]\right)} \to \left(\Z[5\sqrt{2}]/(7)\right)^\times \times \{\pm 1\} \to \Cl_{(7)\infty_2}\!\left(\Z[5\sqrt{2}]\right) \to \Cl\!\left(\Z[5\sqrt{2}]\right) \to 1.
\end{equation}
In a similar method to the above examples (e.g., by another change of order calculation), we can compute
\begin{align}
\Cl\!\left(\Z[5\sqrt{2}]\right) := \Cl_{(1)}\!\left(\Z[5\sqrt{2}]\right) &\isom \Z/2\Z.
\end{align} 
The ring class number is $2$.
As a consequence, by \Cref{thm:main3},
the ring class field for $\Z[5\sqrt{2}]$ is a quadratic extension of $K$, so it is a degree $4$ extension of $\Q$.
We also have 
\begin{align} 
\left(\Z[5\sqrt{2}]/(7)\right)^\times \times \{\pm 1\} &\isom \Z/6\Z \times \Z/6\Z \times \Z/2\Z; \\
\frac{\Z[5\sqrt{2}]^\times}{\U_{(7)\infty_2}\!\left(\Z[5\sqrt{2}]\right)}
= \frac{\langle -1, (1+\sqrt{2})^3 \rangle}{\langle (1+\sqrt{2})^6 \rangle} &\isom \Z/2\Z \times \Z/2\Z.
\end{align}
We see that
\begin{equation}
\abs{\Cl_{(7)\infty_2}\!\left(\Z[5\sqrt{2}]\right)} = \frac{(6 \cdot 6 \cdot 2)(2)}{2 \cdot 2} = 36.
\end{equation}
To determine the ray class group structure, it would be necessary to compute the maps in the exact sequence.

\section{Concluding remarks} \label{sec:final} 
The main results of this paper assigned ray class fields to invertible ray class groups of orders defined for level data specifying an order and a modulus (an ideal and a set of real places). The constructed class fields were characterized in terms of splitting conditions on their prime ideals.

There are several other aspects of class field theory, specified by the main theorems of class field theory listed by Hasse \cite{Hasse:67}, 
that might have analogues in the ray class field theory of orders.
\begin{enumerate}
\item
{\em Norms of ideals in orders.} The Takagi class field theory has an interpretation of ray class groups 
in which the kernels of some group maps involve groups generated by norms of ideals. 
There is a natural way to define norms of integral ideals $\aa$ in orders, as $\Nm_{\OO}(\aa) = \abs{\OO/\aa}$.
Norms of fractional ideals are discussed in \Cref{appendix:norms}.
One subtlety that arises is that norms of non-invertible ideals are in general not multiplicative.
\item
{\em Zeta functions of orders.} Zeta functions played an important role in
the historical development of class field theory. We can associate zeta functions and $L$-functions to orders.
There is a zeta function whose terms involve norms of invertible integral ideals of orders. There
is another zeta function obtained by summing over 
norms of a larger set of ideals of an order, including non-invertible ideals. 
The uniqueness of primary decomposition (\Cref{prop:221}) implies that both of these 
zeta functions have Euler products; the latter one will have unusual factors at the finite set of maximal ideals of $\OO$ that 
contain the conductor ideal $\ff(\OO)$.
\item
{\em Ray class monoids and $L$-functions.}
If we allow non-invertible ideals, then the ray class group is enlarged to become a 
{\em ray class monoid}. We treat aspects of the structure of ray class monoids of orders in \cite{KLmonoid}. Additionally one can define $L$-functions associated to characters of ray class groups of orders.
If one allows non-invertible ideals, then one can also allow new $L$-functions using characters of ray class monoids. See \cite{CliffordP61, CliffordP67, Hill67, MasudaQS15, McAlister68, McAlister72} for the character theory of monoids and semigroups.
\end{enumerate}

There is a development of class field theory for orders 
inside an id\`elic framework formulated by Pengo \cite{Pengo20} and
detailed by Campagna and Pengo \cite[Sec.\ 4]{CampagnaP22}
(see in particular their Defn.\ 4.1). It remains to relate that id\`elic definition to the ray class group defined in \Cref{defn:61} of this paper.
The equivalence of the two notions of class fields of orders (in the case of no ramification at infinity considered by Campagna and Pengo)
has not been definitively established (but see \Cref{rmk:raystructure}). \\

\paragraph{\bf Acknowledgments.}
The authors thank Pete L. Clark, Francesco Campagna and Riccardo Pengo, and John Voight for helpful comments and references.
The first author was partially supported by the NSF grant DMS-2302514 and by University of Bristol, the Heilbronn Institute for Mathematical Research, and Purdue University. He thanks Trevor Wooley for support and for helpful mathematical conversations.
The second author was partially supported by NSF grant DMS-1701576.
\appendix

\section{Norms of ideals in orders of number fields}\label{appendix:norms}

Let $K$ be a number field and $\OO$ an order in $K$.
We give a criterion for multiplicativity of (absolute) norms of integral ideals of an order to hold; it does not hold in general.
We use this criterion to extend the notion of norm of an integral ideal of $\OO$ to norm of a fractional ideal of $\OO$.
We discuss the effect of change of order on norms of fractional ideals.

\begin{defn}\label{defn:norm}
Let $\aa$ be an integral $\OO$-ideal. If $\aa$ is nonzero, define the \textit{norm} of $\aa$ to be 
$\Nm_\OO(\aa) = \left[\OO : \aa\right]$, where $\left[\OO : \aa\right] = \abs{\OO/\aa}$ is the index of $\aa$ in $\OO$ as an abelian group.
Define $\Nm_\OO(0)=0$.
\end{defn}

For invertible ideals $\aa, \bb$, the norm is multiplicative: $\Nm(\aa\bb) = \Nm_{\OO}(\aa)\Nm_{\OO}(\bb)$.
The following proposition shows the stronger result that the norm of the product of two integral ideals is multiplicative 
whenever one of them
is invertible.

\begin{prop}\label{prop:normmult}
Let $\aa \in \rI^\ast(\OO)$ and $\bb \in \rI(\OO)$ (so $\aa$ is invertible, whereas $\bb$ may or may not be invertible).
Then, $\Nm_{\OO}(\aa\bb) = \Nm_{\OO}(\aa)\Nm_{\OO}(\bb)$.
\end{prop}
\begin{proof}
If $\bb=0$, then $\Nm_{\OO}(\aa\bb) = 0 = \Nm_{\OO}(\aa)\Nm_{\OO}(\bb)$. Assume from now on that $\bb \neq 0$.

The norm of $\aa\bb$ and the norm of $\aa$ are related by the following equation:
\begin{equation}\label{eq:norm1}
\Nm_{\OO}(\aa\bb) = [\OO : \aa\bb] = [\OO : \aa] [\aa : \aa\bb] = \Nm_{\OO}(\aa)[\aa : \aa\bb].
\end{equation}
Let $\{\pp_1, \ldots, \pp_k\}$ be the set of maximal ideals of $\OO$ containing $\aa\bb$.
Using primary decomposition (\Cref{prop:221}), we may write
\begin{align}
\aa &= \bigcap_{j=1}^k \qq_j = \prod_{j=1}^k \qq_j, & \bb &= \bigcap_{j=1}^k \rr_j = \prod_{j=1}^k \rr_j,
\end{align}
where $\qq_j$ and $\rr_j$ are separately either primary ideals having radical $\pp_j$, or else equal to the unit ideal $\OO$.
Locally, $\aa\OO_{\pp_j} = \qq_j\OO_{\pp_j}$ and $\bb\OO_{\pp_j} = \rr_j\OO_{\pp_j}$.
Moreover since $\aa$ is invertible, by \Cref{prop:lociso} it  is locally principal, so we may write $\aa\OO_{\pp_j} = \alpha_j\OO_{\pp_j}$ for $1 \leq j \leq k$. 
Choose some $\alpha \in \OO$ such that $\alpha \equiv \alpha_j \Mod{\qq_j \rr_j}$ for $1 \leq j \leq k$.
These conditions imply $\alpha \in \aa$,
Define an additive group homomorphism (indeed, an isomorphism of $\OO$-modules)
\begin{equation}
\phi : \OO \to \aa/\aa\bb
\end{equation}
by $\phi(x) = \alpha x + \aa\bb$.

We first show that $\phi$ is surjective. Consider $y \in \aa$. Locally in $\OO_{\pp_j}$, write $y = \alpha_j x_j$ for some $x_j \in \OO_{\pp_j}$. Choose some $x \in \OO$ such that $x \equiv x_j \Mod{\qq_j\rr_j}$ for $1 \leq j \leq k$. Thus, $y \equiv \alpha x \Mod{\qq_j\rr_j}$ for $1 \leq j \leq k$, so
\begin{equation}
y - \alpha x \in \bigcap_{j=1}^k \qq_j\rr_j = \prod_{j=1}^k \qq_j\rr_j = \aa\bb.
\end{equation}
That is, $\phi(x) = y + \aa\bb$.

We now compute the kernel of $\phi$.
We have $\phi(x)=0$ if and only if $\alpha x \in \aa\bb$. 
Now $\alpha \in \aa$, so $\alpha x \in \aa\bb$ whenever $x \in \bb$.
Conversely, suppose $\alpha x \in \aa\bb$. Then, in the local ring $\OO_{\pp_j}$, we have
$\alpha x \in \aa \bb \OO_{\pp_j} = \alpha_j\rr_j\OO_{\pp_j}$. Also, 
$\alpha-\alpha_j \in \qq_{j}\rr_{j}\OO_{\pp_j} = \alpha_{j}\rr_{j}\OO_{\pp_j}$, and thus $\alpha_j x = \alpha x - (\alpha -\alpha_j)x \in \alpha_j\rr_j\OO_{\pp_j}$. Dividing, $x \in \rr_j\OO_{\pp_j}$, so $x \in \rr_j$ (because $x \in \OO$). Thus,
\begin{equation}
x \in \bigcap_{j=1}^k \rr_j = \bb.
\end{equation}
So $\ker \phi = \bb$. 

By the first isomorphism theorem, there is an isomorphism of abelian groups (indeed, of $\OO$-modules)
\begin{equation}
\OO/\bb \isom \aa/\aa\bb.
\end{equation}
Equating the sizes of the two abelian groups, $[\aa : \aa\bb] = [\OO : \bb] = \Nm_\OO(\bb)$. Substituting into \eqref{eq:norm1}, $\Nm_\OO(\aa\bb) = \Nm_\OO(\aa)\Nm_\OO(\bb)$.
\end{proof}

The norm also enjoys a multiplicativity property for coprime ideals, allowing the norm to be calculated 
from its values on primary ideals.
\begin{prop}\label{prop:multcoprime}
If $\aa,\bb$ are coprime integral $\OO$-ideals, then $\Nm_{\OO}(\aa\bb) = \Nm_{\OO}(\aa)\Nm_{\OO}(\bb)$. In particular, if $\mm$ is any nonzero integral $\OO$-ideal with primary decomposition
\begin{equation}
\mm  = \bigcap_{i=1}^n \qq_i = \prod_{i=1}^n \qq_i,
\end{equation}
then its norm is given by $\Nm_{\OO}(\mm) =  \prod_{i=1}^n \Nm_{\OO}(\qq_i)$.
\end{prop}
\begin{proof}
When $\aa$ or $\bb$ is zero, $\Nm_{\OO}(\aa\bb) = 0 = \Nm_{\OO}(\aa)\Nm_{\OO}(\bb)$. Otherwise, $\aa\bb = \aa \cap \bb$, and by the Chinese Remainder Theorem \cite[Ch.~7, Thm.~17]{DF:04},
\begin{equation}
\OO/(\aa\bb) = \OO/(\aa \cap \bb) \isom \OO/\aa \oplus \OO/\bb,
\end{equation}
so $\Nm_{\OO}(\aa\bb) = \abs{\OO/(\aa\bb)} = \abs{\OO/\aa} \cdot \abs{\OO/\bb} = \Nm_{\OO}(\aa)\Nm_{\OO}(\bb)$.

Primary ideals associated to different maximal ideals are coprime, so the norm of an nonzero integral $\OO$-ideal is the product of the norms of its primary constituents. 
\end{proof}

The norm need not be multiplicative if both ideals $\aa$, $\bb$ are non-invertible and non-coprime. Marseglia \cite{Marseglia20b} observed in a single order $\OO$ instances of both strict submultiplicativity $\Nm_{\OO}(\aa \bb) < \Nm_{\OO}(\aa)\Nm_{\OO}(\bb)$ and strict supermultiplicativity $\Nm_{\OO}(\aa\bb) > \Nm_{\OO}(\aa) \Nm_{\OO}(\bb)$.

\begin{eg}[Super-multiplicativity and sub-multiplicativity of norms] 
The multiplicativity condition $\Nm_{\OO}(\aa\bb) = \Nm_{\OO}(\aa) \Nm_{\OO}(\bb)$ for $\aa, \bb \in \rI(\OO)$ can only fail when both $\aa$ and $\bb$ are not invertible.
In particular, neither $\aa$ nor $\bb$ can be coprime to the conductor $\ff(\OO)$.

\begin{itemize}
\item[(1)] An example of strict super-multiplicativity of norms was given in \Cref{exam:214}. It showed
in $\OO= \Z[2i]$ having conductor ideal $\QQ_2= 2\OO_K$, where $\OO_K= \Z[i]$, that:
\begin{equation}
8 = \Nm_{\OO}(( \QQ_2)^2) > \left(\Nm_{\OO} (\QQ_2)\right)^2 =4.
\end{equation}
\item[(2)] An example  of strict sub-multiplicativity of norms are given by  Marseglia \cite[Ex.\ 3.4]{Marseglia20b}.
The example takes $K=\Q(\alpha)$, where $\alpha$ is a root of a monic irreducible polynomial of degree $4$
with integer coefficients, such as $x^4-x-1$. Consider the order
\begin{equation}
\OO := \Z + p \Z[\alpha]= \Z + p\alpha\Z + p\alpha^2\Z + p\alpha^3\Z,
\end{equation}
where $p \ge 5$ is a rational prime number. Then the lattices
\begin{align}
\aa &:= p\Z + p \alpha \Z + p^2 \alpha^2\Z + p^2\alpha^3 \Z \mbox{ and} \\
\bb &:= p\Z + p^2\alpha\Z + p \alpha^2\Z + p^2\alpha^3\Z.
\end{align}
are $\OO$-ideals, and
\begin{equation}
\aa\bb = p^2 \Z + p^2\alpha\Z + p^2 \alpha^2\Z + p^2 \alpha^3\Z.
\end{equation}
We have 
\begin{equation}
p^5 = \Nm_{\OO}(\aa \bb) < \Nm_{\OO}(\aa) \Nm_{\OO}(\bb) = p^3 \cdot p^3= p^6.
\end{equation}
\end{itemize}
Marseglia gave an example of strict super-multiplicativity in this order $\OO$ as well. 
\end{eg}

\Cref{prop:normmult} will justify an extension of the norm to fractional ideals.

\begin{defn}\label{defn:fracnorm}
Let $\dd \in \rJ(\OO)$, and write $\dd = \colonideal{\aa}{\bb} = \aa\bb^{-1}$ for some $\aa \in \rI(\OO)$ and $\bb \in \rI^\ast(\OO)$.
Define $\Nm_\OO(\dd) = \frac{\Nm_\OO(\aa)}{\Nm_\OO(\bb)}$. (The next proposition shows this norm is well-defined.) 
\end{defn}

\begin{prop}\label{prop:normfracmult}
The norm of a fractional ideal $\dd$  in \Cref{defn:fracnorm} is well-defined. 
If  $\dd \subseteq \OO$, then its fractional ideal norm agrees with its integral ideal norm.
If $\cc \in \rJ^\ast(\OO)$ and $\dd \in \rJ(\OO)$, then $\Nm_\OO(\cc\dd) = \Nm_\OO(\cc)\Nm_\OO(\dd)$. 
\end{prop}
\begin{proof}
If $\dd \in \rJ(\OO)$ and $\dd = \aa_1\bb_1^{-1} = \aa_2\bb_2^{-1}$ for some $\aa_1, \aa_2 \in \rI(\OO)$ and $\bb_1, \bb_2 \in \rI^\ast(\OO)$, then $\aa_1\bb_2 = \aa_2\bb_1$, so $\Nm_\OO(\aa_1)\Nm_\OO(\bb_2) = \Nm_\OO(\aa_2)\Nm_\OO(\bb_1)$ by \Cref{prop:normmult}. Thus, $\frac{\Nm_\OO(\aa_1)}{\Nm_\OO(\bb_1)} = \frac{\Nm_\OO(\aa_2)}{\Nm_\OO(\bb_2)}$, so $\Nm(\dd)$ is well-defined.

In the case $\dd \subseteq \OO$, we may choose $\aa_1= \dd$ and $\bb_1=\OO$, so the integral ideal norm agrees with the fractional ideal norm.

Now, consider $\cc \in \rJ^\ast(\OO)$ and $\dd \in \rJ(\OO)$. Write $\cc = \aa_1\bb_1^{-1}$ and $\dd = \aa_2\bb_2^{-1}$ for $\aa_1,\bb_1,\bb_2 \in \rJ^\ast(\OO)$ and $\aa_2 \in \rJ(\OO)$. Then, $\cc\dd = (\aa_1\aa_2)(\bb_1\bb_2)^{-1}$, so
\begin{equation}
\Nm_\OO(\cc\dd) = \frac{\Nm_\OO(\aa_1\aa_2)}{\Nm_\OO(\bb_1\bb_2)} = \frac{\Nm_\OO(\aa_1)\Nm_\OO(\aa_2)}{\Nm_\OO(\bb_1)\Nm_\OO(\bb_2)} = \Nm_\OO(\cc)\Nm_\OO(\dd),
\end{equation}
using \Cref{prop:normmult} and \Cref{defn:fracnorm}.
\end{proof}

\begin{eg}[Behavior of norms of non-integral fractional ideals having an ideal power that is an integral ideal]
\label{exam:A7} 
Consider the non-maximal order $\OO=\Z[2i]$ of the Gaussian field $K=\Q(i)$, with maximal order $\OO_K= \Z[i]$. 
\begin{itemize}
\item[(1)] The (non-integral) fractional $\OO$-ideal 
\begin{equation}
\rr_1 := (1+i)\OO = 4\Z + (1+i)\Z,
\end{equation}
is a principal fractional $\OO$-ideal, hence it is an invertible fractional $\OO$-ideal. 
Here $\rr_1= \cc \dd^{-1}$ where $\cc= 2(1+i)\OO$ and $\dd= 2\OO$ are invertible integral $\OO$-ideals, from which we may compute $\Nm_{\OO}(\rr_1)=2$.
 
Recall from \Cref{exam:224} that $\rr_1$ has ideal square $\rr_1^2= 2i \OO = \qq_4'$, which was shown to be an irreducible integral $\OO$-ideal in \Cref{exam:214}, having $\Nm_{\OO}(\qq_4')=4$. The fractional $\OO$-ideal $\rr_1$ then has fractional $\OO$-ideal  norm $\Nm_{\OO}(\rr_1) =2$.
 
The equality $\Nm_{\OO}(\rr_1^2) = (\Nm_{\OO}(\rr_1))^2=4$ is consistent with \Cref{prop:normfracmult}, since $\rr_1$ is an invertible ideal.
 
\item[(2)] The (non-integral) fractional $\OO$-ideal
\begin{equation}
\rr_2:= (1+i)\OO_K = 2\Z + (1+i)\Z,
\end{equation}
is a non-invertible fractional $\OO$-ideal. 
Viewed as a fractional $\OO$-ideal, $\Nm_{\OO}(\rr_2) = 2$,
using $\rr_2 = \cc_2 (\dd_2)^{-1}$ with $\cc_2= 2(1+i)\OO_K$ and $\dd_2 = 2\OO$.
Next we have 
\begin{equation}
(\rr_2)^2 = 2 \OO_K = \QQ_2,
\end{equation}
and $\rr_2^2$  is  a non-invertible $\OO$-ideal  since it is an $\OO_K$-ideal.
We have  
$\Nm_{\OO}(\QQ_2) = 2$. 
Consequently, 
\begin{equation}
2= \Nm_{\OO}(\rr_2^2)\ne (\Nm_{\OO}(\rr_2))^2=4,
\end{equation}
showing that the hypothesis of  $\OO$-invertibility of at least  one of $\cc$ and $\dd$
cannot be relaxed  in \Cref{prop:normfracmult}.

\item[(3)] The (non-integral) fractional $\OO$-ideal
\begin{equation}
\rr_3:= i \OO= 2\Z + i\Z
\end{equation}
is an invertible fractional $\OO$-ideal and it has square $(\rr_3)^2= \OO$,
which is  an  invertible integral $\OO$-ideal.
In this case, $\Nm_{\OO}( \rr_3) = 1$.
\end{itemize}
\end{eg} 

The final proposition shows that norms of invertible fractional $\OO$-ideals are preserved under extension. 

\begin{prop}\label{prop:B5} 
Suppose $\OO \subseteq \OO'$ for $K$-orders, and let $\ext$ be the extension map from fractional ideals on $\OO$ to fractional ideals on $\OO'$. 
If $\cc \in \rJ^\ast(\OO)$, then $\Nm_{\OO'}(\ext(\cc)) = \Nm_{\OO}(\cc)$.
\end{prop}
\begin{proof}
By definition $\ext(\cc) = \cc\OO'$, and we wish to show
\begin{equation}
\Nm_{\OO'}(\cc \OO') = \Nm_{\OO}(\cc).
\end{equation}
Pick a positive integer $d \in \N$ so that $d (\OO' + \cc\OO') \subseteq \OO$; it follows that $d\OO' \subseteq \OO$ and $d\cc\OO' \subseteq \OO$. Let $n$ be the degree of the field extension $K/\Q$. Then $d\OO'$ is an integral $\OO'$-ideal with 
\begin{equation}
\Nm_{\OO'}(d\OO') = [\OO' : d\OO'] = d^n.
\end{equation}
Since $d\OO'$  is an invertible $\OO'$-ideal,    \Cref{prop:normfracmult} for $\OO'$-ideals gives
\begin{equation}
\Nm_{\OO'}(d\cc\OO') 
= \Nm_{\OO'}((d\OO')(\cc\OO'))
= \Nm_{\OO'}(d\OO')\Nm_{\OO'}(\cc\OO').
\end{equation}
Since $d\cc\OO'$ is an integral $\OO'$-ideal (because it is contained in $\OO$),
\begin{align}
\Nm_{\OO'}(d\cc\OO') &= [\OO':d\cc\OO'].
\end{align}
We have $d\cc\OO' \subseteq \OO \subseteq \OO'$,
so the following index relations on $\Z$-modules hold:
\begin{equation} 
[\OO':d\cc\OO'] = [\OO': \OO] [\OO: d\cc\OO'] = [\OO':\OO] \Nm_{\OO}(d\cc\OO').
\end{equation} 
Combining the four calculations thus far, in reverse order, we obtain
\begin{equation}\label{eq:normrel1}
[\OO':\OO] \Nm_{\OO}(d\cc\OO') = d^n \Nm_{\OO'}(\cc\OO').
\end{equation}
Since $\cc$ is an invertible $\OO$-ideal, \Cref{prop:normfracmult} 
for $\OO$-ideals gives
\begin{equation}
\Nm_{\OO}(d\cc\OO') 
= \Nm_{\OO}((d\OO')\cc) 
= \Nm_\OO(\cc)\Nm_\OO(d\OO').
\end{equation}
Since $d\OO' \subseteq \OO \subseteq \OO'$, a $\Z$-module index calculation yields
\begin{equation}
\Nm_{\OO}(d\OO') = [\OO : d\OO']= \frac{[\OO': d\OO']}{[\OO': \OO]} = \frac{d^n}{[\OO':\OO]}.
\end{equation}
Combining the calculations in the previous two lines gives
\begin{equation}\label{eq:normrel2}
\Nm_{\OO}(d\cc\OO') 
= \Nm_\OO(\cc) \, \frac{d^n}{[\OO':\OO]}
= \frac{d^n\Nm_\OO(\cc)}{[\OO':\OO]}.
\end{equation}
Substituting the right-hand side of \eqref{eq:normrel2} for $\Nm_{\OO}(d\cc\OO')$   into \eqref{eq:normrel1} yields
\begin{equation}
[\OO':\OO] \, \frac{d^n\Nm_\OO(\cc)}{[\OO':\OO]} = d^n \Nm_{\OO'}(\cc\OO'),
\end{equation}
which simplifies to 
$\Nm_\OO(\cc) = \Nm_{\OO'}(\cc\OO')$.
\end{proof}

\begin{eg}[Change of norm under extension for non-invertible ideals]\label{exam:A9} 
The norm of a non-invertible  $\OO$-ideal  may change under extension.
For any non-maximal order $\OO$ the  (absolute) conductor ideal $\ff(\OO) = \ff_{\OO_K}(\OO)$ 
is non-invertible integral $\OO$-ideal that is an $\OO_K$-ideal, whose norm will always
increase under extension to $\OO_K$. We have $\ext(\ff(\OO)) = \ff(\OO)$, and
$\Nm_{\OO}( \ff(\OO)) = [ \OO: \ff(\OO)]$,  while
\begin{equation}
\Nm_{\OO_K}(\ff(\OO))= [\OO_K: \ff(\OO)] = [\OO_K: \OO] [\OO: \ff(\OO)]=[\OO_K: \OO]\Nm_{\OO}(\ff(\OO)).
\end{equation}

The usefulness of \Cref{prop:B5} is that it applies to invertible fractional ideals $\cc$
that are not coprime to the relative conductor ideal $\ff_{\OO'}(\OO)$. 
For the non-maximal order $\OO=\Z[2i]$ of the Gaussian field $K=\Q(i)$, treated in \Cref{exam:A7}, the ideal
$\rr_1 := (1+i)\OO = 4\Z + (1+i)\Z$ is a principal fractional $\OO$-ideal 
so is $\OO$-invertible. It is not coprime to the conductor ideal
$\ff(\OO) = \ff_{\OO_K}(\OO)= 2\OO_K$. 
Noting that $\rr_1 \OO_K= (1+i)\OO_K$, it follows that $\Nm_{\OO_K}(\rr_1 \OO_K)= \Nm_{\OO}(\rr_1) =2$, consistent with \Cref{prop:B5}. 
\end{eg}


\end{document}